\documentclass[12pt,a4paper]{amsart}
\usepackage[utf8]{inputenc}
\usepackage[active]{srcltx}
\usepackage{amstext}
\usepackage{amsthm}

\makeatletter

\pdfpageheight\paperheight
\pdfpagewidth\paperwidth

\numberwithin{equation}{section}
\numberwithin{figure}{section}

\usepackage{amsmath}

\usepackage{amssymb}
\usepackage{stmaryrd}
\usepackage{amsthm}
\usepackage[mathscr]{eucal}
\usepackage{amscd}
\usepackage[all]{xy}
\usepackage{url}

\newtheorem{Thm}{Theorem}[subsection]

\usepackage{enumerate}

\usepackage{comment}

\usepackage{hyperref}
\hypersetup{%
pdftitle={Preprint},
pdfauthor={},
pdfkeywords={},
bookmarksnumbered,
pdfstartview={FitH},
breaklinks=true,
colorlinks=true,
}%
\usepackage[all]{hypcap} 

\usepackage{breakurl}

\usepackage{color}

\newtheorem{Lem}[Thm]{Lemma}
\newtheorem{Prop}[Thm]{Proposition}
\newtheorem{Cor}[Thm]{Corollary}
\newtheorem{Conj}[Thm]{Conjecture}

\newtheorem{Eg}[Thm]{Example}
\newtheorem{Rem}[Thm]{Remark}
\newtheorem{Def}[Thm]{Definition}

\newtheorem*{Def*}{Definition}
\newtheorem*{Thm*}{Theorem}
\newtheorem{Assumption}{Assumption}
\newtheorem*{Conj*}{Conjecture}

\newcommand{\kk}{\Bbbk}
\newcommand{\Z}{\mathbb{Z}}
\newcommand{\N}{\mathbb{N}}
\newcommand{\Q}{\mathbb{Q}}
\newcommand{\C}{\mathbb{C}}
\newcommand{\R}{\mathbb{R}}

\renewcommand{\hat}[1]{\widehat{#1}}
\renewcommand{\tilde}[1]{\widetilde{#1}}

\newcommand{\opname}[1]{\operatorname{\mathsf{#1}}}

\newcommand{\Spec}{\operatorname{\mathsf{Spec}}}

\newcommand{\pr}{\opname{pr}}

\newcommand{\Tr}{\opname{Tr}}

\newcommand{\End}{\opname{End}}

\newcommand{\Hom}{\opname{Hom}}

\newcommand{\supp}{\opname{supp}}

\renewcommand{\deg}{\opname{deg}}

\newcommand{\Hf}{{\frac{1}{2}}}

\newcommand{\Rm}[1]{{\longmapsto}}
\newcommand{\Lm}[1]{{\longmapsfrom}}

\newcommand{\cA}{{\mathcal A}}

\newcommand{\cD}{{\mathcal D}}

\newcommand{\cF}{{\mathcal F}}

\newcommand{\cH}{{\mathcal H}}

\newcommand{\cS}{{\mathcal S}}

\newcommand{\bA}{{\mathbf A}}

\newcommand{\bL}{{\mathbf L}}

\newcommand{\uc}{{\underline{c}}}

\newcommand{\uh}{{\underline{h}}}
\newcommand{\ui}{{\underline{i}}}

\newcommand{\uk}{{\underline{k}}}

\newcommand{\tB}{{\widetilde{B}}}

\newcommand{\tF}{{\widetilde{F}}}

\newcommand{\tQ}{{\widetilde{Q}}}

\newcommand{\can}{L}

\newcommand{\cor}{{\opname{cor}}}
\newcommand{\clAlg}{{\cA}}
\newcommand{\qClAlg}{\cA_q}

\newcommand{\diag}{{\delta}}

\newcommand{\tw}{{\tilde{w}}}

\usepackage{tikz}
\usetikzlibrary{positioning,shapes,shadows,arrows,snakes}

\pgfdeclarelayer{edgelayer}
\pgfdeclarelayer{nodelayer}
\pgfsetlayers{edgelayer,nodelayer,main}

\tikzstyle{none}=[inner sep=0pt]
\tikzstyle{black box}=[draw=black, fill=black!25]
\tikzstyle{white box}=[draw=black, fill=white]
\tikzstyle{black circle}=[circle,draw=black!50, fill=black!25]
\tikzstyle{red circle}=[circle,draw=red!50, fill=red!25]
\tikzstyle{blue circle}=[circle,draw=blue!50, fill=blue!25]
\tikzstyle{green circle}=[circle,draw=green!50, fill=green!25]
\tikzstyle{yellow circle}=[circle,draw=yellow!50, fill=yellow!25]

\newcommand{\thistheoremname}{}
\newtheorem*{genericthm*}{\thistheoremname}
\newenvironment{namedthm*}[1]
  {\renewcommand{\thistheoremname}{#1}%
   \begin{genericthm*}}
  {\end{genericthm*}}

\renewcommand{\diag}{{d}}

\newcommand{\sym}{\mathcal{S}}

\renewcommand{\can}{{\bL}}

\newcommand{\var}{\opname{var}}

\newcommand{\fv}{\opname{f}}
\newcommand{\ufv}{\opname{uf}}
\newcommand{\codeg}{\opname{codeg}}
\newcommand{\suppDim}{\opname{suppDim}}

\makeatother

\begin{document}
\newtheorem{DefLem}[Thm]{Definition-Lemma}

\renewcommand{\qClAlg}{\clAlg_q}

\newcommand{\bQClAlg}{\overline{\clAlg}_q}

\renewcommand{\Mc}{M^\circ}
\newcommand{\Nc}{N^\circ}
\newcommand{\Nufv}{N_{\opname{uf}}}

\newcommand{\sol}{\opname{TI}}
\newcommand{\intv}{\opname{BI}}

\newcommand{\Perm}{\opname{P}}

\newcommand{\bideg}{\opname{bideg}}

\newcommand{\midClAlg}{\clAlg^{\opname{mid}}}

\newcommand{\upClAlg}{\mathcal{U}}
\newcommand{\qUpClAlg}{\mathcal{U}_q}

\newcommand{\canClAlg}{\clAlg^{\opname{can}}}

\newcommand{\AVar}{\mathbb{A}}
\newcommand{\XVar}{\mathbb{X}}
\newcommand{\bAVar}{\overline{\mathbb{A}}}

\newcommand{\Jac}{\hat{\mathop{J}}}

\newcommand{\wt}{\opname{cl}}
\newcommand{\cl}{\opname{cl}}

\newcommand{\LP}{{\mathcal{LP}}}
\newcommand{\bLP}{{\overline{\mathcal{LP}}}}

\newcommand{\bClAlg}{{\overline{\clAlg}}}
\newcommand{\bUpClAlg}{{\overline{\upClAlg}}}

\newcommand{\cRing}{\mathcal{P}}

\newcommand{\tree}{{\mathbb{T}}}

\renewcommand{\diag}{d'}

\newcommand{\img}{{\opname{Im}}}

\newcommand{\Id}{{\opname{Id}}}
\newcommand{\prin}{{\opname{prin}}}

\newcommand{\mm}{{\mathbf{m}}}

\newcommand{\cPtSet}{{\mathcal{CPT}}}
\newcommand{\bPtSet}{{\mathcal{BPT}}}
\newcommand{\tCPtSet}{{\widetilde{\mathcal{CPT}}}}

\newcommand{\tf}{{\tilde{f}}}
\newcommand{\ty}{{\tilde{y}}}
\newcommand{\tcS}{{\tilde{\mathcal{S}}}}

\newcommand{\frd}{{\mathfrak{d}}}
\newcommand{\frD}{{\mathfrak{D}}}
\newcommand{\frp}{{\mathfrak{p}}}
\newcommand{\frg}{{\mathfrak{g}}}
\newcommand{\frn}{{\mathfrak{n}}}
\newcommand{\frsl}{{\mathfrak{sl}}}

\newcommand{\Quot}{\opname{Quot}}


\newcommand{\col}{\opname{col}}

\newcommand{\MM}{\mathfrak{M}}

\newcommand{\Inj}{\mathbf{I}}

\newcommand{\seq}{\overleftarrow{\mu}}
\newcommand{\seqnu}{\overleftarrow{\nu}}

\newcommand{\hLP}{\widehat{\mathcal{LP}}}
\renewcommand{\sym}{\opname{d}}

\newcommand{\envAlg}{\opname{U_q}}
\newcommand{\frh}{\mathfrak{h}}
\newcommand{\ow}{\overrightarrow{w}}

\renewcommand{\wt}{\opname{wt}}

\newcommand{\trans}{\opname{P}}
\newcommand{\FF}{\mathbb{F}}
\newcommand{\domCone}{\mathbf{D}}

\newtheorem{Definition-Lemma}[Thm]{Definition-Lemma}

\renewcommand{\tw}{\opname{tw}}
\newcommand{\qO}{{\opname{A_q}}}

\newcommand{\up}{{\opname{up}}}

\newcommand{\midAlg}{{\opname{A}}}

\newcommand{\circB}{{\mathring{B}}}

\renewcommand{\bA}{{\mathbb{A}}}

\newcommand{\fd}{{\mathrm{fd}}}
\newcommand{\tropMc}{\mathcal{M}^\circ}

\newcommand{\tDelta}{\widetilde{\Delta}}
\newcommand{\reviseStart}{}
\newcommand{\reviseEnd}{}

\title[]{Dual canonical bases and quantum cluster algebras}
\author{Fan QIN}
\dedicatory{Dedicated to Professor Bernard Leclerc on the occasion of his sixtieth
birthday}
\email{qin.fan.math@gmail.com}
\begin{abstract}
Given any quantum cluster algebra arising from a quantum unipotent
subgroup of symmetrizable Kac-Moody type, we verify the quantization
conjecture in full generality that the quantum cluster monomials are
contained in the dual canonical basis after rescaling.
\end{abstract}

\maketitle
\tableofcontents{}

\section{Introduction}

\label{sec:intro}

\subsection{Background}

Cluster algebras were introduced by Fomin and Zelevinsky around the
year 2000 \cite{FominZelevinsky02}. Their work was rooted in the
desire to understand, in a concrete and combinatorial way, the theory
of total positivity \cite{Lusztig96} and the dual canonical bases
in quantum groups $\envAlg^{-}(\mathfrak{g})$ \cite{Lusztig90,Lusztig91}\cite{Kashiwara93}.
Cluster algebras have distinguished generators called cluster variables.
As a main motivation, Fomin and Zelevinsky conjectured that the cluster
monomials (certain monomials of cluster variables) belong to the dual
canonical basis.

The theory of cluster algebras has enjoyed a rapid growth and it has
soon been linked to many (sometimes unexpected) other topics such
as combinatorics, representation theory, Lie theory, discrete integrable
systems, Poisson geometry, higher Teichm\"uller theory, algebraic
geometry. We refer the reader to the surveys \cite{Keller08Note,Keller09b}.
Yet the motivational conjecture by Fomin and Zelevinsky still remains
open. 

Using the notion of quantum cluster algebras \cite{BerensteinZelevinsky05},
a precise statement for the conjecture relating cluster theory and
the dual canonical bases for quantum groups was later formulated\footnote{The conjecture was called the \emph{quantization conjecture} by Kimura
after the works \cite{GeissLeclercSchroeer10,GeissLeclercSchroeer10b}
on (classical) unipotent subgroups.} \cite[Conjecture 1.1]{Kimura10}.

\begin{Conj}[Quantization conjecture]\label{conj:quantization_conjecture}
Given any symmetrizable Kac-Moody algebra $\mathfrak{g}$ and any
Weyl group element $w\in W$, up to $q$-power rescaling, the corresponding
quantum unipotent subgroup $\qO[N_{-}(w)]$ is a quantum cluster algebra
and its dual canonical basis contains the quantum cluster monomials.

\end{Conj}

Here, $\qO[N_{-}(w)]$ is a quantum analogue of the coordinate ring
of the unipotent subgroup $N_{-}(w)$. We postpone the precise definition
and simply recall that $\qO[N_{-}(w_{0})]=\envAlg^{-}(\mathfrak{g})$
for semisimple Lie algebras $\mathfrak{g}$ and the longest element
$w_{0}$.

The cluster structure part in Conjecture \ref{conj:quantization_conjecture}
is no longer a problem. For symmetric Kac-Moody cases, $\qO[N_{-}(w)]$
are known to be quantum cluster algebras by \cite{GeissLeclercSchroeer10}\cite{Kimura10}\cite{GeissLeclercSchroeer11}.
\cite{GY13} found the quantum cluster structure on $\qO[N_{-}(w)]$
for symmetrizable semisimple Lie algebras, and their arguments remain
effective in general. When the author is preparing this paper, an
explicit treatment for general cases has become available in \cite{goodearl2020integral}.

The dual canonical basis part in Conjecture \ref{conj:quantization_conjecture}
has been largely open for many years. First assume $\frg$ is a symmetric
Kac-Moody algebra. \cite{GeissLeclercSchroeer10b} proved an analogous
statement for the dual semi-canonical basis \cite{Lusztig00}. Partial
results include: quivers of Kronecker type and type $A$ by \cite{Lamp10,lampe2014quantum};
acyclic quivers by \cite{KimuraQin14} following \cite{HernandezLeclerc09}\cite{Nakajima09}.
Symmetric cases were solved only recently. Based on completely different
methods, \cite{qin2017triangular} proved it for $(\frg,w)$ such
that $\frg$ is a semisimple Lie algebra, $w$ arbitrary, or $\frg$
is a general Kac-Moody algebra but $w$ is \emph{adaptable}; \cite{Kang2018}
proved it for all $(\frg,w)$. But, up to now, Conjecture \ref{conj:quantization_conjecture}
has been almost untreated when $\frg$ is a symmetrizable Kac-Moody
algebra.

The main aim of this paper is to prove Conjecture \ref{conj:quantization_conjecture}
in full generality for any symmetrizable Kac-Moody algebra $\frg$
and any $w\in W$.

\subsection{Notions, main results and comments}

\subsubsection*{Quantum cluster algebras}

A seed $t$ of a cluster algebra is a collection of cluster variables
together with some combinatorial data. The set of all seeds of a cluster
algebra is denoted by $\Delta^{+}$, and one can generate one seed
from another (adjacent seed) via an algorithm called mutation. The
cluster algebra is generated by its cluster variables, and its quantization
can be constructed. See Section \ref{sec:Basics-cluster-algebra}
for rigorous definitions and treatments.

The readers unfamiliar with cluster theory might prefer a more intuitive
geometric picture. An upper cluster algebra is the coordinate ring
of a variety called the cluster variety, and it often agrees with
the cluster algebra. The cluster variables correspond to local coordinates
and the seeds to the charts of an atlas. The combinatorial data describe
the gluing of adjacent charts (mutation). When an appropriate (log-canonical)
$2$-form is chosen, the cluster variety becomes a Poisson manifold.
Correspondingly, its coordinate ring can be quantized to the quantum
cluster algebra.

Tropical geometry appears once we look at the cluster varieties, see
\cite{FockGoncharov06a}\cite{FockGoncharov09}\cite{gross2013birational},
and a breakthrough has recently been made in this direction \cite{gross2018canonical}.
Notice that the elements of cluster algebras are multivariate Laurent
polynomials. In this paper, by tropical properties, we mean properties
of the corresponding Laurent degrees, following \cite{qin2019bases}.

For most of this paper, we assume that the cluster algebra is injective-reachable
(equivalently, there exists a green to red sequence \cite{keller2011cluster}
or, the Donaldson-Thomas transformation is cluster). This condition
is satisfied by almost all well-known cluster algebras, such as those
arising from Lie theory or higher Teichm\"uller theory.

\subsubsection*{Triangular bases}

The approach to Conjecture \ref{conj:quantization_conjecture} in
the present paper is based on the triangular bases defined in \cite{qin2017triangular}.

For each given seed $t$, we have a quantum Laurent polynomial ring
$\LP(t)$ which contains the quantum cluster algebra. Let $\Mc(t)$
denote the lattice of Laurent degrees. Inspired by representation
theory, we can endow $\Mc(t)$ with a partial order $\prec_{t}$ called
the dominance order with respect to $t$.

Following \cite{qin2017triangular}, we define the triangular basis
for each single seed. It is characterized by a triangularity property
with respect to the dominance order $\prec_{t}$, whence the name
``triangular''. \cite{qin2017triangular} further introduced the
common triangular basis for all seeds such that it has the compatibility
property (certain tropical property, Definition \ref{def:compatibly_pointed})
expected by the Fock-Goncharov dual basis conjecture \cite{FockGoncharov06a}\cite[Section 5]{FockGoncharov09}\cite[Proposition 0.7]{gross2018canonical}.

It is worth remarking that the common triangular basis contains all
quantum cluster monomials if it exists. In general, we do not know
if the triangular basis for one seed exists or not (though we can
still construct a unique topological basis candidate in the formal
completion, see Theorem \ref{thm:construct_tri_func}). Even it exists,
it is unclear if it provides the common triangular basis for all seeds.
In a word, neither the local existence nor the global existence is
known. Nevertheless, by Lie theory, we have the following results
for a symmetric Kac-Moody algebra $\frg$.

\begin{enumerate}

\item The quantum unipotent subgroup $\qO[N_{-}(w)]$ is a quantum
cluster algebra such that, after rescaling and localization, the dual
canonical basis provides the common triangular basis \cite[Theorem 1.2.1(I), 6.1.6 ]{qin2017triangular}\cite{kashiwara2019laurent}.

\item A similar statement for cluster algebras categorified by representation
theory of the quantum loop algebra $U_{q}(L\frg)$ \cite{HernandezLeclerc09},
where the common triangular basis arises from simple modules \cite[Theorem 1.2.1(II)]{qin2017triangular}.

\end{enumerate}

\begin{Rem}[Positivity assumption]\label{rem:positivity_assumption}

Many previous arguments about the triangular basis \cite{qin2017triangular}
rely on the assumption that the basis has positive multiplication
structure constants. Such a positivity assumption is natural, because
the work \cite{qin2017triangular} was inspired by the monoidal categorification
of cluster algebras in the sense of \cite{HernandezLeclerc09}. In
particular, the basis is expected to consist of the simple modules. 

It is known that the dual canonical basis no longer satisfies the
positivity assumption in symmetrizable cases\footnote{Non-positivity was found in Shigenori Yamane's master thesis at Osaka
University. It can be verified by the algorithm in \cite{leclerc2004dual}.
See \cite{tsuchioka10answer} for more details.}. This is the main obstruction that renders many previous arguments
ineffective.

\end{Rem}

\subsubsection*{Main results and comments}

\begin{Thm}[Main theorem, {Theorem \ref{thm:canonical_triangular}}]\label{thm:main_intro}

After localization and rescaling, the dual canonical bases of quantum
unipotent subgroups become the common triangular bases of quantum
cluster algebras. As a consequence, Conjecture \ref{conj:quantization_conjecture}
holds true.

\end{Thm}

Effectively, for proving Theorem \ref{thm:main_intro}, we verify
Conjecture \ref{conj:quantization_conjecture} and Fock-Goncharov
dual basis conjecture (for this case) simultaneously in this common
triangular basis approach.

It might come as a surprise that, unlike previous works on Conjecture
\ref{conj:quantization_conjecture}, we do not need the positivity
assumption on the basis (which is false in general). This result is
derived as a consequence of a more general existence theorem for the
common triangular bases (Theorem \ref{thm:existence_intro}). For
proving the latter theorem, we base our arguments mostly on an analysis
of tropical properties and an additional input called twist automorphisms,
which we shall explain below.

After localization, the quantum unipotent subgroup $\qO[N_{-}(w)]$
becomes isomorphic to the quantum unipotent cell $\qO[N_{-}^{w}]$.
It possesses an automorphism $\eta_{w,q}$ called the twist automorphism,
which permutes the (localized) dual canonical basis \cite{kimura2017twist}. 

For general quantum cluster algebras, we define the cluster twist
automorphisms $\tw$ (twist automorphism for short, Section \ref{sec:Twist-automorphisms}),
such that they include $\eta_{w,q}$ as a special case after rescaling
(\eqref{eq:rescale_twist_automorphism}, Theorem \ref{thm:compare_twist_automorphism}).
We show that the common triangular basis, if it exists, must be permuted
by the cluster twist automorphism (Proposition \ref{prop:tri_basis_permute_by_twist}).
Conversely, assuming that a basis possesses this property, we prove
the following existence theorem for common triangular bases.

\begin{Thm}[Existence theorem, {Theorem \ref{thm:existence_mutation_sequence}}]\label{thm:existence_intro}

Assume that the quantum cluster algebra possesses the triangular basis
for an initial seed and a twist automorphism. If the triangular basis
is permuted by the twist automorphism and contains the quantum cluster
monomials appearing along a (green to red) mutation sequence, then
it is the common triangular basis. In particular, it contains all
quantum cluster monomials.

\end{Thm}

Briefly speaking, via an analysis of tropical properties, we prove
Theorem \ref{thm:existence_intro} based on two statements for adjacent
seeds which are inspired by desired conjectures. More precisely, we
consider two properties of a basis: 

\begin{itemize}

\item the admissibility (Definition \ref{def:admissibility}) means
that the basis contains certain cluster monomials in the spirit of
Conjecture \ref{conj:quantization_conjecture}; 

\item the compatibility (Definition \ref{def:compatibly_pointed})
means that it has nice tropical properties in the spirit of Fock-Goncharov
dual basis conjecture.

\end{itemize}Then we show that the triangular basis for one seed
gives rise to a compatible triangular basis for an adjacent seed provided
the admissibility (Proposition \ref{prop:admissible_nearby_seed});
conversely, certain compatibility condition implies the admissibility
(Proposition \ref{prop:compatible_to_admissible}), see Sections \ref{subsec:Adjacent-seeds},
\ref{subsec:Criteria-existence}. It is worth remarking that we use
the dominance order decomposition from \cite{qin2019bases} (Definition
\ref{def:dominance_decomposition} Proposition \ref{prop:invariance_dominance_decomposition})
in the proof of Proposition \ref{prop:admissible_nearby_seed}. In
addition, we use some results from \cite{qin2017triangular} such
as Proposition \ref{prop:lift_to_triangular}.

\begin{Rem}[Main ingredients]

To summarize, let us list the main ingredients in our approach to
Conjecture \ref{conj:quantization_conjecture} contained in this paper
and its prerequisite works. 

Lie theoretic side: \cite{Kimura10} showed that quantum unipotent
subgroups $\qO[N_{-}(w)]$ possess the dual canonical bases induced
from those of $\envAlg^{-}$, which enables us to investigate Fomin-Zelevinsky's
conjectural link between the dual canonical bases and cluster theory
for Kac-Moody cases and formulate Conjecture \ref{conj:quantization_conjecture}.
Based on the theory of non-commutative unique factorization domains,
\cite{GY13}\cite{goodearl2020integral} guaranteed the cluster structures
on $\qO[N_{-}(w)]$. \cite{kimura2017twist} showed that the twist
automorphisms on quantum unipotent cells $\qO[N_{-}^{w}]$ preserve
the (localized) dual canonical bases. Their proof was based on the
fascinating representation theoretic properties of these bases.

Cluster theoretic side: We need to know fundamental and important
notions and results in cluster algebras, such as (the sign-coherence
of) $g$-vectors and $c$-vectors following \cite{FominZelevinsky07}\cite{DerksenWeymanZelevinsky09}...\cite{Demonet10}\cite{gross2018canonical}\footnote{For our main theorem, it is enough to know fundamental properties
of cluster algebras arising from quantum unipotent subgroups of symmetrizable
types. Correspondingly, we can rely on \cite{Demonet10}, but it also
follows from the general results of \cite{gross2018canonical}.}. In this paper, we use notions and techniques introduced in \cite{qin2017triangular}
and further developed in \cite{qin2019bases} for analyzing tropical
properties of elements of cluster algebras. In addition, we consider
an analogue of the twist automorphism on $\qO[N_{-}^{w}]$ (the cluster
twist automorphism) and verify some of its properties (Theorem \ref{thm:global_twist_automorphism},
\ref{thm:compare_twist_automorphism}). Then we replace previous arguments
involving geometry or positivity in \cite{qin2017triangular} by an
analysis of tropical properties following \cite{qin2017triangular}\cite{qin2019bases},
improve the theory of triangular bases in \cite{qin2017triangular}
(Propositions \ref{prop:tri_basis_permute_by_twist}, \ref{prop:admissible_nearby_seed},
\ref{prop:compatible_to_admissible}), and obtain existence criteria
of common triangular bases (Theorems \ref{thm:existence_compatible},
\ref{thm:existence_mutation_sequence}) without imposing the positivity
assumption.

\end{Rem}

\begin{Rem}[Generalization of dual canonical bases]

Fomin and Zelevinsky's original motivation was ambitious, such that
the meaning of their ``dual canonical basis'' includes not only
the well-known dual canonical basis for quantum groups but also (unknown)
generalizations for other (quantized) coordinate rings in Lie theory.
It is therefore important to understand all good bases for other (quantum)
cluster algebras. Many works have been devoted to this topic. These
include but are not limited to \cite{dupont2011generic} \cite{GeissLeclercSchroeer10b}
\cite{plamondon2013generic} \cite{qin2019bases} \cite{musiker2013bases}
\cite{felikson2017bases} \cite{HernandezLeclerc09} \cite{qin2017triangular}
\cite{Kang2018} \cite{thurston2014positive} \cite{BerensteinZelevinsky2012}
\cite{gross2018canonical} \cite{lee2014greedy} \cite{lee2014greedyPNAS}.

By this paper, the triangular basis in the sense of \cite{qin2017triangular}
suggests a reasonable generalization of dual canonical bases in cluster
theory.

\end{Rem}

\begin{Rem}[Twist automorphisms]

In the literature, twist automorphisms are automorphisms on the nilpotent
cells $N_{-}^{w}$ introduced for solving the factorization problems
(chamber ansatz) which describe the inverse of the toric chart of
Schubert varieties \cite{BerensteinFominZelevinsky96}\cite{berenstein1997total}.
In the symmetric Kac-Moody cases, \cite{GeissLeclercSchroeer10}\cite{GeissLeclercSchroeer10b}
studied them using categorification via preprojective algebras. A
quantum analogue was introduced and studied in \cite{kimura2017twist}.

The cluster twist automorphisms in this paper are defined for seeds
of general quantum cluster algebras, not necessarily with a Lie theoretic
background. They include the twist automorphisms in Lie theory as
a special case (Theorem \ref{thm:compare_twist_automorphism}). Therefore,
they provide a reasonable generalization in cluster theory. It is
desirable to understand them on a categorical level (Remark \ref{rem:twist_automorphism_meaning}).

In this paper, we will mainly use the cluster twist automorphisms
of Donaldson-Thomas type, but the notion makes sense more generally.
Further discussions about cluster twist automorphisms will appear
in a separate paper.

\end{Rem}

\begin{Rem}[Berenstein-Zelevinsky triangular basis]

The idea to construct a common triangular basis first appeared in
the work of Berenstein and Zelevinsky \cite{BerensteinZelevinsky2012}.
They defined triangular bases for acyclic seeds of quantum cluster
algebras, and their definition is different from ours. Based on \cite{qin2019compare}
and Theorem \ref{thm:canonical_triangular}, we now know that the
Berenstein-Zelevinsky triangular basis agrees with our common triangular
bases (Corollary \ref{cor:compare_triangular_basis} ).

\end{Rem}

We also obtain a result concerning tropical properties of the dual
canonical bases (Corollary \ref{cor:tropical_dual_can_basis}). It
would be desirable to further understand the relation between the
dual canonical bases and the tropical geometry of cluster varieties. 

It is worth noting that, after the appearance of this paper, McNamara
\cite{mcnamara2021cluster} gave a different proof for Conjecture
\ref{conj:quantization_conjecture} based on the work of \cite{Kang2018}
and folding techniques.

\subsection{Contents}

We provide detailed prerequisites with examples. An expert might skip
Sections \ref{sec:Basics-cluster-algebra}, \ref{sec:Prerequisites-dominance-order},
\ref{sec:Prerequisite-quantum-groups} and probably Sections \ref{sec:Similar-seeds-and-correction},
\ref{sec:Quantum-cluster-structures}. Our arguments and proofs are
presented in Sections \ref{sec:Twist-automorphisms}, \ref{sec:Triangular-bases},
\ref{sec:Dual-canonical-bases-results}.

In Section \ref{sec:Basics-cluster-algebra}, we review basic notions
in cluster theory.

Section \ref{sec:Prerequisites-dominance-order} contains prerequisites
on the dominance order and the corresponding decomposition following
\cite{qin2017triangular} \cite{qin2019bases}. Section \ref{sec:Similar-seeds-and-correction}
contains prerequisites on similar seeds, a correction technique \cite{Qin12}
\cite{qin2017triangular}, and a result concerning mutations of similar
elements (Proposition \ref{prop:mutation_similar_elements}).

Sections \ref{sec:Twist-automorphisms} and \ref{sec:Triangular-bases}
contain most crucial arguments in this paper. First, we define cluster
twist automorphisms and check some necessary properties in Section
\ref{sec:Twist-automorphisms}. Then, we give a general construction
for triangular functions as candidates of basis elements (Theorem
\ref{thm:construct_tri_func}) and review notions and properties for
the desired triangular bases. In addition, we show that common triangular
bases must be permuted by any twist automorphisms (Proposition \ref{prop:tri_basis_permute_by_twist}).
Finally, by using an analysis on tropical properties (such as the
compatibility) and the twist automorphisms, we give general existence
criteria for common triangular bases (Theorems \ref{thm:existence_compatible},
\ref{thm:existence_mutation_sequence}).

Section \ref{sec:Prerequisite-quantum-groups} contains prerequisites
on quantum unipotent subgroups.

In Section \ref{sec:Quantum-cluster-structures}, we present the quantum
cluster structure on quantum unipotent subgroups, following \cite{GY13}.

In Section \ref{sec:Dual-canonical-bases-results}, we apply the results
of the previous sections to quantum cluster algebras arising from
quantum unipotent subgroups and quantum unipotent cells. We first
discuss the integral form and localization in Sections \ref{subsec:Integral-form},
\ref{subsec:Localization-and-bases}. Then, we compare the twist automorphisms
in Lie theory and in our sense in Section \ref{subsec:Twist-automorphism-compare}.
Finally, we apply the previous discussion to obtain our main result
(Theorem \ref{thm:canonical_triangular}) and deduce some other consequences
in Section \ref{subsec:Consequences}.

In Appendix \ref{sec:Cluster-structures-on}, we give some technical
details for quantum groups, CGL extensions, and cluster structures
for the convenience of the reader.

\section*{Acknowledgments}

The author thanks Yoshiyuki Kimura for many helpful discussions. He
is grateful to Bernhard Keller for seeing a preliminary version of
the paper. He also thanks Milen Yakimov for informing him about the
recent work \cite{goodearl2020integral}. He thanks Peigen Cao for
interesting discussions. He is grateful to Bernhard Keller and Bernard
Leclerc for remarks on a preliminary version of this article.

\section{Basics of cluster algebras\label{sec:Basics-cluster-algebra}}

We recall basic notions such as seeds and mutations in cluster theory.
We mostly follow the convention in \cite{gross2013birational}, because
it is conceptually natural to introduce various lattices. A reader
might equally choose the convention in \cite{FominZelevinsky07}.
More detailed discussions could be found in \cite{qin2019bases}.

\subsection{Seeds}

\subsubsection*{Fixed data}

Given $I$ a finite set of vertices endowed with a partition $I=I_{\ufv}\sqcup I_{\fv}$
(called the unfrozen and frozen vertices respectively), we further
fix the following data:

\begin{enumerate}

\item a rank-$|I|$ lattice $N$, endowed with a $\Q$-valued skew-symmetric
bilinear form $\omega(\ ,\ )$;

\item a rank-$|I_{\ufv}|$ saturated sublattice $N_{\ufv}\subset N$,
called unfrozen sublattice;

\item strictly positive integers $(d_{i})_{i\in I}$ with the greatest
common divisor $1$;

\item a sublattice $\Nc\subset N$ of finite index, such that $\omega(N_{\ufv},\Nc)\subset\Z$,
$\omega(N,N_{\ufv}\cap\Nc)\subset\Z$;

\item the dual lattices $M=\Hom_{\Z}(N,\Z)$, $\Mc=\Hom_{\Z}(\Nc,\Z)$.

\end{enumerate}

Notice that $\omega(n,\ )$, $n\in N$, is naturally a $\Q$-valued
linear function on $N$, thus an element in $M_{\Q}=M\otimes_{\Z}\Q$.
Correspondingly, we define the map $p^{*}:N\rightarrow M_{\Q}$ such
that $p^{*}(n)=\omega(n,\ )$. The assumption on $\Nc$ implies that
$p^{*}(\Nufv)\subset\Mc$. 

We always make the following assumption throughout this paper.

\begin{Assumption}[Injectivity assumption]

We assume that the map $p^{*}:\Nufv\rightarrow\Mc$ is injective.

\end{Assumption}

We shall soon see explicit examples for the fixed data. In fact, we
will mainly work with the lattices $\Nufv$ and $\Mc$ throughout
this paper. 

\begin{Rem} 

If all symmetrizers $d_{i}$ are $1$, one can understand the lattices
as Grothendieck groups of some categories, in which one can also consider
bases and bilinear forms, see \cite[Sections 2.3, 2.4]{Qin10} for
details.

\end{Rem}

\subsubsection*{Seeds as bases}

\begin{Def}[Seed]

A seed for the fixed data is an $I$-labeled collection $E=(e_{i}|i\in I)$
such that $E$ is a basis of $N$, $\{e_{k}|k\in I_{\ufv}\}$ is a
$\Z$-basis of $N_{\ufv}$, and $\{d_{i}e_{i}|i\in I\}$ is a basis
of $\Nc$.

\end{Def}

Given a seed $E$, we shall also denote it by $t$ as in most works
in cluster theory. Let $E^{*}=\{e_{i}^{*}|i\in I\}$ denote the dual
basis of $E$ in $M$. Denote $F^{*}:=\{f_{i}^{*}|i\in I\}:=\{d_{i}e_{i}|i\in I\}$.
Then $F:=\{f_{i}|i\in I\}:=\{\frac{1}{d_{i}}e_{i}^{*}|i\in I\}$ is
a basis of $\Mc$. By using these bases, the abstract lattices in
the fixed data can be represented as sets of (column) vectors:

\begin{align*}
\Mc(t) & :=\oplus_{i\in I}\Z f_{i}\simeq\Z^{I}\\
N(t) & :=\oplus_{i\in I}\Z e_{i}\simeq\Z^{I}\\
\Nufv(t) & :=\oplus_{k\in I_{\ufv}}\Z e_{k}\simeq\Z^{I_{\ufv}},
\end{align*}
where we identify $f_{i}$, $e_{i}$ as $i$-th unit vectors respectively.

Let us define $\omega_{ij}=\omega(e_{i},e_{j})$, $b_{ji}=\omega_{ij}d_{j}$,
$\forall i,j\in I$. Apparently, we have $b_{ji}d_{i}=-b_{ij}d_{j}$,
$\forall i,j\in I$. Direct computation shows that $p^{*}(e_{j})(f_{i}^{*})=b_{ij}$,
$\forall i,j\in I$. Therefore, the linear map $p^{*}:N\rightarrow M_{\Q}$
is represented by the matrix $(b_{ij})_{i,j\in I}$, and $p^{*}(e_{j})$
is represented by its $j$-th column vector $\col_{j}((b_{ij})_{i,j\in I})$.

\begin{Def}[$\tB$-matrix]

The $\tB$-matrix of the seed $t$ is defined to be $\tB(t)=(b_{ik})_{i\in I,k\in I_{\ufv}}=(\omega_{ki}d_{i})_{i\in I,k\in I_{\ufv}}$.
The corresponding $B$-matrix is defined to be the $I_{\ufv}\times I_{\ufv}$-submatrix
$B(t)$, which is called the principal part of $\tB(t)$.

\end{Def}

We often omit the symbol $t$ when the context is clear. Under the
injectivity assumption, $\tB$ is of full rank $|I_{\ufv}|$, and
the column vectors $\col_{k}\tB=p^{*}e_{k}$, $k\in I_{\ufv}$, are
linearly independent.

\begin{Rem}[Different conventions: seeds]\label{rem:different_convention}

In \cite{FominZelevinsky02,FominZelevinsky07}, a seed $t$ (with
geometric coefficients) is defined to be the collection of the matrix
$\tB(t)$ and the cluster variables $X_{i}(t)$, $i\in I$. See \cite[Section 2.1]{qin2019bases}
for changing conventions.

\end{Rem}

\begin{Eg}[Type $A_2$]\label{eg:A2}

Let us take $I=I_{\ufv}=\{1,2\}$, $d_{1}=d_{2}=1$, $N=\Z^{2}$,
an initial seed $t_{0}=E=\{e_{1},e_{2}\}=\{\left(\begin{array}{c}
1\\
0
\end{array}\right),\left(\begin{array}{c}
0\\
1
\end{array}\right)\}$, and the bilinear form $\omega$ satisfies $(\omega_{ij})=\left(\begin{array}{cc}
0 & 1\\
-1 & 0
\end{array}\right)$. Then we have $\Nc=N$, $\Mc=M\simeq\Z^{2}$ such that $F=\{f_{1},f_{2}\}=\{\left(\begin{array}{c}
1\\
0
\end{array}\right),\left(\begin{array}{c}
0\\
1
\end{array}\right)\}$, $(b_{ij})=\tB=B=\left(\begin{array}{cc}
0 & -1\\
1 & 0
\end{array}\right)$, $p^{*}e_{1}=f_{2}$, $p^{*}e_{2}=-f_{1}$.

\end{Eg}

\subsubsection*{Poisson structure}

By \cite{GekhtmanShapiroVainshtein03,GekhtmanShapiroVainshtein05},
since $\tB$ is of full rank, we can find an $I\times I$ $\Z$-valued
skew-symmetric matrix $\Lambda=(\Lambda_{ij})$ such that $(\tB,\Lambda)$
is a compatible pair in the sense of \cite{BerensteinZelevinsky05},
i.e. $\tB^{T}\Lambda=\left(\begin{array}{cc}
D' & 0\end{array}\right)$ for some diagonal matrix $D'=\mathrm{diag}(\diag_{k})_{k\in I_{\ufv}}$,
$\diag_{k}\in\N_{>0}$. Notice that $\Lambda$ is not necessarily
unique. In addition, $\tB^{T}\Lambda\tB=D'B$ is skew-symmetric, i.e.,
we have $\diag_{i}b_{ij}=-\diag_{j}b_{ji}$ $\forall i,j\in I_{\ufv}$.

The matrix $\Lambda$ gives rise to a bilinear form $\lambda$ on
$\Mc$ such that $\lambda(f_{i},f_{j})=\Lambda_{ij}$. One can use
it to construct Poisson structure on the corresponding algebraic torus,
see \cite{GekhtmanShapiroVainshtein03,GekhtmanShapiroVainshtein05}
for example.

\begin{Lem} \label{lem:compatible_pair}The following claims are
true for any $i\in I$, $j,k\in I_{\ufv}$.

(1) $\lambda(f_{i},p^{*}e_{k})=-\delta_{ik}\diag_{k}$.

(2) $\lambda(\ ,p^{*}e_{k})=-\diag_{k}d_{k}\cdot e_{k}=-\diag_{k}f_{k}^{*}$.

(3) $\lambda(p^{*}e_{i},p^{*}e_{k})=-\diag_{k}b_{ki}.$

\end{Lem}

\begin{proof}

(1) The claim follows form the compatibility $\Lambda(-\tB)=\left(\begin{array}{c}
D'\\
0
\end{array}\right)$.

(2) The claim follows from (1).

(3) We have $\lambda(p^{*}e_{i},p^{*}e_{k})=\langle p^{*}e_{i},-\diag_{k}f_{k}^{*}\rangle=b_{ki}\cdot(-\diag_{k}).$

\end{proof}

\begin{Def}[Quantum seed]

The collection $(E,\lambda)$ is called a quantum seed.

\end{Def}

The skew-symmetric matrix $D'B=\tB^{T}\Lambda\tB$ give rise to a
skew-symmetric bilinear form on $\Nufv$, still denoted by $\lambda$,
such that 
\begin{eqnarray*}
\lambda(e_{i},e_{j}) & :=\lambda(p^{*}e_{i},p^{*}e_{j})= & \diag_{i}b_{ij}.
\end{eqnarray*}
The following property follows from definition.

\begin{Lem}\label{lem:compatible_poisson-M}

For any $n,n'\in N_{\ufv}$, we have

\begin{eqnarray*}
\lambda(p^{*}n',p^{*}n) & = & \lambda(n',n).
\end{eqnarray*}

\end{Lem}

\begin{Eg}

In Example \ref{eg:A2}, we have $d=1$ and $d_{1}^{\vee}=d_{2}^{\vee}=1$.
Take any $\alpha\in\N_{>0}$ and the matrix $\Lambda=(\Lambda_{ij})=\alpha\left(\begin{array}{cc}
0 & -1\\
1 & 0
\end{array}\right)$. Then $(\Lambda,\tB)$ is compatible with $\diag_{1}=\diag_{2}=\alpha$.
Notice that $\Lambda$ gives a skew-bilinear form $\lambda$ on $\Mc$.
Then we have
\begin{eqnarray*}
(\lambda(e_{i},e_{j}))_{i,j\in I_{\ufv}} & = & (\diag_{i}b_{ij})_{i,j}=\alpha\left(\begin{array}{cc}
0 & -1\\
1 & 0
\end{array}\right)=-\alpha\omega.
\end{eqnarray*}

\end{Eg}

For completeness, let us discuss the relation between $\omega$ and
$\lambda$, though it will not be needed for the purpose of this paper.

Denote the least common multiplier $d=\mathrm{lcm}(d_{i})$. Define
the Langlands dual $d_{i}^{\vee}=\frac{d}{d_{i}}$. Then they satisfy
$d_{i}^{\vee}b_{ij}=-d_{j}^{\vee}b_{ji}$, $\forall i,j\in I$. Denote
the scaling constant $\text{\ensuremath{\alpha_{k}=\frac{\diag_{k}}{d_{k}^{\vee}}\in\Q_{>0}}}$
for $k\in I_{\ufv}$, then we have $\alpha_{k}=\alpha_{j}$ whenever
$b_{kj}\neq0$, $\forall k,j\in I_{\ufv}$. It follows that $\alpha_{k}$
is constant on each connected component of $I_{\ufv}$ in the following
sense.

\begin{Def}[Connected components]

Given an $I_{\ufv}\times I_{\ufv}$ matrix $B$, a connected component
$V$ of $I_{\ufv}$ with respect to $B$ is a maximal collection of
vertices in $I_{\ufv}$ such that for any $i,j\in V$, there exists
finitely many vertices $i_{s}\in V$, $0\leq s\leq l$, such that
$i_{0}=i$, $i_{l}=j$, and $b_{i_{s}i_{s+1}}\neq0$ for any $0\leq s\leq l-1$.

\end{Def}

We can choose the matrix $\Lambda$ such that there exists $\alpha\in\Q_{>0}$
such that $\diag_{k}=\alpha d_{k}^{\vee}$ for all $k\in I_{\ufv}$,
i.e., $\alpha_{k}=\alpha$, see \cite[Theorem 2.1]{GekhtmanShapiroVainshtein05}. 

\begin{Lem} \label{lem:compatible_poisson_MN}The following statements
hold for any $n\in\Nufv$.

(1) $\lambda(\ ,p^{*}n)=-\alpha d\cdot n$

(2) $\lambda=-\alpha d\cdot\omega$ on $\Nufv$.

\end{Lem}

\begin{proof}

(1) The claims follow from Lemma \ref{lem:compatible_pair}.

(2) By Lemma \ref{lem:compatible_pair}, we have $\lambda(e_{i},e_{j})=\diag_{i}b_{ij}=\alpha d_{i}^{\vee}\cdot d_{i}\omega_{ji}=\alpha d\cdot(-\omega_{ij})$.

\end{proof}

\subsection{Laurent polynomial rings}

Let $\kk$ denote a base ring and $v\in\kk$ an invertible element.
We will take $(\kk,v)=(\Z,1)$ and $(\kk,v)=(\Z[q^{\pm\Hf}],q^{\Hf})$
for classical cluster algebra and quantum cluster algebras respectively,
where $q^{\Hf}$ is a formal parameter. Later, we will also take $(\kk,v)=(\Q(q),q)$
for quantized enveloping algebras, where $q$ is a formal parameter.
Notice that $\kk$ has the $\Z$-linear (or $\Q$-linear) bar involution
$\overline{(\ )}$ which sends $v$ to $v^{-1}$ in these cases.

\subsubsection*{Ring of characters}

Let there be given any lattice $L$ of finite rank. Denote its dual
lattice $L^{\vee}=\Hom_{\Z}(L,\Z)$. Define the following group algebra
endowed with the natural addition and multiplication $(+,\cdot)$
\begin{eqnarray*}
\kk[L] & = & \oplus_{p\in L}\kk\chi^{p}
\end{eqnarray*}
such that $\chi^{p}\cdot\chi^{p'}=\chi^{p+p'}$. $\kk[L]$ is called
the Laurent polynomial ring associated to $L$, and $L$ the lattice
of Laurent degrees. We often omit the commutative multiplication symbol
$\cdot$ for simplicity.

It is worth reminding that $\chi^{p}$ can be viewed as a character
on the split algebraic torus $T_{L^{\vee}}=L^{\vee}\otimes_{\Z}\FF^{*}$,
where $\FF$ is a chosen field and $\FF^{*}=\FF\backslash\{0\}$ its
multiplicative group of invertible elements. We will assume $\FF=\C$
throughout this paper.

\subsubsection*{Quantization}

Assume that $L$ possesses a $\Z$-valued skew-symmetric bilinear
form $\lambda$. We then define the quantum algebra $\kk[L;\lambda]$
such that $\kk[L;\lambda]$ is given by the commutative algebra $(\kk[L],+,\cdot)$
as before and, in addition, endowed with the twisted product $*$:

\begin{eqnarray}
\chi^{p}*\chi^{p'} & = & v^{\lambda(p,p')}\chi^{p+p'}.\label{eq:twist_product}
\end{eqnarray}
Unless otherwise specified, by the algebra structure of $\kk[L;\lambda]$,
we mean $(+,*)$. When the context is clear, we omit the symbol $\lambda$
for simplicity. Notice that, when $v=1$, the twisted product $*$
agrees with the usual product $\cdot$.

\subsubsection*{Formal Laurent series}

We have the Euclidean space $L_{\R}=L\otimes_{\Z}\R$. Let there be
given a strictly convex rational polyhedral cone $\sigma\subset L_{\R}$,
i.e., $\sigma=\sum_{j}\R_{\geq0}p_{j}$ for finitely many generators
$p_{j}\in L$. It determines the monoid $\sigma_{L}=\sigma\cap L$
and the abelian group $\Z\sigma_{L}=\sigma_{L}\otimes_{\N}\Z$. Since
$\sigma$ is strictly convex, the only invertible element in $\sigma_{L}$
is $0$. Denote $\sigma_{L}^{+}=\sigma_{L}\backslash\{0\}$.

We have the monoid algebra $\kk[\sigma_{L}]=\oplus_{p\in\sigma_{L}}\kk\chi^{p}$
and its maximal ideal $\MM=\kk[\sigma_{L}^{+}]=\oplus_{p\in\sigma_{L}^{+}}\kk\chi^{p}$.
Let $\widehat{\kk[\sigma_{L}]}$ denote the completion of $\kk[\sigma_{L}]$
with respect to this maximal ideal. We then define the formal completion
of $\kk[L]$ with respect to $\sigma_{L}$ to be

\begin{eqnarray*}
\widehat{\kk[L;\sigma_{L}]} & = & \kk[L]\otimes_{\kk[\sigma_{L}]}\widehat{\kk[\sigma_{L}]}.
\end{eqnarray*}
When the context is clear, we write $\widehat{\kk[L;\sigma_{L}]}=\widehat{\kk[L]}$
for simplicity. Similarly, we define the formal completion of the
quantum algebra $\kk[L;\lambda]$ to be

\begin{eqnarray*}
\widehat{\kk[L;\sigma_{L},\lambda]} & = & \kk[L;\lambda]\otimes_{\kk[\sigma_{L};\lambda]}\widehat{\kk[\sigma_{L};\lambda]}.
\end{eqnarray*}
$\widehat{\kk[L;\sigma_{L}]}$ (resp. $\widehat{\kk[L;\sigma_{L},\lambda]}$)
is called the algebra (resp. quantum algebra) of formal Laurent series.

A possibly infinite sum will be called a formal sum. Let there be
given a collection of formal sums $Z^{(j)}=\sum_{m_{i}\in L}b_{m_{i}}^{(j)}\chi^{m_{i}}$,
$b_{m_{i}}^{(j)}\in\kk$, $j\in\N$, their formal sum $\sum_{j}Z^{(j)}$
is well defined if, for each $m_{i}$, there are only finitely many
$Z^{(j)}$ with non-vanishing coefficients $b_{m_{i}}^{(j)}$.

\begin{Rem}

The formal completion $\widehat{\kk[L]}$ can be shown to be the Cauchy
completion with respect to an explicit topology on $\kk[L]$, see
\cite{davison2019strong}. It would be interesting to understand calculation
in $\widehat{\kk[L]}$ from a topological point of view, though we
will not pursue this direction in this paper.

\end{Rem}

\subsubsection*{Algebras associated to seeds}

Let there be given a quantum seed $t=(E,\lambda)$ as before. Recall
that we have lattices

\begin{eqnarray*}
\Mc(t) & = & \oplus_{i\in I}\Z f_{i}\simeq\Z^{I}\\
N(t) & = & \oplus_{i\in I}\Z e_{i}\simeq\Z^{I}\\
\Nufv(t) & = & \oplus_{k\in I_{\ufv}}\Z e_{k}\simeq\Z^{I_{\ufv}},
\end{eqnarray*}
where $f_{i}$, $e_{i}$ are viewed as the $i$-th unit vectors.

Consider the quantum algebras of Laurent polynomials $\kk[\Mc(t);\lambda]$
and $\kk[\Nufv(t);\lambda]$. Denote $X_{i}=X_{i}(t)=\chi^{f_{i}}$
and $X^{m}=X(t)^{m}=\chi^{m}$ for $m\in\Mc(t)$. The quantum Laurent
polynomial ring associated to the quantum seed $t$ is defined as
\begin{align*}
\LP(t) & :=\kk[X_{i}^{\pm}]_{i\in I}=\kk[X^{m}]_{m\in M}=\kk[\Mc(t);\lambda].
\end{align*}
 Its skew-field of fractions is denoted by $\cF(t)$.

We also define the Laurent monomials $Y_{k}=Y_{k}(t)=\chi^{p^{*}e_{k}},Y^{n}=Y(t)^{n}=\chi^{p^{*}n}$,
$n\in\Nufv(t)$, and the corresponding subalgebra of $\LP(t)$:
\begin{eqnarray*}
\kk[Y_{k}^{\pm}]_{k\in I_{\ufv}} & = & \kk[Y^{n}]_{n\in\Nufv(t)}\simeq\kk[\Nufv(t);\lambda].
\end{eqnarray*}

Notice that $\kk[Y_{k}]_{k\in I_{\ufv}}$ has the maximal ideal generated
by $Y_{k}$, $k\in I_{\ufv}$. Let $\widehat{\kk[Y_{k}]_{k\in I_{\ufv}}}$
denote the corresponding completion under the adic topology. As before,
the formal completion of $\LP(t)$ is defined by 
\begin{eqnarray*}
\hLP(t) & = & \LP(t)\otimes_{\kk[Y_{k}]_{k\in I_{\ufv}}}\widehat{\kk[Y_{k}]_{k\in I_{\ufv}}}.
\end{eqnarray*}
Its elements will be called formal Laurent series or functions.

\begin{Def}[Quantum cluster variables]

The quantum cluster variables of the seed $t$ are defined to be $X_{i}(t)$,
$i\in I$. Their monomials are called quantum cluster monomials. The
$X_{j}(t)$, $j\in I_{\fv}$ are called the frozen variables. The
group of frozen factors is defined to be $\cRing=\{X(t)^{u}|u\in\Z^{I_{\fv}}\}$. 

The Laurent monomials of the form $X(t)^{m}$, $m=\sum m_{i}f_{i}$,
$m_{i}\in\Z$, such that $m_{i}\geq0$ whenever $i\in I_{\ufv}$ ,
are called localized quantum cluster monomials.

The $Y_{k}(t)$, $k\in I_{\ufv}$, are called the (unfrozen) quantum
$Y$-variables of the seed $t$.

\end{Def}

For simplicity, we often skip the word quantum and the symbol $t$.

\subsection{Mutation}

\subsubsection*{Seeds}

Let $[\ ]_{+}$ denote the function $\max(\ ,0)$. Define $[(g_{i})]_{+}=([g_{i}]_{+})$
for any vectors $(g_{i})$. Recall that a seed is a collection of
basis elements $t=E=(e_{i})_{i\in I}$ in $N$. Let us define new
seeds in the following procedure.

We start by choosing a sign $\varepsilon\in\{+,-\}$. For any $k\in I_{\ufv}$,
define $I\times I$ transformation matrices $\trans_{E,\varepsilon}$
and $\trans_{F,\varepsilon}$ whose entries are given as follows:

\begin{eqnarray*}
(\trans_{E,\varepsilon})_{ij} & = & \begin{cases}
\delta_{ij} & i\neq k\\
-1 & i=j=k\\{}
[\varepsilon b_{kj}]_{+} & i=k,j\neq k
\end{cases}
\end{eqnarray*}

\begin{eqnarray*}
(\trans_{F,\varepsilon})_{ij} & = & \begin{cases}
\delta_{ij} & j\neq k\\
-1 & i=j=k\\{}
[-\varepsilon b_{ik}]_{+} & i\ne k,j=k
\end{cases}.
\end{eqnarray*}
Denote their $I_{\ufv}\times I_{\ufv}$-submatrices by $\trans_{E,\varepsilon}^{I_{\ufv}}$
and $\trans_{F,\varepsilon}^{I_{\ufv}}$ respectively\footnote{The matrices $\trans_{E,\varepsilon}$ and $\trans_{F,\varepsilon}^{I_{\ufv}}$
are denoted by $F_{\varepsilon}$ and $E_{\varepsilon}$ in \cite[Section 5.6]{Keller12}\cite[(3.2) (3.3)]{BerensteinZelevinsky05}
respectively.}.

Given any unfrozen vertex $k\in I_{\ufv}$, we define the new seed
$t'=E'=(e_{i}')_{i\in I}$ such that $e_{i}'=\sum_{j\in I}e_{j}\cdot(\trans_{E,\varepsilon})_{ji}\in N$:

\begin{eqnarray*}
e_{i}' & = & \begin{cases}
e_{i}+[\varepsilon b_{ki}]_{+}e_{k} & i\neq k\\
-e_{k} & i=k
\end{cases}.
\end{eqnarray*}
Recall that $b_{ij}=\omega_{ji}d_{i}=\omega(e_{j},e_{i})d_{i}$ and
it depends on the seed $t=E$. 

\begin{Def}[Mutation]

We denote the above operation by $t'=\mu_{k,\varepsilon}t$ and call
$\mu_{k,\varepsilon}$ the mutation in the direction $k$.

\end{Def}

We could also compute the dual basis $(E^{*})'=\{e_{i}^{*}|i\in I\}$
in $M$ and the corresponding basis $F'=\{\frac{1}{d_{i}}e_{i}^{*}|i\in I\}$
in $\Mc$. It is straightforward to check that we get $f_{i}'=\sum_{j\in I}f_{j}\cdot(\trans_{F,\varepsilon})_{ji}\in\Mc$:

\begin{eqnarray*}
f_{i}' & = & \begin{cases}
f_{i} & i\neq k\\
-f_{k}+\sum_{j}[-\varepsilon b_{jk}]_{+}f_{j} & i=k
\end{cases}.
\end{eqnarray*}

The new bases enable us to represent the abstract lattices as sets
of new coordinate (column) vectors:

\begin{eqnarray*}
\Mc(t') & = & \oplus\Z f_{i}'\simeq\Z^{I}\\
N(t') & = & \oplus\Z e_{i}'\simeq\Z^{I}\\
\Nufv(t') & = & \oplus_{k\in I_{\ufv}}\Z e_{k}'\simeq\Z^{I_{\ufv}}
\end{eqnarray*}
where $f_{i}'$, $e_{i}'$ are viewed as $i$th unit vectors.

The above basis change gives a linear transformation $\tau_{k,\varepsilon}:N(t')\rightarrow N(t)$
such that $\tau_{k,\varepsilon}(e_{i}')=\sum_{j\in I}e_{j}\cdot(\trans_{E,\varepsilon})_{ji}$
and, similarly, $\tau_{k,\varepsilon}:\Mc(t')\rightarrow\Mc(t)$ such
that $\tau(f_{i}')=\sum_{j\in I}f_{j}\cdot(\trans_{F,\varepsilon})_{ji}$.
These linear maps are represented by (left multiplication of) the
matrices $\trans_{E,\varepsilon}$ and $\trans_{F,\varepsilon}$ respectively.

\begin{Rem}

On the one hand, these matrices are idempotent, i.e., $(\trans_{E,\varepsilon})^{2}=(\trans_{F,\varepsilon})^{2}=\Id_{I}$.
On the other hand, it is straightforward to check that $\mu_{k,-\varepsilon}(\mu_{k,\varepsilon}(t))=t$.
Notice that the matrix depends on the seeds, and $\mu_{k,\varepsilon}\mu_{k,\varepsilon}(t)\neq t$!

\end{Rem}

It is straightforward to compute $b_{ij}'=\omega_{ji}'d_{i}=\omega(e_{j}',e_{i}')d_{i}$:

\begin{eqnarray*}
b_{ij}' & = & \begin{cases}
-b_{ij} & k\in\{i,j\}\\
b_{ij}+b_{ik}[\varepsilon b_{kj}]_{+}+[-\varepsilon b_{ik}]_{+}b_{kj} & k\neq i,j
\end{cases}.
\end{eqnarray*}
In addition, one can check that the new matrix $(b_{ij}')_{i,j\in I}$
only depends on $(b_{ij})_{i,j\in I}$ and is independent of the sign
$\varepsilon$. Consequently, we define the mutation of matrices $(b_{ij}')_{i,j\in I}=\mu_{k}((b_{ij})_{i,j\in I})$
without choosing a sign.

\begin{Lem}\label{lem:same_conn_components}

The connected components of $I_{\ufv}$ with respect to $B(t)$ are
the same as those with respect to $B(t')$.

\end{Lem}

Assume that $t$ is a quantum seed $(E,\lambda)$ with a chosen $I\times I$-matrix
$\Lambda=(\Lambda_{ij})_{i,j\in I}$ such that $\Lambda_{ij}=\lambda(f_{i},f_{j})$,
$\tB^{T}\Lambda=\left(\begin{array}{cc}
D' & 0\end{array}\right)$, $D'=\mathrm{diag}(\diag_{i})_{i\in I_{\ufv}}$. Let $\Lambda'=(\Lambda'_{ij})_{i,j\in I}$
denote the matrix such that $\Lambda_{ij}'=\lambda(f'_{i},f'_{j})$.
The following lemma shows that $\mu_{k}t:=t':=(E',\lambda)$ is also
a quantum seed.

\begin{Lem}[{\cite{BerensteinZelevinsky05}}]\label{lem:mutate_compatible_pair}

Take any sign $\varepsilon$, the following claims are true.

(1) $(b_{ij}')_{i,j\in I}=\trans_{F,\varepsilon}(b_{ij})_{i,j\in I}\trans_{E,\varepsilon}$

(2) $\Lambda'=\trans_{F,\varepsilon}^{T}\Lambda\trans_{F,\varepsilon}$

(3) $(\tB')^{T}\Lambda'=\left(\begin{array}{cc}
D' & 0\end{array}\right)$

\end{Lem}

\begin{proof}

(1)(2) The statements can be proved by calculation or using linear
algebras.

(3) The statement follows from (1)(2) and the equality $\trans_{F,\varepsilon}^{2}=\Id_{I}$,
$(\trans_{E,\varepsilon}^{I_{\ufv}})^{T}D'=D'(\trans_{F,\varepsilon}^{I_{\ufv}})$.

\end{proof}

\begin{Eg}\label{eg:A2_mutation}

In Example \ref{eg:A2}, let us take the initial seed $t_{0}=E=\{e_{1},e_{2}\}=\{\left(\begin{array}{c}
1\\
0
\end{array}\right),\left(\begin{array}{c}
0\\
1
\end{array}\right)\}$, $(b_{ij})=\left(\begin{array}{cc}
0 & -1\\
1 & 0
\end{array}\right)$. Choose $k=2$. Then $\trans_{E,+}=\left(\begin{array}{cc}
1 & 0\\
1 & -1
\end{array}\right)$, $t_{2}=\mu_{2,+}t_{0}=E'=\{e_{1}',e_{2}'\}=\{e_{1}+e_{2},-e_{2}\}$.
It follows that $(b'_{ij})=\left(\begin{array}{cc}
0 & 1\\
-1 & 0
\end{array}\right)$. Choosing $k=2$ again, we get $\trans_{E',-}=\left(\begin{array}{cc}
1 & 0\\
1 & -1
\end{array}\right)$, $\mu_{2,-}(t_{2})=E''=\{e_{1}'+e_{2}',-e_{2}'\}=E=t$. On the other
hand, $\trans_{E',+}=\left(\begin{array}{cc}
1 & 0\\
0 & -1
\end{array}\right)$ and, correspondingly, $\mu_{2,+}(t_{2})\neq t_{0}$!

\end{Eg}

\subsubsection*{Signs of seed mutations and $c$-matrices}

Let $t_{0}$ denote a given initial seed. Let $\uk=k_{1}k_{2}\cdots k_{l}$
denote a word whose letters are unfrozen vertices. We are going to
apply mutations along the sequence $\uk$. In cluster theory, there
is a default choice of the signs $\varepsilon_{j}$ for the corresponding
mutations once the initial seed $t_{0}$ is given.

Let $I'=\{i'|i\in I_{\ufv}\}$ denote a copy of $I_{\ufv}$. The principal
coefficient matrix associated to $t_{0}$ is the $(I_{\ufv}\sqcup I')\times(I_{\ufv}\sqcup I')$-matrix
given by $(b_{ij}^{\prin})_{i,j\in I_{\ufv}\sqcup I'}=\left(\begin{array}{cc}
(b_{ij})_{i,j\in I_{\ufv}} & -C(t_{0})\\
C(t_{0}) & 0
\end{array}\right)$, where $C(t_{0})$ is the $I'\times I_{\ufv}$-matrix whose non-zero
entries are $C(t_{0})_{i',i}=1$, $\forall i\in I_{\ufv}$. Under
the natural identification $I'\simeq I_{\ufv}$, we also view $C(t_{0})$
as the identity matrix $\Id_{I_{\ufv}}$.

For any $s\in[0,l]$, denote the word $\uk_{\leq s}=k_{1}k_{2}\ldots k_{s}$.
Assume that the sign $\varepsilon_{j}$ have been chosen for $j\leq s$,
then we have the mutation sequence $\seq_{\leq s}=\mu_{k_{s},\varepsilon_{s}}\cdots\mu_{k_{1},\varepsilon_{1}}$
(reading from right to left) and the corresponding seed $t_{s}=\seq_{\leq s}t_{0}$.
If $s=0$, $\uk_{\leq s}$ denote an empty word, $\seq_{\leq s}$
an empty sequence, and $t_{s}=t_{0}$. Applying $\seq_{\leq s}$ to
$(b_{ij}^{\prin})$, we obtain the matrix $(b_{ij}^{\prin}(t_{s}))_{i,j\in I_{\ufv}\sqcup I'}=\left(\begin{array}{cc}
(b_{ij}(t_{s}))_{i,j\in I_{\ufv}} & -C(t_{s})\\
C(t_{s}) & 0
\end{array}\right)$.

\begin{Def}\label{def:C-matrix}

The matrix $C^{t_{0}}(t_{s}):=C(t_{s})$ is called the $C$-matrix
of $t_{s}$ with respect to the initial seed $t_{0}$. Its $k$-th
column vectors are called the $k$-th $c$-vectors of $t_{s}$, denoted
by $c_{k}(t_{s})$.

\end{Def}

Notice that the $C$-matrices depend on the $B(t_{0})$.

Recall that a vector $g=(g_{i})_{i}$ is called \emph{sign-coherent},
if all components $g_{i}\geq0$ or all components $g_{i}\leq0$, in
which case we denote $g\geq0$ or $g\leq0$ respectively.

\begin{Thm}\cite[Theorem  1.7]{DerksenWeymanZelevinsky09}\cite[Corollary 5.5]{gross2018canonical}

For any $k\in I_{\ufv}$, the $c$-vector $c_{k}(t_{s})$ is non-zero
and sign-coherent.

\end{Thm}

We refer the reader to \cite{nakanishi2012tropical} for more discussions.

Following \cite[section 5.6]{Keller12}, by using $c$-vectors, we
choose the sign $\varepsilon_{k_{s+1}}=+$ if $c_{k_{s+1}}(t_{s})>0$,
; choose $\varepsilon_{k_{s+1}}=-$ if $c_{k_{s+1}}(t_{s})<0$. Repeating
this process, we obtain the signs for all mutations appearing in $\seq=\mu_{k_{l},\varepsilon_{l}}\cdots\mu_{k_{1},\varepsilon_{1}}$.

\begin{Def}

Let there be given an initial seed $t_{0}$. We use $\Delta_{t_{0}}^{+}$
to denote the seeds obtain from $t_{0}$ by applying any mutations
sequences, where we choose the signs by using $C$-vectors.

\end{Def}

Throughout this paper, we will discuss seeds taken from the same set
$\Delta_{t_{0}}^{+}$ for some given initial seed $t_{0}$ except
in Section \ref{sec:Similar-seeds-and-correction}. In particular,
the signs of seed mutations are determined by $t_{0}$. We will denote
$\Delta^{+}=\Delta_{t_{0}}^{+}$ and $\mu_{k}=\mu_{k,\varepsilon}$
for simplicity.

\subsubsection*{Mutation birational maps}

Now, let us relate the quantum algebras of Laurent polynomials associated
to $t$ and $t'=\mu_{k}t$. 

We define the isomorphism $\mu_{k}^{*}:\cF(t')\simeq\cF(t)$ such
that 
\begin{eqnarray}
\mu_{k}^{*}(X_{i}') & = & \begin{cases}
X_{i} & i\neq k\\
X^{-f_{k}+\sum_{j}[-b_{jk}]_{+}f_{j}}+X^{-f_{k}+\sum_{j}[b_{jk}]_{+}f_{j}} & i=k
\end{cases}.\label{eq:X_mutation}
\end{eqnarray}

Denote $v_{k}=v^{\diag_{k}}$. By Lemma \ref{lem:compatible_pair},
for $j\in I$, $i,k\in I_{\ufv}$, we have $\diag_{i}b_{ik}=-\diag_{k}b_{ki}$
and 
\begin{eqnarray*}
X_{j}*Y_{k} & = & v_{k}^{-\delta_{jk}}X_{j}\cdot Y_{k},\\
Y_{i}*Y_{k} & = & v^{\diag_{i}b_{ik}}Y_{i}\cdot Y_{k},
\end{eqnarray*}

We can rewrite
\begin{eqnarray*}
\mu_{k}^{*}(X_{k}) & = & X^{-f_{k}+\sum_{j}[-b_{jk}]_{+}f_{j}}\cdot(1+Y_{k})\\
 & = & X^{-f_{k}+\sum_{j}[-b_{jk}]_{+}f_{j}}*(1+v_{k}^{-1}Y_{k})
\end{eqnarray*}

The following equation is called the exchange relation:
\begin{align}
X_{k}*\mu_{k}^{*}(X_{k}) & =v^{\lambda(f_{k},\sum_{j}[-b_{jk}]_{+}f_{j})}X^{\sum_{j}[-b_{jk}]_{+}f_{j}}+v^{\lambda(f_{k},\sum_{i}[b_{ik}]_{+}f_{i})}X^{\sum_{i}[b_{ik}]_{+}f_{i}}\label{eq:exchange_relation}\\
 & =v^{\lambda(f_{k},\sum_{j}[-b_{jk}]_{+}f_{j})}X^{\sum_{j}[-b_{jk}]_{+}f_{j}}\cdot(1+v_{k}^{-1}Y_{k})\nonumber \\
 & =v^{\lambda(f_{k},\sum_{j}[-b_{jk}]_{+}f_{j})}X^{\sum_{j}[-b_{jk}]_{+}f_{j}}*(1+v_{k}^{-1}Y_{k}).\nonumber 
\end{align}

The mutation rule for the $Y$-variables $Y_{i}=X(t)^{p^{*}e_{i}}$,
$Y_{i}'=X(t')^{p^{*}e_{i}'}$, $i\in I_{\ufv}$, can be computed:

\begin{eqnarray}
\mu_{k}^{*}(Y_{k}') & = & Y_{k}^{-1},\label{eq:Y_mutation}\\
\mu_{k}^{*}(Y_{i}') & = & Y_{i}\cdot\sum_{s=0}^{-b_{ki}}\left(\begin{array}{c}
-b_{ki}\\
s
\end{array}\right)_{v_{k}}Y_{k}^{s},\ \mathrm{if\ }b_{ki}\leq0,\nonumber \\
\mu_{k}^{*}((Y_{i}')^{-1}) & = & Y_{i}^{-1}\cdot Y_{k}^{-b_{ki}}\cdot\sum_{s=0}^{b_{ki}}\left(\begin{array}{c}
b_{ki}\\
s
\end{array}\right)_{v_{k}}Y_{k}^{s},\ \mathrm{if\ }b_{ki}\geq0,\nonumber 
\end{eqnarray}
where the quantum numbers are defined as $[a]_{v}=\frac{v^{a}-v^{-a}}{v-v^{-1}}$,
$[a]_{v}!=[a]_{v}[a-1]_{v}\cdots[1]_{v}$, $\left(\begin{array}{c}
a\\
b
\end{array}\right)_{v}=\frac{[a]_{v}!}{[a-b]_{v}!\cdot[b]_{v}!}$, for $0\leq b\leq a\in\N$.

\begin{Eg}

Take some $c\in\N_{>0}$, $I=I_{\ufv}=\{1,2\}$, $(\omega_{ij})_{i,j\in I}=\left(\begin{array}{cc}
0 & -1\\
1 & 0
\end{array}\right)$, $(d_{1},d_{2})=(c,1)$. Define $B(t_{0})=B=\left(\begin{array}{cc}
0 & c\\
-1 & 0
\end{array}\right)$ such that $b_{ij}=\omega_{ji}d_{i}$. Then we have $Y_{1}=X_{2}^{-1}$,
$Y_{2}=X_{1}^{c}$. Take $\Lambda(t_{0})=\Lambda=\alpha\left(\begin{array}{cc}
0 & 1\\
-1 & 0
\end{array}\right)$ for $\alpha\in\N_{>0}$, then $D'=B^{T}\Lambda=\alpha\left(\begin{array}{cc}
1 & 0\\
0 & c
\end{array}\right)$, $v_{1}=v^{\alpha}$, $v_{2}=v^{\alpha c}$. Denote $t_{1}=\mu_{1}t_{0}$
and $t_{-1}=\mu_{2}t_{0}$. Direct computation shows that 
\begin{eqnarray*}
\mu_{1}^{*}(X_{1}(t_{1})) & = & X^{-f_{1}+f_{2}}*(1+v_{1}^{-1}Y_{1})=(1+v_{1}Y_{1})*X^{-f_{1}+f_{2}}\\
\mu_{2}^{*}(X_{2}(t_{-1})) & = & (X_{2}^{-1})*(1+v_{2}^{-1}Y_{2})=(1+v_{2}Y_{2})*X_{2}^{-1}
\end{eqnarray*}

For $Y$-variables, we have $Y_{1}*Y_{2}=v^{\alpha c}Y_{1}\cdot Y_{2}$,
$v_{1}^{-b_{12}}Y_{1}*Y_{2}=v_{2}^{-b_{21}}Y_{2}*Y_{1}$, and $X_{k}^{-1}*Y_{k}=v_{k}^{2}Y_{k}*X_{k}^{-1}$.
The mutation rule reads as

\begin{eqnarray*}
\mu_{2}^{*}(Y_{1}(t_{-1})) & = & \mu_{2}^{*}(X_{2}(t_{-1}))\\
 & = & Y_{1}*(1+v_{2}^{-1}Y_{2})\\
 & = & Y_{1}\cdot(1+Y_{2})
\end{eqnarray*}

\begin{eqnarray*}
\mu_{1}^{*}(Y_{2}(t_{1})^{-1}) & = & \mu_{1}^{*}(X_{1}(t_{1})^{c})\\
 & = & ((1+v_{1}Y_{1})*X^{-f_{1}+f_{2}})^{c}\\
 & = & (1+v_{1}^{1}Y_{1})*(1+v_{1}Y_{1}v_{1}^{2})*\cdots*(1+v_{1}Y_{1}v_{1}^{2(c-1)})*X^{-cf_{1}+cf_{2}}\\
 & = & (1+v_{1}Y_{1})*(1+v_{1}^{3}Y_{1})*\cdots*(1+v_{1}^{2c-1}Y_{1})*X^{-cf_{1}+cf_{2}}
\end{eqnarray*}

We can further compute 
\begin{align*}
\mu_{1}^{*}(Y_{2}(t_{1})^{-1}) & =(1+v_{1}^{1-c}Y_{1})\cdot(1+v_{1}^{3-c}Y_{1})\cdot\cdots\cdot(1+v_{1}^{c-1}Y_{1})\cdot X^{-cf_{1}+cf_{2}}\\
 & =Y_{2}^{-1}\cdot Y_{1}^{-c}\cdot\sum_{s=0}^{c}\left(\begin{array}{c}
c\\
s
\end{array}\right)_{v_{1}}Y_{1}^{s}
\end{align*}

\end{Eg}

\begin{Rem}[Birational map for tori]

Notice that \eqref{eq:Y_mutation} also gives the isomorphism (mutation
rule) between the skew-fields of fractions $\mu_{k}^{*}:\cF(\kk[N(t')])\simeq\cF(\kk[N(t)])$
in which we replace $Y_{i}$, $Y_{i}'$ by $\chi^{e_{i}}$, $\chi^{e_{i}'}$
respectively and take $i\in I$, see \cite[Section 3]{FockGoncharov09}\cite[Section 2.4.1]{davison2019strong}
for more details.

We call $\mu_{k}^{*}$ the mutation birational maps for the following
reason. Consider the split algebraic tori $\Spec\FF[\Mc(t)]=T_{\Nc(t)}=\Nc(t)\otimes_{\Z}\FF^{*}$
and $\Spec\FF[N(t)]=T_{M(t)}$ over a characteristic $0$ field $\FF$.
In the classical case ($v=1$), the maps $\mu_{k}^{*}$ are the pullback
of the birational map between the tori: 
\begin{eqnarray*}
\mu_{k} & : & T_{\Nc(t)}\dasharrow T_{\Nc(t')}\\
\mu_{k} & : & T_{M(t)}\dasharrow T_{M(t')}.
\end{eqnarray*}

There is no sign $\varepsilon$ in the definition of the mutation
birational map $\mu_{k}^{*}$. This is in contrast to the mutation
of bases $\tau_{k,\varepsilon}$. See \cite{gekhtman2017hamiltonian}
for their relation.

\end{Rem}

\begin{Rem}[Different conventions: mutations]\label{rem:diff_mutation_convention}

Recall that we have two conventions for the seeds following \cite{gross2013birational}
and \cite{FominZelevinsky02} respectively, see Remark \ref{rem:different_convention}.
The seeds in the formal sense are mutated as sets of vectors, while
those in the latter sense contain cluster variables which are mutated
by birational maps. Correspondingly, we have two sets of seeds $\Delta^{+}$
and $\tDelta{}^{+}$ constructed via iterated mutations respectively.
We have the natural but non-trivial result that the two sets are in
bijection, because the cluster variables are determined by their extended
$g$-vectors by \cite[Conjecture 7.10(1)]{FominZelevinsky07}.

The verification of \cite[Conjecture 7.10(1)]{FominZelevinsky07}
was due to \cite{DerksenWeymanZelevinsky09} for quiver cases. \cite{Demonet10}
generalized \cite{DerksenWeymanZelevinsky09} to some valued quiver
cases (such as those concerned in our main theorem), and the general
cases were solved by \cite{gross2018canonical}. Alternatively, given
the sign-coherence of $g$-vectors, we know that the cluster monomials
are always compatibly pointed by \cite[Remark 7.13]{FominZelevinsky07}.
Then \cite[Conjecture 7.10(1)]{FominZelevinsky07} can be deduced
from \cite[Lem 3.4.11]{qin2019bases} for injective-reachable cluster
algebras.

In this paper, we choose the convention of \cite{gross2013birational}
because it is conceptually more natural to introduce the lattices
$\Mc(t),N_{\ufv}(t)$ and related structures. A reader might equally
choose the convention in \cite{FominZelevinsky02}.

\end{Rem}

\subsubsection*{Mutation sequences}

For any seeds $t_{1},t_{2}\in\Delta^{+}$. Let $\seq_{t_{2},t_{1}}$
denote any chosen sequence of mutations such that $t_{2}=\seq_{t_{2},t_{1}}t_{1}$.
Then $(\seq_{t_{2},t_{1}})^{-1}$ is a mutation sequence that sends
$t_{2}$ to $t_{1}$, which we can choose as $\seq_{t_{1},t_{2}}$.
The following property immediately follows from the definition of
seeds in the sense of \cite{FominZelevinsky02}. See Remark \ref{rem:diff_mutation_convention}
for changing the convention.

\begin{Lem}\label{lem:composition_mutation_maps}

For any seeds $t_{1},t_{2},t_{3}\in\Delta^{+}$, we have $\seq_{t_{2},t_{3}}^{*}\seq_{t_{1},t_{2}}^{*}=\seq_{t_{1},t_{3}}^{*}$
as isomorphisms from $\cF(t_{1})$ to $\cF(t_{3})$. Moreover, $\seq_{t_{1},t_{1}}^{*}$
is always the identity.

\end{Lem}

\subsubsection*{Tropical transformations}

Let there be given any $t,t'\in\Delta^{+}$. If $t'=\mu_{k}t$ for
some $k\in I_{\ufv}$, the tropical transformation $\phi_{t',t}:\Mc(t)\simeq\Mc(t')$
is defined as the piecewise linear map such that, for any $g=(g_{i})\in\Mc(t)\simeq\Z^{I}$,
we have $\phi_{t',t}(g)=g'=(g'_{i})\in\Mc(t')\simeq\Z^{I}$ such that 

\begin{eqnarray*}
g'_{i} & = & \begin{cases}
-g_{k} & i=k\\
g_{i}+[b_{ik}]_{+}g_{k} & i\neq k,\ g_{k}\geq0\\
g_{i}+[-b_{ik}]_{+}g_{k} & i\neq k,\ g_{k}<0
\end{cases}
\end{eqnarray*}

If $t'=\seq t$ by a mutation sequence $\seq=\mu_{k_{r}}\cdots\mu_{k_{2}}\mu_{k_{1}}$,
then the tropical transformation $\phi_{t',t}:\Mc(t)\simeq\Mc(t')$
is defined as the composition of the above defined tropical transformations
for adjacent seeds. It is known that $\phi_{t',t}$ does not depend
on the choice of $\seq$, because it can be realized as the tropicalization
of the mutation birational map for certain cluster varieties associated
to $t$ and $t'$, see \cite{gross2013birational}. 

\begin{Def}[Tropical points {\cite{qin2019bases}}]\label{def:tropical_points}

Let $\tropMc$ denote the set of equivalent classes $\sqcup_{t\in\Delta^{+}}\Mc(t)$
under the identification $\phi_{t',t}$. It is called the set of tropical
points.

\end{Def}

\begin{Rem}

The set $\tropMc$ corresponds to the tropical points of certain cluster
variety, see \cite{gross2013birational,gross2018canonical}. By choosing
a seed $t$, we have chosen a specific representative $\Mc(t)$ of
the set of tropical points $\tropMc$. 

\end{Rem}

\subsection{Cluster algebras}

Let $\pr_{I_{\ufv}}$ denote the natural projection from $\Z^{I}$
to $\Z^{I_{\ufv}}$.

Let there be given an initial seed $t_{0}$ and the corresponding
set of seeds $\Delta^{+}=\Delta_{t_{0}}^{+}$. For any seeds $t\in\Delta^{+}$,
choose any mutation sequence $\seq_{t,t_{0}}$ such that $t=\seq_{t,t_{0}}t_{0}$.
Recall that we have the corresponding isomorphism $\seq_{t,t_{0}}^{*}:\cF(t)\simeq\cF(t_{0})$
which does not depend on the choice of $\seq_{t,t_{0}}$. By the well
known Laurent phenomenon \cite{FominZelevinsky02}\cite{BerensteinZelevinsky05},
the quantum cluster variable $X_{i}(t)$, $i\in I$, is sent to its
Laurent expansion $\seq_{t,t_{0}}^{*}X_{i}(t)$ in $\LP(t_{0})\subset\cF(t_{0})$.
Moreover, we have the following result.

\begin{Thm}[{\cite{FominZelevinsky07}\cite{Tran09}\cite{DerksenWeymanZelevinsky09}\cite{gross2018canonical}}]\label{thm:cluster_expansion}

The Laurent expansion of $X_{i}(t)$ in $\LP(t_{0})$ takes the following
form

\begin{eqnarray*}
\seq_{t,t_{0}}^{*}X_{i}(t) & = & X(t_{0})^{g_{i}^{t_{0}}(t)}\cdot\sum_{n\in\N^{I_{\ufv}}}c_{n}Y(t_{0})^{n}
\end{eqnarray*}
where the coefficients $c_{n}\in\kk$ such that $c_{0}=1$ and the
vector $g_{i}^{t_{0}}(t)\in\Z^{I_{\ufv}}\oplus\N^{I_{\fv}}$. Moreover,
$\pr_{I_{\ufv}}g_{i}^{t_{0}}(t)$ and $c_{n}$ are determined by $B(t_{0})$,
$(\diag_{k})_{k\in I_{\ufv}}$ and $\seq_{t,t_{0}}$.

\end{Thm}

The vector $g_{i}^{t_{0}}(t)$ is called the $i$-th (extended) $g$-vector
of the seed $t$ with respect to the initial seed $t_{0}$. Notice
that $c_{0}=1$ is a consequence of the sign-coherence of $c$-vectors,
see \cite[Proposition 5.6]{FominZelevinsky07}. In particular, it
implies that the cluster variable $\seq_{t,t_{0}}^{*}X_{i}(t)$ is
pointed at $g_{i}^{t_{0}}(t)$ in the sense of Definition \ref{def:pointed}.

Recall that we have frozen variables $\seq_{t,t_{0}}^{*}X_{j}(t)=X_{j}(t_{0})$
for $j\in I_{\fv}$, which we denote by $X_{j}$.

\begin{Def}[Quantum cluster algebras]

The (partially compactified) quantum cluster algebra $\bQClAlg(t_{0})$
with initial seed $t_{0}$ is defined to be the $\Z[q^{\pm\Hf}]$-subalgebra
of $\LP(t_{0})$ generated by the quantum cluster variables $\seq_{t,t_{0}}^{*}X_{i}(t)$,
$i\in I$, $t\in\Delta_{t_{0}}^{+}$.

The (localized) quantum cluster algebra $\qClAlg(t_{0})$ with initial
seed $t_{0}$ is the localization of $\bQClAlg(t_{0})$ at the frozen
variables $X_{j}$, $j\in I_{\fv}$.

The quantum upper cluster algebra $\qUpClAlg(t_{0})$ with initial
seed $t_{0}$ is defined to be the intersection $\cap_{t\in\Delta_{t_{0}}^{+}}\seq_{t,t_{0}}^{*}\LP(t)$.

\end{Def}

It follows from Theorem \ref{thm:cluster_expansion} that we have
$\bQClAlg(t_{0})\subset\qClAlg(t_{0})\subset\qUpClAlg(t_{0})\subset\LP(t_{0})$.

Notice that the quantum cluster algebra $\qClAlg(t_{0})$ only depends
on $\tB(t_{0})$ and $\Lambda(t_{0})$.

By choosing a different initial seed $t$, we obtain the same set
of seeds $\Delta_{t}^{+}=\Delta_{t_{0}}^{+}$, but another quantum
cluster algebra $\qClAlg(t)$ contained in $\LP(t)$. The two algebras
are isomorphic such that $\qClAlg(t_{0})=\seq_{t,t_{0}}^{*}\qClAlg(t)$.

As in literature, we sometimes identify $\cF(t)$ and $\cF(t_{0})$
via $\seq_{t,t_{0}}^{*}$ and omit the symbols $t_{0}$ and $\seq_{t,t_{0}}^{*}$
for simplicity. Then we say that the (partially compactified)  quantum
cluster algebra $\bQClAlg$ is generated by the quantum cluster variables
$X_{i}(t)$.

\subsection{Injective-reachability\label{subsec:Injective-reachability}}

\begin{Def}[Injective-reachability {\cite{qin2017triangular}}]\label{def:inj_reachable}

A seed $t$ is said to be injective-reachable if there exists a seed
$t=\seq t'$ and a permutation $\sigma$ of the unfrozen vertex $I_{\ufv}$
such that we have

\begin{eqnarray}
\pr_{I_{\ufv}}g_{\sigma k}^{t}(t') & = & -f_{k}\label{eq:injective-reachable-condition}
\end{eqnarray}
 for all $k\in I_{\ufv}$. In this case, we denote $t'=t[1]$, $t=t'[-1]$.
We say $t[1]$ is shifted from $t$ and vice versa.

\end{Def}

Let there be given some seed $t\in\Delta^{+}$ and assume that it
is injective-reachable. Then all seeds in $\Delta^{+}$ are injective-reachable
(\cite[Proposition 5.1.4]{qin2017triangular}\cite[Theorem 3.2.1]{muller2015existence}).
In particular, $\mu_{k}t$ is injective-reachable such that $(\mu_{k}t)[1]=\mu_{\sigma k}\seq\mu_{k}(\mu_{k}t)$
for any $k\in I_{\ufv}$. For $k\in I_{\ufv}$, we denote the quantum
cluster variable 
\begin{align}
I_{k}(t) & :=\seq^{*}X_{\sigma k}(t[1])\in\LP(t).\label{eq:inj_cl_var}
\end{align}

\begin{Rem}

Injective-reachability is equivalent to the existence of a green to
red sequence in the sense of \cite{keller2011cluster}. This property
only depends on the principal $B$-matrix $B(t)$. When $B(t)$ is
skew-symmetric, it means that the shift function $[1]$ in the cluster
category and injective modules of the Jacobian algebra can be constructed
using mutations. In particular, $I_{k}(t)$ corresponds to the $k$-th
injective modules, whence the symbol ``$I$''. See \cite[Section 5]{qin2017triangular}\cite{qin2019bases}
for more details.

\end{Rem}

\begin{Rem}

It is worth reminding that \eqref{eq:injective-reachable-condition}
is a special equation that holds in $\Mc(t)$ where $t$ plays the
role of the initial seed. For a general seed $s\in\Delta^{+}$, we
have $\pr_{I_{\ufv}}g_{\sigma k}^{s}(t')\neq-\pr_{I_{\ufv}}g_{k}^{s}(t)\in\Mc(s)$,
because the $g$-vector is transformed under the tropical transformations.

\end{Rem}

Denote $\tB(t)^{T}\Lambda(t)=\left(\begin{array}{cc}
D' & 0\end{array}\right)$, $D'=\mathrm{diag}(\diag_{k})_{k\in I_{\ufv}}$. By Lemma \ref{lem:mutate_compatible_pair},
we have $\tB(t')^{T}\Lambda(t')=\left(\begin{array}{cc}
D' & 0\end{array}\right)$. Recall that $\diag_{i}b_{ik}=-\diag_{k}b_{ki}$ for $i,k\in I_{\ufv}$.

\begin{Lem}[{\cite[Proposition 2.3.3]{qin2019bases}}]\label{lem:inj_symmetrizer}

We have $\diag_{k}=\diag_{\sigma k}$ for any $k\in I_{\ufv}$.

\end{Lem}

\section{Prerequisites on dominance order and decomposition\label{sec:Prerequisites-dominance-order}}

This section mainly follows \cite{qin2017triangular}\cite{qin2019bases}.

\subsection{Dominance order and pointedness}

Fix an initial seed $t_{0}$ and $\Delta^{+}=\Delta_{t_{0}}^{+}$.
For any $t=E=(e_{i})_{i\in I}$ in $\Delta^{+}$, we define the following
partial order on $\Mc(t)$.

\begin{Def}[Dominance order]

For any $m,m'\in\Mc(t)$, we say $m'$ is dominated by $m$ (or inferior
than $m$) with respect to $t$, denoted by $m'\preceq_{t}m$, if
there exists some $n\in\Nufv^{\geq0}(t)=\oplus_{k\in I_{\ufv}}\N e_{k}$
such that $m'=m+p^{*}n$.

\end{Def}

Recall that we can represent $\Mc(t)$ and $\Nufv(t)$ by lattices
of (column) vectors:

\begin{align*}
 & \Mc(t)=\oplus_{i\in I}\Z f_{i}\simeq\Z^{I}\\
 & \Nufv(t)=\oplus_{k\in I_{\ufv}}\Z e_{k}\simeq\Z^{I_{\ufv}}.
\end{align*}
Then the dominance order reads as 
\begin{eqnarray*}
m' & = & m+\tB(t)n.
\end{eqnarray*}

\begin{Lem}[Finite interval {\cite[Lemma 3.1.2]{qin2017triangular}}]\label{lem:finite_interval}

For any $m,m'$ in $\Mc$, there exists finitely many $m''\in\Mc$
such that $m''\preceq_{t}m$ and $m''\succeq_{t}m'$.

\end{Lem}

Let $\cF(\kk)$ denote the fraction field of the base ring $\kk$.
Let there be given a formal sum $Z=\sum_{m\in\Mc(t)}b_{m}X^{m}$,
$b_{m}\in\cF(\kk)$. Its Laurent degree support is defined to be $\supp_{\Mc(t)}Z:=\{m\in\Mc(t)|b_{m}\neq0\}$.
Given $b\in\cF(\kk)$, $bZ$ is called a rescaling of $Z$.

Likewise, for a function $\alpha\in\Hom_{\mathrm{set}}(\Mc(t),\cF(\kk))$,
we define its support to be $\supp_{\Mc(t)}\alpha=\{m\in\Mc(t)|\alpha(m)\neq0\}$.

\begin{Def}[Degrees and pointedness]\label{def:pointed}

The formal sum $Z$ is said to have (leading) degree $g_{Z}\in\Mc(t)$,
denoted by $\deg^{t}Z=g_{Z}$, if $g_{Z}$ is the unique $\prec_{t}$-maximal
element in $\supp_{\Mc(t)}Z$. It is further said to be pointed at
$g_{Z}$ (or $g_{Z}$-pointed) if $b_{g_{Z}}=1$ or, equivalently,
it takes the following form

\begin{eqnarray*}
Z & = & X^{g_{Z}}\cdot(1+\sum_{n\in\Nufv^{>0}(t)}c_{n}Y^{n}),\ c_{n}=b_{g_{Z}+p^{*}n}.
\end{eqnarray*}

Similarly, $Z$ is said to have codegree $\eta$, denoted by $\codeg^{t}Z=\eta_{Z}$,
if $\eta_{Z}$ is the unique $\prec_{t}$-minimal element in $\supp_{\Mc(t)}Z$.
It is further said to be copointed at $\eta_{Z}$ if $b_{\eta_{Z}}=1$,
or, equivalently, it takes the following form

\begin{eqnarray*}
Z & = & X^{\eta_{Z}}\cdot(\sum_{n\in\Nufv^{>0}(t)}c'_{-n}Y^{-n}+1),\ c'_{-n}=b_{\eta_{Z}-p^{*}n}.
\end{eqnarray*}

\end{Def}

\begin{Def}[Normalization]

If $Z=\sum b_{m}X^{m}\in\hLP(t)$ has degree $g$, its normalization
is defined to be the pointed element $[Z]^{t}=b_{g}^{-1}Z\in\hLP(t)\otimes_{\kk}\cF(\kk)$.

\end{Def}

\begin{Def}[Bidegree]

If $Z$ has degree $g$ and codegree $\eta$, we say that it has bidegree
$(g,\eta)$. If it is pointed at $g$ and copointed at $\eta$, we
say that it is bipointed at bidegree $(g,\eta)$.

\end{Def}

\begin{Def}[$F$-function, support dimension]\label{def:F-dim}

Let there be given $Z=X^{m}\cdot\sum_{n\in\Nufv^{\geq0}(t)}c_{n}Y^{n}$
with degree $m$ (equivalently, $c_{0}\neq0$). Its $F$-function
is defined to be $F_{Z}=\sum c_{n}Y^{n}$. We define its $I$-support
to be the set of vertices $\supp_{I}Z=\{k\in I_{\ufv}|\exists c_{n}\neq0,\ n_{k}\neq0\}$.

If $Z$ has bidegree $(g,\eta)$, then its $F$-function is a polynomial
$F_{Z}=\sum_{0\leq n\leq f_{Z}}c_{n}Y^{n}$ with $c_{0}\neq0$, $c_{f_{Z}}\neq0$,
$\eta=g+p^{*}f_{Z}$, which we call the $F$-polynomial. In this case,
we define its support dimension to be $\suppDim Z=f_{Z}$.

\end{Def}

Following \cite{FominZelevinsky07}\cite{Tran09}, we also call the
leading degree $g_{Z}\in\Mc(t)$ the (extended) $g$-vector of $Z$.
Similarly, following \cite{fujiwara2018duality}, we can also call
the support dimension $f_{Z}$ the $f$-vector of $Z$.

Let there be given a subset $\domCone\subset\Mc(t)$.

\begin{Def}[Pointed set]

A $\domCone$-labeled set $\cS=\{S_{g}|g\in\domCone\}$ in $\hLP(t)$
is said to be $\domCone$-pointed, if each $S_{g}$ is $g$-pointed.

\end{Def}

\subsection{Unitriangularity}

Let there be given a $\domCone$-pointed set $\cS=\{S_{g}|g\in\domCone\}\subset\hLP(t)$.
Let $\mm$ denote a chosen non-unital subring of of $\kk$.

\begin{Def}[Unitriangular transition]

A formal sum $\sum_{m\in\domCone}b_{m}S_{m}$, $b_{m}\in\kk$ is said
to be $\prec_{t}$-unitriangular, if $\{m|b_{m}\neq0\}$ has a unique
$\prec_{t}$-maximal element $g$ and $b_{g}=1$.

The formal sum is further called $(\prec_{t},\mm)$-unitriangular,
if we further have $b_{m}\in\mm$ for any $m\neq g$.

A function $Z\in\hLP(t)$ is said to be $\prec_{t}$-unitriangular
to $\cS$ (respectively, $(\prec_{t},\mm)$-unitriangular to $\cS$)
if it has a formal sum decomposition in $\cS$ which is $\prec_{t}$-unitriangular
(respectively, $(\prec_{t},\mm)$-unitriangular).

A collection of functions $Z\in\hLP(t)$ is said to be $\prec_{t}$-unitriangular
to $\cS$ (respectively, $(\prec_{t},\mm)$-unitriangular to $\cS$)
if each function $Z$ is.

\end{Def}

Notice that $\prec_{t}$ restricts to a partial order on $\domCone\subset\Mc(t)$.

\begin{Lem}[Inverse transition {\cite[Lemma 3.1.11]{qin2017triangular}}]\label{lem:inverse_triangular_transition}

Let there be given two $\domCone$-pointed sets $\cS=\{\cS_{g}|g\in\domCone\}$
and $\can=\{\can_{g}|g\in\domCone\}$ in $\hLP(t)$. Then $\can$
is $(\prec_{t},\mm)$-unitriangular to $\cS$ if and only if $\cS$
is $(\prec_{t},\mm)$-unitriangular to $\can$.

\end{Lem}

\begin{proof}

Assume that $\can$ is $(\prec_{t},\mm)$-unitriangular to $\cS$.
Then, for each $g\in\domCone$, we have a $(\prec_{t},\mm)$-unitriangular
decomposition

\begin{eqnarray*}
\can_{g} & = & \sum b_{g,g'}\cS_{g'}
\end{eqnarray*}
where non-zero coefficients are given by: $b_{g,g}=1$, $b_{g,g'}\in\mm$
for $g'\prec_{t}g\in\domCone$.

Therefore, the transition matrix $(b_{g,g'})_{g,g'\in\domCone}$ is
a $\domCone\times\domCone$-matrix which is lower triangular with
respect to the partial order $\prec_{t}$ on $\domCone$, whose diagonal
entries are $1$ and non-diagonal entries belong to $\mm$. We try
to construct its inverse matrix $(c_{g,g'})_{g,g'\in\domCone}$. It
must be a lower triangular matrix such that $c_{g,g}=1$ for all $g$
and, for any $g''\prec_{t}g$, 
\begin{eqnarray*}
\sum_{g''\preceq_{t}g'\preceq_{t}g}c_{g,g'}b_{g',g''} & = & 0.
\end{eqnarray*}
Notice that this is a finite sum by Lemma \ref{lem:finite_interval}.
If $c_{g,g'}$ have been determined and belong to $\mm$ for $g'\succ_{t}g''$,
then $c_{g,g''}$ is uniquely determined and belongs to $\mm$. Recursively,
we determine $c_{g,g'}\in\mm$ for all $g,g'$ as in \cite[Lemma 3.1.11]{qin2017triangular}.
It follows that $\cS$ is $(\prec_{t},\mm)$-unitriangular to $\can$.

\end{proof}

\subsection{Dominance order decomposition}

Let there be given a $\Mc(t)$-pointed set $\cS=\{S_{g}|g\in\Mc(t)\}$
in $\hLP(t)$. For any $Z\in\hLP(t)$, \cite{qin2019bases} gives
an algorithm to decompose it into a well defined sum of $S_{g}$.

\begin{Definition-Lemma}[{\cite[4.1.1]{qin2019bases}}]\label{def:dominance_decomposition}

There exists a unique decomposition 
\begin{eqnarray*}
Z & = & \sum_{g\in\Mc(t)}\alpha_{t}(Z)(g)\cdot S_{g},\ \alpha_{t}(Z)\in\Hom_{\mathrm{set}}(\Mc(t),\kk)
\end{eqnarray*}
in $\widehat{\LP}(t)$ such that the support $\supp(\alpha_{t}(Z))=\{g|\alpha_{t}(Z)(g)\neq0\}$
has finitely many $\prec_{t}$-maximal elements. It is called the
$\prec_{t}$-decomposition of $Z$ in terms of $\cS$.

\end{Definition-Lemma}

This $\prec_{t}$-decomposition will be called the dominance order
decomposition with respect to the seed $t$. We recall that, by \cite[4.1.1]{qin2019bases},
the maximal elements in $\supp(\alpha_{t}(Z))$ are the same as the
$\prec_{t}$-maximal Laurent degrees in the Laurent degree support
$\supp_{\Mc(t)}Z$. 

Notice that $\cS$ is a topological $\kk$-module basis of $\hLP(t)$
in the sense of \cite[Lemma 2.4]{davison2019strong}.

Let $\cS'$ denote any subset of $\cS$, then it is $\domCone$-pointed
for some subset $\domCone\subset\Mc(t)$. Conversely, any $\domCone$-pointed
set $\cS'$ can be extended to a $\Mc(t_{0})$-pointed set by appending
more functions such as Laurent monomials. If the $\prec_{t}$-decomposition
of $Z$ in terms of $\cS$ only has elements in $\cS'$ appearing,
we call it the $\prec_{t}$-decomposition of $Z$ in terms of $\cS'$.

\begin{Lem}[{\cite[Lemma 3.1.10]{qin2017triangular}\cite[4.1.1]{qin2019bases}}]\label{lem:finite_decomposition_triangular}

If $Z\in\hLP(t)$ has a finite decomposition in terms of $\cS'$,
then the decomposition is the $\prec_{t}$-decomposition. In particular,
if $Z$ is pointed, then the decomposition is $\prec_{t}$-unitriangular.

\end{Lem}

Let there be given a ring $R$ containing $\kk$ as a subring. Denote
$\hLP(t)_{R}=\hLP(t)\otimes_{\kk}R$. Replacing $\kk$ by $R$ in
Definition-Lemma \ref{def:dominance_decomposition}, we can define
the unique $\prec_{t}$-decomposition for $Z'\in\hLP(t)_{R}$ in terms
of a given $\Mc(t)$-pointed set with coefficients in $R$. The following
statement will be convenient for latter proofs.

\begin{Lem}\label{lem:decomposition_subring}

If $Z\in\hLP(t)$ has a finite decomposition in terms of $\cS'$ in
$\hLP(t)_{R}$, then it is the $\prec_{t}$-decomposition in $\hLP(t)$
with coefficients in $\kk$.

\end{Lem}

\begin{proof}

On the one hand, by Lemma \ref{lem:finite_decomposition_triangular},
the finite decomposition of $Z$ in terms of $\cS'$ is the $\prec_{t}$-decomposition
in $\hLP(t)_{R}$.

On the other hand. Extend $\cS'$ to a $\Mc(t)$-pointed set $\cS\subset\hLP(t)$.
Working in $\hLP(t)$, $Z$ has a $\prec_{t}$-decomposition in terms
of $\cS$ with coefficients in $\kk$. By the uniqueness of the $\prec_{t}$-decomposition,
this is also the $\prec_{t}$-decomposition of $Z$ in terms of $\cS$
in $\hLP(t)_{R}$. 

The above two decomposition are the same by the uniqueness of $\prec_{t}$-decomposition.
The claim follows.

\end{proof}

\subsection{Change of seeds}

Let there be given seeds $t=\overleftarrow{\mu}_{t,t'}t'$ in $\Delta^{+}$.
For simplicity, we denote $g'=\phi_{t',t}g$ for any $g\in\Mc(t)$. 

\begin{Def}[Compatibly pointedness]\label{def:compatibly_pointed}

A Laurent polynomial $Z\in\LP(t)\cap\seq_{t',t}^{*}\LP(t')$ is said
to be compatibly pointed at $t,t'$ if $Z$ is $g_{Z}$-pointed for
some $g_{Z}\in\Mc(t)$ and $\overleftarrow{\mu}_{t,t'}^{*}Z$ is a
$g_{Z}'$-pointed Laurent polynomial in $\LP(t')$.

Given a subset $\Delta'\subset\Delta^{+}$ and $t\in\Delta'$, $Z\in\cap_{t'\in\Delta'}\seq_{t',t}^{*}\LP(t')$
is said to be compatibly pointed at $\Delta'$ if it is compatibly
pointed at any pair of seeds from $\Delta'$.

Given a subset $\domCone\subset\Mc(t)$, a $\domCone$-pointed set
$\cS=\{S_{g}|g\in\domCone\}\subset\LP(t)$ is said to be compatibly
pointed at $t$ and $t'$, if all $S_{g}$ do. Similarly, it is said
to be compatibly pointed at $\Delta'$ if all $S_{g}$ do.

A $\domCone$-pointed set $\cS=\{S_{g}|g\in\domCone\}$ in $\LP(t)$
and a $\phi_{t',t}\domCone$-pointed set $\cS'=\{S'_{g'}|g'\in\phi_{t',t}\domCone\}$
in $\LP(t')$ are said to be compatible, if $\seq_{t,t'}^{*}S_{g}=S'_{\phi_{t',t}g}$
for all $g$.

\end{Def}

It might be possible to rephrase Definition \ref{def:compatibly_pointed}
for suitable formal Laurent series. But Laurent polynomial cases suffice
for our purpose in this paper.

Now assume that $t=\mu_{k}t'$ in $\Delta^{+}$ for some $k\in I_{\ufv}$.
We further assume that the given $\Mc(t)$-pointed set $\cS=\{S_{g}|g\in\Mc(t)\}$
consists of elements from $\LP(t)\cap(\mu_{k}^{*})^{-1}\LP(t')$ compatibly
pointed at $t,t'$. Then the corresponding dominance order decomposition
enjoys the following invariance property.

\begin{Prop}[Mutation-invariance {\cite[Proposition 4.2.1]{qin2019bases}}]\label{prop:invariance_dominance_decomposition}

For any Laurent polynomial $Z\in\LP(t)\cap(\mu_{k}^{*})^{-1}\LP(t')$,
its $\prec_{t}$-decomposition in terms of $\cS$ in $\LP(t)$ and
the $\prec_{t'}$-decomposition of $\mu_{k}^{*}Z$ in terms of $\mu_{k}^{*}\cS$
in $\LP(t')$ have the same coefficients:

\begin{eqnarray*}
\alpha_{t}(Z)(g) & = & \alpha_{t'}(\mu_{k}^{*}Z)(\phi_{t',t}g),\ \forall g\in\Mc(t).
\end{eqnarray*}

\end{Prop}

\begin{Def}[Degree transformation \cite{qin2019bases}]

We define the linear map $\psi_{t',t}:\Mc(t)\rightarrow\Mc(t')$ such
that $\psi_{t',t}(\sum g_{i}f_{i})=\sum g_{i}\phi_{t',t}f_{i}$ for
any $(g_{i})\in\Z^{I}$.

\end{Def}

\reviseStart

Note that $\psi_{t',t}$ is bijective. By \cite[Lemma 3.3.4]{qin2019bases},
the mutation map $\seq_{t,t'}^{*}:\cF(t)\simeq\cF(t')$ induces a
natural embedding $\seq_{t,t'}^{*}$ from $\LP(t)$ to $\hLP(t')$
called the formal Laurent series expansion. By \cite[Lemma 3.3.7]{qin2019bases},
$\psi_{t',t}$ describes the change of degrees of Laurent monomials
under this embedding. 

\reviseEnd

\begin{Def}[Support dimensions {\cite{qin2019bases}}]

Assume $t$ is injective-reachable. For any $g\in\Mc(t)$, if there
exists some $n\in N_{\ufv}^{\geq0}(t)\simeq\N^{I_{\ufv}}$ such that
\begin{eqnarray*}
\psi_{t[-1],t}^{-1}\phi_{t[-1],t}g & = & g+\tB(t)\cdot n,
\end{eqnarray*}
we say $g$ has the support dimension $\suppDim g=n$.

\end{Def}

Let there be given injective-reachable seeds $t=\seq t[-1]$.

\begin{Prop}[{\cite[Proposition 3.4.7]{qin2019bases}}]\label{prop:support_compatible_equivalent}

Let there be given a $g$-pointed element $Z\in\LP(t)\cap(\seq^{*})^{-1}\LP(t[-1])$,
$g\in\Mc(t)$. Then $Z$ is compatibly pointed at $t,t[-1]$ if and
only if it is bipointed with support dimension $\suppDim g$.

\end{Prop}

\begin{proof}

Assume $Z$ is copointed at $\eta=g+p^{*}\suppDim g$. By \cite[Propositions 3.3.9, 3.3.10]{qin2019bases},
we deduce that $\deg^{t[-1]}\seq^{*}Z=\psi_{t[-1],t}\eta=\phi_{t[-1],t}g$,
where the last equality follows from the definition of $\suppDim g$. 

Conversely, if $Z$ is compatibly pointed at $t,t[-1]$, then \cite[Propositions 3.3.9, 3.3.10]{qin2019bases}
implies that $Z$ must be copointed at some $\eta$, such that $\psi_{t[-1],t}\eta=\phi_{t[-1],t}g$.
It follows from definition that $\eta=g+p^{*}\suppDim g$.

\end{proof}

In the proof of Proposition \ref{prop:compatible_to_admissible},
we will need the fact that the cluster monomials are compatibly pointed,
which is a consequence of the sign-coherence of $g$-vectors \cite{DerksenWeymanZelevinsky09}\cite{gross2018canonical},
see \cite[Remark 7.13]{FominZelevinsky07}. By Proposition \ref{prop:support_compatible_equivalent},
this fact implies that the support dimension of a $g$-pointed cluster
monomial (Definition \ref{def:F-dim}) provides $\suppDim g$. 

\begin{Rem}\label{rem:suppdim_theta_func}

Although we do not need the following result for proving the main
Theorem (Theorem \ref{thm:main_intro}), it is worth remarking that
the $\theta$-basis in \cite{gross2018canonical} provides the existence
of $\suppDim g$ for all $g$, where the initial seed $t$ is assumed
to be injective-reachable. In particular, $\suppDim g$ only depends
on the principal part $B(t)$.

In view of the Fock-Goncharov dual basis conjecture, the dimensions
$\suppDim g$ should be understood as the correct support dimension
possessed by good basis elements, see \cite{qin2019bases} for more
details.

\end{Rem}

\section{Similar seeds and a correction technique\label{sec:Similar-seeds-and-correction}}

\subsection{Similar seeds \label{subsec:similar-seeds}}

Following \cite{Qin12}\cite[Section 4]{qin2017triangular}, we consider
similar seeds with the same principal $B$-matrix. Let there be given
fixed data such as $I,I_{\ufv},\Mc$ and $N_{\ufv}$ in Section \ref{sec:Basics-cluster-algebra}.
The following definition is more restricted than those in \cite[Section 4]{qin2017triangular}
but it suffices for our purpose.

\begin{Def}[Similar seeds]

Let $\sigma$ denote a permutation of $I_{\ufv}$. Two seeds $t^{(1)}$,
$t^{(2)}$ (not necessarily related by mutations) are called similar
up to $\sigma$ if $b_{\sigma(i)\sigma(j)}(t^{(2)})=b_{ij}(t^{(1)})$
for any $i,j\in I_{\ufv}$. Two quantum seeds are called similar up
to $\sigma$ if we further have $\diag_{k}(t^{(1)})=\diag_{\sigma k}(t^{(2)})$
for all $k\in I_{\ufv}$.

\end{Def}

For simplicity, we replace the symbol $(t^{(r)})$ by $^{(r)}$ when
describing data associated to the seeds $t^{(r)}$, $r=1,2$. 

The permutation $\sigma$ on $I_{\ufv}$ induces a linear map $\sigma:N_{\ufv}^{(1)}\simeq N_{\ufv}^{(2)}$
such that $\sigma e_{k}^{(1)}=e_{\sigma k}^{(2)}$ where $e_{k}^{(r)}$
are the $k$-th unit vector in $N_{\ufv}{}^{(r)}$, $r=1,2$. Extend
$\sigma$ by identity to a permutation $\sigma$ on $I$. Then we
obtain a linear map $\sigma:\Mc(t{}^{(1)})\simeq\Mc(t^{(2)})$ such
that $\sigma f_{i}^{(1)}=f_{\sigma i}^{(2)}$ where $f_{i}^{(r)}$
are the $i$-th unit vectors in $\Mc(t^{(r)})$, $r=1,2$.

Make the natural identification $\Mc(t^{(r)})\simeq\Z^{I}$ and $N_{\ufv}^{(r)}\simeq\Z^{I_{\ufv}}$.
We have the following result.

\begin{Lem}\label{lem:compare_order_similar_seeds}

If $\eta=g+\tB^{(1)}n$, for $n\in\Z^{I_{\ufv}}$, then $\sigma\eta=\sigma g+\tB^{(2)}\sigma n+u$
for some $u\in\Z^{I_{\fv}}$. In particular, $\eta\prec_{t^{(1)}}g$
implies that $\sigma\eta\prec_{t^{(2)}}(\sigma g+u)$ for some $u\in\Z^{I_{\fv}}$.

\end{Lem}

\begin{proof}

We have that $\sigma\tB^{(1)}e_{k}^{(1)}=\sigma\sum b_{ik}^{(1)}f_{i}^{(1)}=\sum_{i\in I}b_{ik}^{(1)}f_{\sigma i}^{(2)}=\sum_{i\in I}b_{\sigma i,\sigma k}^{(2)}f_{\sigma i}^{(2)}+u=\tB^{(2)}e_{\sigma k}^{(2)}+u=\tB^{(2)}\sigma e_{k}^{(1)}+u$
for some $u\in\Z^{I_{\fv}}$. The claim follows by the linearity of
$\sigma$.

\end{proof}

Let $\seq=\mu_{k_{r}}\cdots\mu_{k_{1}}$ denote any given mutation
sequence. Define $\sigma\seq=\mu_{\sigma k_{r}}\cdots\mu_{\sigma k_{1}}$.
The following result is a consequence of Theorem \ref{thm:cluster_expansion}.

\begin{Thm}[Quantum $F$-polynomial]\label{thm:similar_cluster_variable}

For any $j\in I$, denote the (quantum) cluster variable $\seq^{*}X_{j}(\seq t^{(1)})\in\qClAlg(t^{(1)})$
by

\begin{eqnarray*}
 &  & \seq^{*}X_{j}(\seq t^{(1)})=X(t^{(1)})^{g^{(1)}}\cdot\sum_{n\in\N^{I_{\ufv}}}c_{n}Y(t^{(1)})^{n}
\end{eqnarray*}
where $g^{(1)}\in\Z^{I}\simeq\Mc(t^{(1)})$, $c_{n}\in\kk$. Then
there exists some frozen factor $p\in\cRing$ such that the quantum
cluster variable $(\sigma\seq)^{*}X_{\sigma j}((\sigma\seq)t^{(2)})\in\qClAlg(t^{(2)})$
takes the form
\begin{eqnarray*}
 & (\sigma\seq)^{*}X_{\sigma j}((\sigma\seq)t^{(2)})= & p\cdot X(t^{(2)})^{\sigma g^{(1)}}\cdot\sum_{n\in\N^{I_{\ufv}}}c_{n}Y(t^{(2)})^{\sigma n}.
\end{eqnarray*}

\end{Thm}

\subsection{A correction technique\label{subsec:Correction-technique}}

Following \cite{Qin12}\cite[Section 4]{qin2017triangular}, we explain
how to translate algebraic relations from one seed to a similar one,
called a correction technique. Let there be given similar quantum
seeds $t^{(1)}$, $t^{(2)}$ as before, which are not necessarily
related by mutations.

For $r=1,2$, given $g^{(r)}$-pointed function $Z^{(r)}\in\hLP(t^{(r)})$,
where $g^{(r)}\in\Z^{I}\simeq\Mc(t^{(r)})$. Then $Z^{(r)}$ takes
the form $X(t^{(r)})^{g^{(r)}}\cdot\sum_{n\in\N^{I_{\ufv}}}c_{n}^{(r)}Y(t^{(r)})^{n}$,
where $c_{0}^{(r)}=1$, $c_{n}^{(r)}\in\kk$. Recall that we have
the natural projection $\pr_{I_{\ufv}}:\Z^{I}\rightarrow\Z^{I_{\ufv}}$.

\begin{Def}

$Z^{(1)}$ and $Z^{(2)}$ are called similar (as pointed functions)
if $\pr_{I_{\ufv}}g^{(2)}=\sigma\pr_{I_{\ufv}}g^{(1)}$ and $c_{\sigma n}^{(2)}=c_{n}^{(1)}$
for $n\in\N^{I_{\ufv}}$.

\end{Def}

It follows from definition that similar pointed Laurent polynomials
have the same support dimension up to the permutation $\sigma$ on
$I_{\ufv}$.

\begin{Thm}[Correction technique {\cite[Theorem 4.2.1]{qin2017triangular}}]\label{thm:correction_technique}

Let there be given pointed functions $M_{0}^{(r)},M_{1}^{(r)},\cdots,M_{s}^{(r)}$
and (possibly infinitely many) $Z_{j}^{(r)}$ in $\hLP(t^{(r)})$,
$r=1,2$, such that $M_{i}^{(1)}$ and $M_{i}^{(2)}$ are similar,
$Z_{j}^{(1)}$ and $Z_{j}^{(2)}$ are similar. Assume that all $Z_{j}^{(1)}$
satisfy $\deg^{t^{(1)}}Z_{j}^{(1)}=\deg^{t^{(1)}}Z_{0}^{(1)}+\tB(t^{(1)})u_{j}$
for some $u_{j}\in\N^{I_{\ufv}}$, and we have the following algebraic
equations in $\hLP(t^{(1)})$

\begin{align*}
M_{0}^{(1)} & =[M_{1}^{(1)}*\cdots*M_{s}^{(1)}]^{t^{(1)}}\\
c_{0}Z_{0}^{(1)} & =\sum_{j\neq0}c_{j}Z_{j}^{(1)}
\end{align*}
such that $c_{j}\in\kk$, $c_{0}\neq0$. Then we have the following
algebraic equations in $\hLP(t^{(2)})$:

\begin{align*}
M_{0}^{(2)} & =p_{M}\cdot[M_{1}^{(2)}*\cdots*M_{s}^{(2)}]^{t^{(2)}}\\
c_{0}Z_{0}^{(2)} & =\sum_{j\neq0}c_{j}p_{j}\cdot Z_{j}^{(2)}
\end{align*}
where the correction coefficients $p_{M},p_{j}\in\cRing$ are determined
by
\begin{align*}
\deg^{t^{(2)}}p_{M} & =\deg^{t^{(2)}}M_{0}^{(2)}-\sum_{i=1}^{s}\deg^{t^{(2)}}M_{i}^{(2)}\\
\deg^{t^{(2)}}p_{j} & =\deg^{t^{(2)}}Z_{0}^{(2)}+\tB(t^{(2)})\sigma u_{j}-\deg^{t^{(2)}}Z_{j}^{(2)}.
\end{align*}

\end{Thm}

\begin{Lem}\label{lem:similar_inverse}

If $Z^{(1)}$and $Z^{(2)}$ are similar, so are their inverses $(Z^{(1)})^{-1}$
and $(Z^{(2)})^{-1}$.

\end{Lem}

\begin{proof}

For $r=1,2$, we have 
\begin{align*}
Z^{(r)} & =X(t^{(r)})^{g^{(r)}}\cdot\sum c_{n}^{(r)}Y(t^{(r)})^{n}\\
 & =X(t^{(r)})^{g^{(r)}}*\sum v^{\sum_{k\in I_{\ufv}}\diag_{k}g_{k}^{(r)}n_{k}}c_{n}^{(r)}Y(t^{(r)})^{n}\\
 & =:X(t^{(r)})^{g^{(r)}}*\sum_{n\in\N^{I_{\ufv}}}b_{n}^{(r)}Y(t^{(r)})^{n}\\
 & =:X(t^{(r)})^{g^{(r)}}*G^{(r)}
\end{align*}

Because $t^{(1)},t^{(2)}$ are similar and $Z^{(1)},Z^{(2)}$ are
similar, we obtain that $G^{(1)}$and $G^{(2)}$ are similar. It follows
that $(G^{(1)})^{-1}$ and $(G^{(2)})^{-1}$ are similar formal series
in the variables $Y_{k}(t^{(1)})$ and $Y_{\sigma k}(t^{(2)})$, $k\in I_{\ufv}$,
respectively. Therefore, the inverses $Z^{(r)}=(G^{(r)})^{-1}*X(t^{(r)})^{-g^{(r)}}$
are similar by Theorem \ref{thm:correction_technique} or direct computation.

\end{proof}

Let there be given any mutation sequence $\seq$. Denote $s^{(1)}=\seq t^{(1)}$
and $s^{(2)}=(\sigma\seq)t^{(2)}$. 

\begin{Lem}[{\cite[Lemma 4.2.2]{qin2017triangular}}]\label{lem:similar_cluster_mutation}

The following statements are true.

(1) $s^{(1)}$ and $s^{(2)}$ are similar quantum seeds up to the
permutation $\sigma$.

(2) The quantum cluster monomials $\seq^{*}X(s^{(1)})^{m}\in\LP(t^{(1)})$
and $(\sigma\seq)^{*}X(s^{(2)})^{\sigma m}\in\LP(t^{(2)})$, $m\in\N^{I}$,
are similar.

\end{Lem}

As a consequence of Lemma \ref{lem:similar_cluster_mutation}, we
obtain the following result.

\begin{Lem}\label{lem:similar_Laurent_monomial_mutation}

For any $m\in\Z^{I_{\ufv}}$, the pointed functions $\seq^{*}X(s^{(1)})^{m}\in\hLP(t^{(1)})$
and $(\sigma\seq)^{*}X(s^{(2)})^{\sigma m}\in\hLP(t^{(2)})$ are similar.

\end{Lem}

\begin{proof}

\reviseStart

By Lemma \ref{lem:similar_cluster_mutation}, we know that $M_{1}^{(r)}:=(\sigma^{r-1}\seq)^{*}X(s^{(r)})^{[\sigma^{r-1}m]_{+}}$,
$r=1,2$, are similar, and $(\sigma^{r-1}\seq)^{*}X(s^{(r)})^{[-\sigma^{r-1}m]_{+}}$,
$r=1,2$, are also similar. Then Lemma \ref{lem:similar_inverse}
implies that the inverse $M_{2}^{(r)}:=(\sigma^{r-1}\seq)^{*}X(s^{(r)})^{-[-\sigma^{r-1}m]_{+}}$,
$r=1,2$, are similar. Finally, let us use Theorem \ref{thm:correction_technique}.
Denote 
\begin{align*}
M_{0}^{(1)} & :=\seq^{*}X(s^{(1)})^{m}\\
 & =[\seq^{*}X(s^{(1)})^{[m]_{+}}*\seq^{*}X(s^{(1)})^{-[-m]_{+}}]^{s^{(1)}}\\
 & =[M_{1}^{(1)}*M_{2}^{(1)}]^{s^{(1)}}.
\end{align*}
Take any $M_{0}^{(2)}$ in $\hLP(t^{(2)})$ such that it is similar
to $M_{0}^{(1)}$. Then Theorem \ref{thm:correction_technique} implies
that we have $M_{0}^{(2)}=p\cdot[M_{1}^{(2)}*M_{2}^{(2)}]^{s^{(2)}}$
for some frozen factor $p\in\cRing$. It follows that $[M_{1}^{(2)}*M_{2}^{(2)}]^{s^{(2)}}$
is similar to $M_{0}^{(1)}$. The desired claim follows from the equality
$[M_{1}^{(2)}*M_{2}^{(2)}]^{s^{(2)}}=(\sigma\seq)^{*}X(s^{(2)})^{\sigma m}$.\reviseEnd

\end{proof}

The following natural result is not needed for the main Theorem, though
it will be used in the proof of Proposition \ref{prop:compatible_to_chain_bases}.

\begin{Prop}\label{prop:mutation_similar_elements}

Denote $t^{(1)}=\seq t'^{(1)}$ and $t^{(2)}=(\sigma\seq)t'{}^{(2)}$.
Let there be given similar pointed functions $Z^{(1)}\in\LP(t^{(1)})\cap(\seq^{*})^{-1}\LP(t'^{(1)})$
and $Z^{(2)}\in\LP(t^{(2)})$ such that  $\seq^{*}Z^{(1)}$ is pointed.
Then $(\sigma\seq)^{*}Z^{(2)}$ is contained in $\LP(t'^{(2)})$,
is pointed, and is similar to $\seq^{*}Z^{(1)}$.

\end{Prop}

\begin{proof}

Let us omit the symbol $t$ for simplicity, and denote $X(t'^{(i)})$,
$Y(t'^{(i)})$ by $X'^{(i)}$, $Y'^{(i)}$ respectively.

Denote $Z^{(1)}=p_{1}\cdot\sum_{n\in\N^{I_{\ufv}}}c_{n}(X^{(1)})^{m}\cdot(Y^{(1)})^{n}$
for $m\in\Z^{I_{\ufv}}$, $p_{1}\in\cRing$, $c_{n}\in\kk$, $c_{0}=1$.
Then the similar element $Z^{(2)}$ takes the form $Z^{(2)}=p_{2}\cdot\sum c_{n}(X{}^{(2)})^{\sigma m}\cdot(Y{}^{(2)})^{\sigma n}$,
$p_{2}\in\cRing$. Applying the mutation $\seq^{*}$, we obtain the
following equality in $\hLP(t'^{(1)})$:

\begin{eqnarray*}
\seq^{*}Z^{(1)} & = & p_{1}\cdot\sum c_{n}\seq^{*}((X^{(1)})^{m}\cdot(Y{}^{(1)})^{n}).
\end{eqnarray*}
Similarly, we have the following equality in $\hLP(t'^{(2)})$:

\begin{eqnarray*}
(\sigma\seq)^{*}Z^{(2)} & = & p_{2}\cdot\sum c_{n}(\sigma\seq)^{*}((X^{(2)})^{\sigma m}\cdot(Y{}^{(2)})^{\sigma n}).
\end{eqnarray*}

Denote $\deg^{t^{'(i)}}$ by $\deg^{(i)}$ for simplicity. By \cite[Proposition 3.3.8]{qin2019bases},
we have $\deg^{(i)}\seq^{*}(Y^{(i)})^{n}=\deg^{(i)}(Y'^{(i)})^{C^{(i)}n}$,
where $C^{(i)}:=C^{t'^{(i)}}(t{}^{(i)})$ is the $C$-matrix associated
to the seed $t^{(i)}$ with respect to the initial seed $t'^{(i)}$,
$i=1,2$.

By construction of the $C$-matrices (Definition \ref{def:C-matrix})
and the similarity between $t^{(1)}$, $t^{(2)}$, we have $C_{ij}^{(1)}=C_{\sigma i,\sigma j}^{(2)}$.
It follows that $\sigma(C^{(1)}n)=C^{(2)}\sigma n$, $n\in\Z^{I_{\ufv}}$
(see the proof of Lemma \ref{lem:compare_order_similar_seeds}). Denote
$C^{(1)}=C$ and $g^{(i)}=\deg^{(i)}(\sigma^{i-1}\seq)^{*}(X{}^{(i)})^{\sigma^{i-1}m}$.
Then $(\sigma^{i-1}\seq)^{*}((X^{(i)})^{\sigma^{i-1}m}\cdot(Y^{(i)})^{\sigma^{i-1}n})$
is pointed at $g^{(i)}+\deg^{(i)}(Y'^{(i)})^{\sigma^{i-1}Cn}$, where
$i=1,2$.

Notice that $(\sigma\seq)^{*}((X^{(2)})^{\sigma m}\cdot(Y{}^{(2)})^{\sigma n})$
are similar to $\seq^{*}((X^{(1)})^{m}\cdot(Y{}^{(1)})^{n})$ by Lemma
\ref{lem:similar_Laurent_monomial_mutation}. Let $V$ denote a pointed
function in $\LP(t'^{(2)})$ similar to $\seq^{*}Z^{(1)}$. Then the
correction technique (Theorem \ref{thm:correction_technique}) implies
that there exists some $p_{3}\in\cRing$ such that the following holds
in $\hLP(t'^{(2)})$:

\begin{eqnarray*}
V & = & p_{3}\cdot\sum c_{n}(\sigma\seq)^{*}((X^{(2)})^{\sigma m}\cdot(Y{}^{(2)})^{\sigma n}).
\end{eqnarray*}
It follows that we have $(\sigma\seq)^{*}Z^{(2)}=V\cdot p_{3}^{-1}\cdot p_{2}$,
which is pointed and similar to the Laurent polynomial $\seq^{*}Z^{(1)}$.

\end{proof}

\section{Cluster twist automorphisms\label{sec:Twist-automorphisms}}

Twist automorphisms exist on unipotent cells and is related to Chamber
Ansatz problem, see \cite[Section 1.3]{kimura2017twist} for more
details and a quantum analogue on quantum unipotent cells. In this
section, for general quantum cluster algebras not necessarily having
a Lie theoretic background, we consider an analogous notion. We call
our construction cluster twist automorphisms or twist automorphisms
for short.

\subsection{Twist automorphisms passing through similar seeds\label{subsec:Twist-automorphisms}}

Let there be given quantum seeds $t'=\seq t\in\Delta^{+}$ such that
they are similar up to a permutation $\sigma$. Assume that we have
an algebra homomorphism $\var^{t}:\cF(t)\simeq\cF(t')$ which preserves
the twisted product. In view of Theorem \ref{thm:similar_cluster_variable}
and the correction technique in Section \ref{subsec:Correction-technique},
we want it to send $Y$-variables to $Y$-variables as follows.

\begin{Def}\label{def:variation_map}

The algebra isomorphism $\var^{t}:\cF(t)\simeq\cF(t')$ is called
a variation map if, for any $k\in I_{\ufv}$, there exists some $p_{k}\in\cRing$
such that

\begin{eqnarray*}
\var^{t}(X_{k}) & = & p_{k}\cdot X_{\sigma k}\\
\var^{t}(Y_{k}(t)) & = & Y_{\sigma k}(t')\\
\var^{t}(\cRing) & = & \cRing.
\end{eqnarray*}

\end{Def}

By definition, if $\var^{t}$ is a variation map, so does its inverse.

We will be interested in the case that $\var^{t}$ is the map induced
from a linear bijection $\var^{t}:\Mc(t)\simeq\Mc(t')$ such that 

\begin{eqnarray*}
\var^{t}(f_{k}) & = & f_{\sigma k}'+u_{k}\\
\var^{t}(\sum_{i\in I}b_{ik}f_{i}) & = & \sum_{i\in I}b_{\sigma i,\sigma k}'f_{\sigma i}'\\
\var^{t}(\oplus_{j\in I_{\fv}}\Z f_{j}) & = & \oplus_{j\in I_{\fv}}\Z f_{j}'
\end{eqnarray*}
where $f_{i}$, $f_{i}'$ are the $i$-th unit vectors in $\Mc(t)$
and $\Mc(t')$ respectively, such that $u_{k}\in\Z^{I_{\fv}}$ and
$p_{k}=X^{u_{k}}$. This linear map has the following property by
definition.

\begin{Lem}\label{lem:similar_dominance_order}

If $g'\prec_{t}g$ in $\Mc(t)$, then $\var^{t}g'\prec_{t'}\var^{t}g$
in $\Mc(t')$.

\end{Lem}

Recall that we have the mutation birational map $\seq^{*}:\cF(t')\simeq\cF(t)$.

\begin{Def}[Cluster twist automorphism]\label{def:twist_automorphism}

The composition $\tw^{t}=\seq^{*}\var^{t}:\cF(t)\simeq\cF(t)$ is
called a cluster twist automorphism passing through $t'$ (twist automorphism
for short). Here, $\var^{t}:\cF(t)\simeq\cF(t')$ is called its variation
part and $\seq^{*}:\cF(t')\simeq\cF(t)$ its mutation part.

\end{Def}

Notice that $t=\seq^{-1}t'$. If $\var^{t}$ is a variation map, then
$\var^{t'}:=(\var^{t})^{-1}:\cF(t')\simeq\cF(t)$ is a variation map
too. Similar to the definition of the twist automorphism $\tw^{t}$,
it is natural to consider the composition in the opposite order $\var^{t'}(\seq^{-1})^{*}:\cF(t)\simeq\cF(t)$.
It turns out that $\var^{t'}(\seq^{-1})^{*}=(\tw^{t})^{-1}$, since
$(\seq^{-1})^{*}=(\seq^{*})^{-1}$.

\begin{Lem}\label{lem:twist_permute_loc_cl_monom}

A twist automorphism $\tw^{t}$ is a permutation on the set of the
localized quantum cluster monomials.

\end{Lem}

\begin{proof}

Let there be given any localized quantum cluster monomial $Z\in\clAlg(t)\subset\LP(t)$.
By the definition of a variation map and Theorem \ref{thm:similar_cluster_variable},
$\var^{t}(Z)$ is a localized quantum cluster monomial in $\clAlg(t')\subset\LP(t')$.
Then $\seq^{*}\var^{t}(Z)$ becomes its Laurent expansion in $\clAlg(t)=\seq^{*}\clAlg(t')\subset\LP(t)$. 

The map $\tw^{t}$ is a permutation because its inverse $(\tw^{t})^{-1}=(\var^{t})^{-1}(\seq^{*})^{-1}$
sends localized quantum monomials to localized quantum cluster monomials
by similar arguments as above.

\end{proof}

\begin{Rem}

By definition, given a seed $t$, a twist automorphism might not be
unique, and its existence is not yet clear. We will discuss general
twist automorphisms elsewhere.

\end{Rem}

Next, assume that a twist automorphism $\tw^{t}$ on $\cF(t)$ is
given, we want to show that it gives rise to twist automorphisms for
all seeds.

For any $k\in I_{\ufv}$ and $t\in\Delta^{+}$, consider seeds $s=\mu_{k}t,s'=\mu_{\sigma k}t'\in\Delta^{+}$.
Then we have $s'=\mu_{\sigma k}\seq\mu_{k}(s)=:\seqnu s$ and the
corresponding mutation birational map $(\seqnu)^{*}:\cF(s')\simeq\cF(s)$.
Recall that we also have isomorphisms $\mu_{k}^{*}:\cF(s)\simeq\cF(t)$,
$\mu_{\sigma k}^{*}:\cF(s')\simeq\cF(t')$. Then $\tw^{t}$ gives
rise to an automorphism $\tw^{s}=(\mu_{k}^{*})^{-1}\tw^{t}\mu_{k}^{*}$
on $\cF(s)$. Moreover, we have the decomposition $\tw^{s}=(\seqnu)^{*}\var^{s}$
where we define the algebra isomorphism $\var^{s}:=(\mu_{\sigma k}^{*})^{-1}\var^{t}(\mu_{k}^{*})$.
We summarize our construction in the following commutative diagrams.

\begin{eqnarray*}
\begin{array}{ccc}
s' & \xleftarrow{\seqnu} & s\\
\uparrow\mu_{\sigma k} &  & \uparrow\mu_{k}\\
t' & \xleftarrow{\seq} & t
\end{array}
\end{eqnarray*}

\begin{eqnarray*}
\begin{array}{ccccc}
\cF(s) & \xrightarrow{\var^{s}} & \cF(s') & \xrightarrow{\seqnu{}^{*}} & \cF(s)\\
\downarrow\mu_{k}^{*} &  & \downarrow\mu_{\sigma k}^{*} &  & \downarrow\mu_{k}^{*}\\
\cF(t) & \xrightarrow{\var^{t}} & \cF(t') & \xrightarrow{\seq^{*}} & \cF(t)
\end{array}
\end{eqnarray*}

\begin{Prop}\label{prop:adjacent_variantion_map}

The map $\var^{s}:\cF(s)\simeq\cF(s')$ for $s=\mu_{k}t$ is a variation
map in the sense of Definition \ref{def:variation_map}. 

\end{Prop}

\begin{proof}

We check the defining properties of a variation map one by one, using
the mutation rules for $X$-variables and $Y$-variables.

Take any $i\in I_{\ufv}$. By the mutation rule of $Y$-variables,
$\mu_{k}^{*}Y_{i}(s)$ is a function of $Y_{i}(t)$ and $Y_{k}(t)$.
Similarly, $\mu_{\sigma k}^{*}Y_{\sigma i}(s')$ is a function of
$Y_{\sigma i}(t')=\var^{t}Y_{i}(t)$ and $Y_{\sigma k}(t')=\var^{t}Y_{k}(t)$.
Moreover, the two functions are the same because they only depends
on the principal $B$-matrices associated to the similar seeds $t$
and $t'$, see \eqref{eq:Y_mutation}. It follows that $\var^{t}\mu_{k}^{*}Y_{i}(s)=\mu_{\sigma k}^{*}Y_{\sigma i}(s')$.
Therefore, $Y_{\sigma i}(s')=(\mu_{\sigma k}^{*})^{-1}\var^{t}\mu_{k}^{*}Y_{i}(s)=\var^{s}Y_{i}(s)$.

If $i\neq k$, we get $\var^{t}\mu_{k}^{*}X_{i}(s)=\var^{t}X_{i}(t)=p_{i}\cdot X_{\sigma i}(t')=\mu_{\sigma k}^{*}(p_{i}\cdot X_{\sigma i}(s'))$.
Therefore, $\var^{s}X_{i}(s)=p_{i}\cdot X_{\sigma i}(s')$ where $p_{i}\in\cRing$.

For $i=k$, we have 
\begin{align*}
\var^{t}\mu_{k}^{*}X_{k}(s) & =\var^{t}(X(t)^{-f_{k}+\sum_{j\in I}[-b_{jk}]_{+}f_{j}}*(1+v_{k}^{-1}Y_{k}(t))\\
 & =\var^{t}X(t)^{-f_{k}+\sum_{j\in I}[-b_{j,k}]_{+}f_{j}}*\var^{t}(1+v_{k}^{-1}Y_{k}(t))\\
 & =X(t')^{u-f_{\sigma k}'+\sum_{j\in I}[-b_{\sigma j,\sigma k}']_{+}f_{\sigma j}'}*(1+v_{k}^{-1}Y_{\sigma k}(t'))
\end{align*}
for some $u\in\Z^{I_{\fv}}$ determined by $\var^{t}$. Notice that
we have $\diag_{k}=\diag_{\sigma k}$. Then the above result turns
out to be

\begin{eqnarray*}
\var^{t}\mu_{k}^{*}X_{k}(s) & = & X(t')^{u-f_{\sigma k}'+\sum_{j\in I}[-b_{\sigma j,\sigma k}']_{+}f_{\sigma j}'}*(1+v_{\sigma k}^{-1}Y_{\sigma k}(t'))\\
 & = & p\cdot\mu_{\sigma k}^{*}X_{\sigma k}(s')
\end{eqnarray*}
where $p=X^{u}$. It follows that $\var^{s}X_{k}(s)=p\cdot X_{\sigma k}(s')$.

Finally, we have $\var^{t}\mu_{k}^{*}(\cRing)=\var^{t}\cRing=\cRing=\mu_{\sigma k}^{*}\cRing$,
which implies $\var^{s}\cRing=\cRing$.

\end{proof}

As a consequence, $\tw^{s}$ is a twist automorphism on $\cF(s)$.
Therefore, once we are given a twist automorphism $\tw^{t}$ on $\cF(t)$
for some seed $t$, we obtain twist automorphisms $\tw^{s}$ on $\cF(s)$
for all seeds $s\in\Delta^{+}$ by repeatedly using Proposition \ref{prop:adjacent_variantion_map}.
We deduce the following result.

\begin{Thm}\label{thm:global_twist_automorphism}

The twist automorphism $\tw^{t}$ on $\cF(t)$ gives rise to a twist
automorphism $\tw^{s}=(\seq_{t,s}^{*})^{-1}\tw^{t}\seq_{t,s}^{*}$
on $\cF(s)$ for any seed $s=\seq_{s,t}t\in\Delta^{+}$.

\end{Thm}

Notice that the twist automorphisms $\tw^{s}$ in Theorem \ref{thm:global_twist_automorphism}
become the same automorphism once we identify the isomorphic fraction
fields $\cF(s)$ by mutations, $s\in\Delta^{+}$. Correspondingly,
we can use $\tw$ to denote the twist automorphism.

\subsection{Twist automorphism of Donaldson-Thomas type\label{subsec:Twist-automorphism-DT-type}}

In this paper, we will mainly be interested in the case when the pair
of similar seeds is given by $t'=\seq t$ such that $t'=t[1]$. Notice
that $t[1]$ and $t$ are similar as quantum seeds by Lemma \ref{lem:inj_symmetrizer}.

Assume we are given a twist automorphism $\tw^{t}=\seq^{*}\var^{t}$
on $\cF(t)$ passing through $t'=t[1]$, where $\var^{t}:\cF(t)\simeq\cF(t')$
is the variation part. Then for any $k\in I_{\ufv}$ and $s=\mu_{k}t$,
$s'=\mu_{\sigma k}t'$, we have $s'=s[1]$, see \cite[Proposition 5.1.4]{qin2017triangular}\cite[Theorem 3.2.1]{muller2015existence}.
Therefore, the twist automorphism $\tw^{s}$ passes through $s[1]$
as well. Repeat this argument, we obtain that for all $s\in\Delta^{+}$,
the twist automorphisms $\tw^{s}$ pass through $s[1]$.

We say such a twist automorphism $\tw$ to be of Donaldson-Thomas
type (DT-type for short).

\begin{Lem}\label{lem:twist_injective_variable}

Let there be given seeds $t[1]=\seq t$ and an associated twist automorphism
$\tw$ that passes through $t[1]$. Then, for any $k\in I_{\ufv}$,
there exists a frozen factor $p_{k}\in\cRing$ such that 
\begin{eqnarray}
\tw X_{k}(t) & = & p_{k}\cdot I_{k}(t)\label{eq:twist_cl_var}
\end{eqnarray}
where the cluster variable $I_{k}(t)$ is given by \eqref{eq:inj_cl_var}. 

\end{Lem}

\begin{proof}

The variation map $\var^{t}$ identify $X_{k}(t)$ with $p_{k}\cdot X_{\sigma k}(t[1])$
for some $p_{k}\in\cRing$. The claim follows from definitions of
$\tw$ and $I_{k}(t)$.

\end{proof}

\begin{Rem}\label{rem:twist_automorphism_meaning}

We call a twist automorphism $\tw=\seq^{*}\var^{t}$ on $\cF(t)$
passing through $t[1]$ to be of Donaldson-Thomas type for the following
reason. When $B(t)$ is skew-symmetric, there is an associated Jacobian
algebra of a quiver with potential \cite{DerksenWeymanZelevinsky09}
in the categorification approach to cluster algebras. Then the mutation
part $\seq^{*}$ of $\tw$ is often called the Donaldson-Thomas transformation,
and it is categorified by the shift functor in the associated $2$-Calabi-Yau
category, see the inspiring work \cite{Nagao10} for more details.

It would be desirable to further investigate the twist automorphism
in categories. In general, it is unclear how to categorify $\seq^{*}$
for skew-symmetrizable $B(t)$. Moreover, we do not know the categorical
meaning of the variation map.

\end{Rem}

\section{Triangular bases\label{sec:Triangular-bases}}

In this section, we prove some general results concerning triangular
bases and their relation with twist automorphisms. We will apply the
results to cluster algebras arising from quantum unipotent subgroups
in Section \ref{sec:Dual-canonical-bases-results}.

We consider the quantum case $(\kk,v)=(\Z[q^{\pm\Hf}],q^{\Hf})$,
where $q^{\Hf}$ is a formal parameter. Let $\mm=v^{-1}\Z[v^{-1}]$
denote the chosen non-unital subring of $\kk$. Recall that we have
the automorphism of $\kk$ as a $\Z$-module such that $\overline{v}=v^{-1}$,
which extends to the bar involution ($\Z$-linear anti-automorphism)
on the quantum algebra $\hLP(t)$ such that $\overline{vX^{m}}=v^{-1}X^{m}$.
By bar-invariance we mean invariant under this bar involution.

Let $t$ denote a given quantum seed in $\Delta^{+}$.

\subsection{Triangular functions}

Let there be given a $\Mc(t)$-pointed set $\cS=\{S_{g}|g\in\Mc(t)\}$
in the formal Laurent polynomial ring $\hLP(t)$. Then any $Z\in\widehat{\LP}(t)$
has a unique $\prec_{t}$-decomposition in terms of $\cS$ (Definition
\ref{def:dominance_decomposition}). Notice that $\cS$ is a topological
basis in the sense of \cite{davison2019strong}.\reviseStart We choose
such a topological basis $\cS$ and call it the set of distinguished
functions in $\hLP(t)$. \reviseEnd

\begin{Def}[Triangular functions]

For any $g\in\Mc(t)$, the triangular function $\can_{g}$ with respect
to the set of distinguished functions $\cS$ is the $g$-pointed bar-invariant
element in $\hLP(t)$, such that its $\prec_{t}$-decomposition in
terms of $\cS$ is $(\prec_{t},\mm)$-unitriangular.

\end{Def}

Before constructing the triangular functions, we first prove the following
general statement for constructing a bar-invariant difference.

\begin{Prop}[Bar-invariant difference]\label{prop:construct_bar_inv_func}

For any $Z=\sum b_{m}X^{m}\in\hLP(t)$, there exists a unique formal
sum $\sum\alpha_{g}S_{g}$, $\alpha_{g}\in\mm$, with finitely many
$\prec_{t}$-maximal degrees in $\{g|\alpha_{g}\neq0\}$, such that
$Z-\sum\alpha_{g}S_{g}$ is bar-invariant.

\end{Prop}

\begin{proof}

Similar to the construction of $\prec_{t}$-decomposition \cite[Section 4]{qin2019bases},
we give an algorithm for the construction by tracking the Laurent
degrees from larger ones to smaller ones.

For any $Z=\sum b_{m}X^{m}\in\hLP(t)$, denote its (finitely many)
$\prec_{t}$-maximal degrees by $g^{(i)}$. For each corresponding
coefficient $b_{g^{(i)}}$, there exists a unique $\alpha_{g^{(i)}}\in\mm$
such that $b_{g^{(i)}}-\alpha_{g^{(i)}}$ is bar-invariant. More precisely,
$\alpha_{g^{(i)}}$ is the unique solution in $\mm=v^{-1}\Z[v^{-1}]$
such that 
\begin{eqnarray*}
\overline{b_{g^{(i)}}}-b_{g^{(i)}} & = & \overline{\alpha_{g^{(i)}}}-\alpha_{g^{(i)}}.
\end{eqnarray*}
Notice that, $\alpha_{g^{(i)}}\neq0$ if and only if $b_{g^{(i)}}$
is not bar-invariant.

The $\prec_{t}$-maximal degree terms in the difference $Z-\sum\alpha_{g^{(i)}}S_{g}$
have bar-invariant coefficients $b_{g^{(i)}}-\alpha_{g^{(i)}}$. Then
we obtain a bar-invariant term $\sum(b_{g^{(i)}}-\alpha_{g^{(i)}})X^{g^{(i)}}$.
Let us proceed with the difference $Z'=Z-\sum\alpha_{g^{(i)}}S_{g^{(i)}}-\sum(b_{g^{(i)}}-\alpha_{g^{(i)}})X^{g^{(i)}}\in\hLP(t)$.
Repeating this (possibly infinite) process, we obtain a formal sum
$Z-\sum\alpha_{g}S_{g}$ which equals the formal sum of the bar-invariant
terms, where each term is of the form $\sum(b_{g^{(i)}}-\alpha_{g^{(i)}})X^{g^{(i)}}$.

\end{proof}

\begin{Thm}[Existence of triangular functions]\label{thm:construct_tri_func}

For any $g\in\Mc(t)$, the triangular function $\can_{g}$ with respect
to $\cS$ exists.

\end{Thm}

\begin{proof}

By applying Proposition \ref{prop:construct_bar_inv_func} to $S_{g}$
and the $\Mc(t)$-pointed set $\cS$ , we can find unique $\alpha_{g'}\in\mm$
for all $g'\preceq_{t}g$ such that the difference $S_{g}-\sum_{g'}\alpha_{g'}S_{g'}$
is bar-invariant. Moreover, since the leading term of $S_{g}$ has
the bar-invariant coefficient $1$, we must have $\alpha_{g}=0\in\mm$.
Now we obtain the triangular function $\can_{g}=S_{g}-\sum_{g'\prec_{t}g}\alpha_{g'}S_{g'}$,
$\alpha_{g'}\in\mm$.

\end{proof}

\begin{Lem}[Uniqueness]\label{lem:tri_func_unique}

For any $g\in\Mc(t)$, the triangular function $\can_{g}$ with respect
to $\cS$ is unique.

\end{Lem}

\begin{proof}

Let there be given two such triangular functions $\can_{g}$ and $\can_{g}'$.
Then $\can_{g}-\can_{g}'=\sum_{g'\prec_{t}g}\beta_{g'}S_{g'}$ for
some $\beta_{g'}\in\mm$. If the difference does not vanish, the coefficients
of its $\prec_{t}$-maximal degrees take the form $0\neq\beta_{g'}\in\mm$.
But $\can_{g}-\can_{g}'$ is also bar-invariant. It follows that $\beta_{g'}=0$,
a contradiction.

\end{proof}

\begin{Rem}

Recursive constructions similar to those for Theorem \ref{thm:construct_tri_func}
have been used for studying well-known bases in representation theory
\cite[Lemma 8.4]{Nakajima04}\cite[7.10]{Lusztig90}.

Notice that the construction of the triangular functions depends on
the chosen set $\cS$ of the distinguished functions. For the purpose
in this paper, our chosen set of distinguished functions will be $\Inj^{t}$
defined in \eqref{eq:inj_function}. One might also consider other
distinguished functions. For example, by choosing the dual PBW bases
or the standard modules of quantum affine algebras, one can show that
the corresponding triangular functions are the dual canonical bases
or the characters of the simple modules, see \cite[Section 9.1]{qin2017triangular}. 

It is worth reminding that, in our situation, the chosen set $\cS$
is often NOT a basis of the algebra that we consider, unlike the dual
PBW bases or the bases consisting of standard modules.

\end{Rem}

\subsection{The triangular bases}

Given a subalgebra $\midAlg(t)$ of the quantum upper cluster algebra
$\qUpClAlg(t)$, we have $\midAlg(t)\subset\qUpClAlg(t)\subset\LP(t)$.
Under the isomorphism $\seq_{t,t'}^{*}:\cF(t)\simeq\cF(t')$ for any
given $t'=\seq_{t',t}t\in\Delta^{+}$, we obtain that $\midAlg(t'):=\seq_{t,t'}^{*}\midAlg(t)$
satisfies $\midAlg(t')\subset\qUpClAlg(t')\subset\LP(t')$. For simplicity,
we might omit $\seq_{t,t'}^{*}$ and the choice of initial seed, and
simply write $\midAlg\subset\qUpClAlg$.

We further assume that $t$ is injective-reachable (Section \ref{subsec:Injective-reachability}).

\begin{Def}[Triangular bases \cite{qin2017triangular}]\label{def:triangular_basis}

We say a $\kk$-basis $\can$ of the subalgebra $\midAlg(t)\subset\qUpClAlg(t)$
is a triangular basis with respect to the seed $t$, if the following
conditions hold:

\begin{enumerate}

\item (Bar invariance) The basis elements in $\can$ are invariant
under the bar involution.

\item The quantum cluster monomials in $t$ and $t[1]$ are contained
in $\can$.

\item(Parametrization) $\can$ is $\Mc(t)$-pointed, i.e., it takes
the form $\can=\{\can_{g}|g\in\Mc(t)\}$, such that $\can_{g}$ is
$g$-pointed.

\item(Triangularity) For any basis elements $\can_{g}$ and $i\in I$,
the decomposition of the normalized product $[X_{i}(t)*\can_{g}]^{t}$
in terms of $\can$ is $(\prec_{t},\mm)$-unitriangular:

\begin{eqnarray*}
[X_{i}(t)*\can_{g}]^{t} & = & \can_{g+f_{i}}+\sum_{g'\prec_{t}g+f_{i}}b_{g'}\can_{g'},\ b_{g'}\in\mm.
\end{eqnarray*}

\end{enumerate}

\end{Def}

A basis $\can$ of $\midAlg(t)$ naturally gives rise to a basis $\seq_{t,t'}^{*}\can$
of $\midAlg(t')$ for any seeds $t=\seq_{t,t'}t'$.

\begin{Def}[Common triangular bases \cite{qin2017triangular}]

The triangular basis $\can$ of $\midAlg(t)$ with respect to the
seed $t$ is said to be the common triangular basis, if $\seq_{t,t'}^{*}\can$
are the triangular bases of $\midAlg(t')$ for all seeds $t'\in\Delta^{+}$
respectively and they are pairwise compatible.

\end{Def}

Let $\pr_{I_{\ufv}}$ and $\pr_{I_{\fv}}$ denote the natural projection
from $\Mc(t)\simeq\Z^{I}$ to $\Z^{I_{\ufv}}$ and $\Z^{I_{\fv}}$
respectively. For any $g=(g_{i})_{i\in I}\in\Mc(t)=\oplus\Z f_{i}$,
let $[g]_{+}$ denote $([g_{i}]_{+})_{i}$. For any $m\in\N^{I_{\ufv}}$,
let $I^{m}$ denote the normalization of the twisted product $[\prod I_{_{k}}^{m_{k}}(t)]^{t}$
in $\LP(t)$, where the cluster variables $I_{k}(t)$ are given by
\eqref{eq:inj_cl_var}. Then the twisted product $X^{[\pr_{I_{\ufv}}g]_{+}}*I^{[-\pr_{I_{\ufv}}g]_{+}}$
has degree $g+u$ for some $u\in\Z^{I_{\fv}}$. We define the following
$g$-pointed function in $\LP(t)$:

\begin{eqnarray}
\Inj_{g}^{t} & := & [p_{g}*X(t)^{[\pr_{I_{\ufv}}g]_{+}}*I(t)^{[-\pr_{I_{\ufv}}g]_{+}}]^{t}\label{eq:inj_function}
\end{eqnarray}
where $p_{g}=X(t)^{-u}\in\cRing$.

For the rest of the paper, we choose the set of distinguished functions
in $\hLP(t)$ to be

\begin{eqnarray*}
\Inj^{t} & = & \{\Inj_{g}^{t}|g\in\Mc(t)\}.
\end{eqnarray*}

\begin{Def}[Weakly triangular basis {\cite[Definition 6.3.1]{qin2017triangular}}]

Let $\can^{t}$ denote the set of triangular functions $\{\can_{g}|g\in\Mc(t)\}$
with respect to the set of distinguished functions $\Inj^{t}$. If
$\can^{t}$ is a $\kk$-basis of $\midAlg(t)\subset\qUpClAlg(t)$,
we call it the weakly triangular basis of $\midAlg(t)$ with respect
to the seed $t$.

\end{Def}

By Lemma \ref{lem:tri_func_unique}, the weakly triangular basis is
unique if it exists. Notice that a weakly triangular basis contains
the quantum cluster monomials in $t,t[1]$. If it further satisfies
condition (4) in Definition \ref{def:triangular_basis}, it becomes
the triangular basis.

\begin{Lem}[{\cite[Lemma 6.3.2]{qin2017triangular}}]

If $\midAlg$ has a triangular basis $\can$ for seed $t$, then $\can$
is the weakly triangular basis $\can^{t}$. In particular, the triangular
basis is unique.

\end{Lem}

\cite[Lemma 6.2.1]{qin2017triangular} and the correction technique
(Theorem \ref{thm:correction_technique}) implies the following result.

\begin{Lem}[{\cite[Lemma 6.3.4]{qin2017triangular}}]\label{lem:similar_triangular_basis}\reviseStart

Let there be given two similar quantum seeds $t$ and $t'$ (not necessarily
related by mutations). Assume that a subalgebra $\midAlg(t)\subset\qUpClAlg(t)$
possesses the weakly triangular basis (resp. triangular basis) $\can^{t}$
with respect to the seed $t$. Let $\can^{t'}$ denote the set of
all pointed functions in $\LP(t')$ similar to the elements of $\can^{t}$.
Let $\midAlg(t')$ denote the $\kk$-module spanned by $\can^{t'}$.
Then $\midAlg(t')$ is a $\kk$-subalgebra of $\qUpClAlg(t')$ and
$\can^{t'}$ is its weakly triangular basis (resp. triangular basis)
with respect to $t'$.\reviseEnd

\end{Lem}

We recall some useful properties of the set of distinguished functions
$\Inj^{t}$ and the (weakly) triangular basis $\can^{t}$.

\begin{Lem}[Substitution {\cite[Lemma 6.2.4]{qin2017triangular}} ]\label{lem:substitution}

Assume that $[p*X(t)^{d_{X}}*I(t)^{d_{I}}]^{t}$ is $(\prec_{t},\mm)$-unitriangular
to $\Inj^{t}$ for any $p\in\cRing$, $d_{X},d_{I}\in\N^{I_{\ufv}}$.
If a pointed function $Z\in\hLP(t)$ is $(\prec_{t,}\mm)$-unitriangular
to $\Inj^{t}$, so is the product $[p*X(t)^{d_{X}}*Z*I(t)^{d_{I}}]^{t}$.

\end{Lem}

\begin{Lem}[{\cite[Lemma 5.4.2]{qin2017triangular}}]\label{lem:inj_set_adjacent_compatible}

For any $k\in I_{\ufv}$, the set of distinguished functions $\Inj^{t}$
is compatibly pointed at $t$ and $\mu_{k}t$.

\end{Lem}

Finally, the following result gives a sufficient condition for strengthening
a weakly triangular basis into a triangular basis.

\begin{Prop}[\cite{qin2017triangular}]\label{prop:lift_to_triangular}

Let there be given $k\in I_{\ufv}$ and seeds $t=\mu_{k}t'$. If the
triangular basis $\can^{t}$ for seed $t$ and the weakly triangular
basis $\can^{t'}$ for $t'$ both exist, such that they are compatible,
i.e., $\mu_{k}^{*}\can_{g}^{t}=\can_{\phi_{t',t}g}^{t'}$ $\forall g\in\Mc(t)$,
then $\can^{t'}$ is the triangular basis for seed $t'$.

\end{Prop}

\begin{proof}

The claim is proved as Claim (ii) in the proof of \cite[Proposition 6.6.3]{qin2017triangular}.
Its proof is based on basic properties of triangular bases, as well
as the observation that the exchange relation of quantum cluster variables
gives a $(\prec_{t'},\mm)$-triangular decomposition of $X_{k}(t')*X_{k}(t)$.

\end{proof}

\subsection{Triangular bases and twist automorphisms}

Let there be given similar quantum seeds $t,t'=\seq t\in\Delta^{+}$.
Assume that $\tw=\tw^{t}=\seq^{*}\var^{t}$ is a twist automorphism
on $\cF(t)$ passing through $t'$ where $\var^{t}:\cF(t)\simeq\cF(t')$
is the variation map.

\begin{Prop}\label{prop:twist_map_similar_seed}

Assume that $\can^{t}$ and $\can^{t'}$ are the weakly triangular
basis with respect to the seeds $t$ and $t'$ respectively. If $\can^{t'}=(\seq^{*})^{-1}\can^{t}$,
then $\can^{t}$ is permuted by the twist automorphism $\tw$.

\end{Prop}

\begin{proof}

Notice that $\var^{t}\can^{t}$ consists of all elements similar to
those of $\can^{t}$. Therefore, $\var^{t}\can^{t}$ is the weakly
triangular basis with respect to $t'$ by Lemma \ref{lem:similar_triangular_basis}.
We deduce from the uniqueness of triangular functions that $\var^{t}\can^{t}=\can^{t'}=(\seq^{*})^{-1}\can^{t}$
and, equivalently, $\tw\can^{t}=\can^{t}.$

\end{proof}

As a consequence of Proposition \ref{prop:twist_map_similar_seed},
we obtain the following nice property of common triangular bases.

\begin{Prop}\label{prop:tri_basis_permute_by_twist}

If $\midAlg$ possesses the common triangular basis $\can$, then
$\can$ is permuted by any twist automorphisms.

\end{Prop}

\subsection{Adjacent seeds\label{subsec:Adjacent-seeds}}

We prove some important statements concerning adjacent seeds, which
will be useful for studying the existence of common triangular bases. 

For any $t'=\mu_{k}t$, $j\neq k\in I_{\ufv}$, recall that we have
$\mu_{k}^{*}X_{j}(t')=X_{j}(t)$ and $\mu_{k}^{*}I_{j}(t')=I_{j}(t)$.

\begin{Def}[Admissibility]\label{def:admissibility}

Let there be given seeds $t'=\mu_{k}t$. A triangular basis $\can^{t}$
with respect to seed $t$ is said to be admissible in direction $k$,
if the quantum cluster monomials $\mu_{k}^{*}X_{k}(t')^{d}$, $\mu_{k}^{*}I_{k}(t')^{d}$,
$d\in\N$, are contained in $\can^{t}$.

\end{Def}

Let $t$ be a given injective-reachable seed. By the following Lemma,
a twist automorphism could reduce the burden to check the admissibility
condition by half.

\begin{Lem}\label{lem:reduce_admissible}

Assume that we have a DT-type twist automorphism $\tw$. If $\tw\can^{t}=\can^{t}$,
then $\can^{t}$ is admissible in direction $k$ if and only if it
contains $X_{k}(t')^{d}$, $d\in\N$.

\end{Lem}

\begin{proof}

If $\can^{t}$ contains $X_{k}(t')^{d}$ for some $d\in\N$, then
$\can^{t}=\tw\can^{t}$ contains $\tw X_{k}(t')^{d}$. By Lemma \ref{lem:twist_injective_variable},
there exists some frozen factor $p_{k}$ such that $\tw X_{k}(t')^{d}=p_{k}^{d}\cdot I_{k}(t')^{d}$.
By \cite[Lemma 6.2.1]{qin2017triangular}, $\can^{t}$ contains $I_{k}(t')^{d}$
as well.

\end{proof}

The following crucial result tells us that the admissibility condition
implies the existence of the compatible triangular basis for an adjacent
seed. Unlike previous works, no positivity assumption on the basis
is imposed. 

\begin{Prop}[Adjacent compatibility]\label{prop:admissible_nearby_seed}

Let there be given adjacent seeds $t'=\mu_{k}t$ related by one mutation,
$k\in I_{\ufv}$, and the triangular basis $\can^{t}$ with respect
to the seed $t$. If $\can^{t}$ is admissible in direction $k$,
then $\can^{t'}:=(\mu_{k}^{*})^{-1}\can^{t}$ is the triangular basis
with respect to $t'$. Moreover, $\can^{t}$ and $\can^{t'}$ are
compatible.

\end{Prop}

\begin{proof}

For any $g\in\Mc(t)$, denote $g'=\phi_{t',t}g\in\Mc(t')$ as before. 

For any $g'$, the $g'$-pointed function $\Inj_{g'}^{t'}$ in $\LP(t')$
takes the form

\begin{eqnarray*}
\Inj_{g'}^{t'}=\begin{cases}
[p_{g'}*X(t')^{d_{X}}*X_{k}(t')^{g'_{k}}*I(t')^{d_{I}}]^{t'} & g'_{k}\geq0\\{}
[p_{g'}*X(t')^{d_{X}}*I_{k}(t')^{-g'_{k}}*I(t')^{d_{I}}]^{t'} & g'_{k}\leq0
\end{cases}
\end{eqnarray*}
where $p_{g'}\in\cRing$ and $d_{X},d_{I}\in\N^{I_{\ufv}\backslash\{k\}}$.
Moreover, $\mu_{k}^{*}\Inj_{g'}^{t'}$ is $g$-pointed in $\LP(t)$
by Lemma \ref{lem:inj_set_adjacent_compatible}. Then it takes the
following form in $\LP(t)$:

\begin{eqnarray*}
\mu_{k}^{*}\Inj_{g'}^{t'}=\begin{cases}
[p_{g'}*X(t)^{d_{X}}*\mu_{k}^{*}X_{k}(t')^{g'_{k}}*I(t)^{d_{I}}]^{t} & g'_{k}\geq0\\{}
[p_{g'}*X(t)^{d_{X}}*\mu_{k}^{*}I_{k}(t')^{-g'_{k}}*I(t)^{d_{I}}]^{t} & g'_{k}\leq0
\end{cases}.
\end{eqnarray*}

Since $\can^{t}$ is admissible in direction $k$, it contains the
quantum cluster monomials $\mu_{k}^{*}X(t')^{d}$, $\mu_{k}^{*}I_{k}(t')^{d}$
viewed as elements in $\LP(t)$. Then these quantum cluster monomials
are $(\prec_{t},\mm)$-unitriangular to $\Inj^{t}$. By the Substitution
Lemma \ref{lem:substitution}, $\mu_{k}^{*}\Inj_{g'}^{t'}$ is $(\prec_{t},\mm)$-unitriangular
to $\Inj^{t}$, and thus $(\prec_{t},\mm)$-unitriangular to $\can^{t}$.
Then, by the inverse transition (Lemma \ref{lem:inverse_triangular_transition}),
$\can^{t}$ is $(\prec_{t},\mm)$-unitriangular to the $\Mc(t)$-pointed
set $\mu_{k}^{*}\Inj^{t'}=\{\mu_{k}^{*}\Inj_{g'}^{t'}|g'\in\Mc(t')\}$.
More precisely, for any $g\in\Mc(t)$, $\can_{g}^{t}$ has the following
$\prec_{t}$-decomposition in $\hLP(t)$:

\begin{eqnarray*}
Z:=\can_{g}^{t} & = & \mu_{k}^{*}\Inj_{g'}^{t'}+\sum_{\eta\prec_{t}g}b_{g,\eta}\mu_{k}^{*}\Inj_{\eta'}^{t'}
\end{eqnarray*}
where the coefficients $b_{g,\eta}\in\mm$, $\eta'=\phi_{t',t}\eta$.

Because the set of distinguished functions $\Inj^{t'}$ is compatibly
pointed at $t$ and $t'$ (Lemma \ref{lem:inj_set_adjacent_compatible}),
by the mutation-invariance property of dominance order decomposition
(Proposition \ref{prop:invariance_dominance_decomposition}), the
above $\prec_{t}$-decomposition of $\can_{g}^{t}$ in terms of $\Inj^{t'}$
also gives the $\prec_{t'}$-triangular decomposition in $\hLP(t')$:
\begin{eqnarray}
Z':=(\mu_{k}^{*})^{-1}\can_{g}^{t} & = & \Inj_{g'}^{t'}+\sum_{\eta\prec_{t}g}b_{g,\eta}\Inj_{\eta'}^{t'}.\label{eq:decomposition_adjacent}
\end{eqnarray}

We claim that \eqref{eq:decomposition_adjacent} is $(\prec_{t'},\mm)$-unitriangular
with the leading term $\Inj_{g'}^{t'}$, i.e., $(\mu_{k}^{*})^{-1}\can_{g}^{t}$
is $g'$-pointed. For the proof, first notice that $(\mu_{k}^{*})^{-1}$
commutes with the bar involution. It follows that $Z'$ is bar-invariant.
Consider the support $\supp_{\Inj^{t'}}(Z'):=\{g'\}\cup\{\eta'|b_{g,\eta}\neq0\}$.
Denote $b_{g,g}=1$. Any $\prec_{t'}$-maximal elements $m'\in\supp_{\Inj^{t'}}(Z')$
contributes a Laurent monomial $b_{g,m}X(t')^{m'}$ with a $\prec_{t'}-$maximal
degree in the Laurent expansion of $Z'\in\LP(t')$. Notice that $b_{g,g}=1$
and $b_{g,\eta}\in\mm$ for $\eta\prec_{t}g$. Then, the bar-invariance
of $Z'$ implies that $g'$ is the unique $\prec_{t'}$-maximal element
in the $\supp_{\Inj^{t'}}Z$. Our claim has been verified.

By the above claim, $\can^{t'}:=(\mu_{k}^{*})^{-1}\can^{t}$ is compatible
with $\can^{t}$, and it is $(\prec_{t'},\mm)$-unitriangular to $\Inj^{t'}$.
It follows from definition that it is the weakly triangular basis
for $t$. By Proposition \ref{prop:lift_to_triangular}, $\can^{t'}$
is further the triangular basis for the seed $t'$.

\end{proof}

We also need to verify the admissibility condition, provided some
compatibility condition (tropical properties). Let there be given
any $k\in I_{\ufv}$ and $d\in\N.$ Denote $t'=\mu_{k}t$ and $t'[1]=\mu_{\sigma k}(t[1])$.

\begin{Lem}\label{lem:support_to_admissible}

Let there be given $g=d\deg^{t}\mu_{k}^{*}X_{k}(t')$ and $g$-pointed
element $Z\in\LP(t)\cap\mu_{k}^{*}\LP(t')$. If $Z$ is bipointed
and its support dimension is the same as that of $\mu_{k}^{*}X_{k}(t')^{d}$,
then $Z=\mu_{k}^{*}X_{k}(t')^{d}$.

\end{Lem}

\begin{proof}

We can view $Z$ and $\mu_{k}^{*}X_{k}(t')^{d}$ as elements of the
(type $A_{1}$) quantum cluster algebra $\LP(t)\cap\mu_{k}^{*}\LP(t')$,
where $k$ is viewed as the only unfrozen vertex. Then $Z$ and $\mu_{k}^{*}X(t')^{d}$
are compatibly pointed at the set of all seeds $\{t,t'\}$, and we
have $Z=\mu_{k}^{*}X_{k}(t')^{d}$ by Proposition \ref{prop:support_compatible_equivalent}
and \cite[Lemma 3.4.11]{qin2019bases}.

Alternatively, let us give a fundamental proof. Notice that $Z$ takes
the form $X^{g}*(1+\sum_{0<s<d}b_{s}Y_{k}^{s}+v_{k}^{d}Y_{k}^{d})$
by assumption, $b_{s}\in\kk$. Denote $Y_{k}(t')=Y_{k}'$ and $X_{i}(t')=X_{i}'$
for simplicity. We have $\mu_{k}^{*}Y_{k}'=Y_{k}^{-1}\in\LP(t)$ and
$(\mu_{k}^{*})^{-1}Y_{k}=(Y_{k}')^{-1}\in\LP(t')$. Notice that $(\mu_{k}^{*})^{-1}X^{g}$
is a pointed function in $\hLP(t')$ whose $F$-function is a formal
series in $Y_{k}'$. We compute that 
\begin{eqnarray*}
Z' & := & (\mu_{k}^{*})^{-1}Z\\
 & = & (\mu_{k}^{*})^{-1}X^{g}*(1+\sum b_{s}(Y_{k}')^{-s}+v_{k}^{d}(Y_{k}')^{-d})\\
 & = & X'^{g'}*(\sum_{r>0}^{R}c_{r}(Y_{k}')^{r}+1)
\end{eqnarray*}
where $c_{r}\in\kk$, $R\in\N$, $g'=\deg^{t'}((\mu_{k}^{*})^{-1}X^{g}*(Y_{k}'){}^{-d})$.
By the same computation, $(\mu_{k}^{*})^{-1}(\mu_{k}^{*}X_{k}(t')^{d})=X_{k}(t')^{d}$
also has this degree. Consequently, $X'^{g'}=(X_{k}')^{d}$.

Now consider $\mu_{k}^{*}Z'\in\LP(t)$. Similar to the above computation,
its degree is given by $\deg^{t}(\mu_{k}^{*}(X_{k}')^{d}*Y_{k}{}^{R})$.
Since $\deg^{t}\mu_{k}^{*}Z'=\deg^{t}Z=\deg^{t}\mu_{k}^{*}(X_{k}')^{d}$,
we must have $R=0$. Consequently, $Z'=(X_{k}')^{d}$.

\end{proof}

\begin{Prop}[Admissiblity by compatibility]\label{prop:compatible_to_admissible}

(1) Take $g=d\deg^{t}\mu_{k}^{*}X_{k}(t')$. Let there be given $g$-pointed
element $Z\in\LP(t)\cap\seq_{t[-1],t}^{*}\LP(t[-1])$. If $Z$ is
compatibly pointed at $t,t[-1]$, then $Z=\mu_{k}^{*}X_{k}(t')^{d}$.

(2) Take $g=d\deg^{t}\seq_{k}^{*}I_{k}(t')$. Let there be given $g$-pointed
element $Z\in\LP(t)\cap\seq_{t[1],t}^{*}\LP(t[1])$. If $Z$ is compatibly
pointed at $t[1],t$, then $Z=\seq_{k}^{*}I_{k}(t')^{d}$.

\end{Prop}

\begin{proof}

Denote $t'=\mu_{k}t=\seq_{t',t}t$. Recall that $\seq_{t_{2},t_{3}}^{*}\seq_{t_{1},t_{2}}^{*}=\seq_{t_{1},t_{3}}^{*}$
for any seeds $t_{1},t_{2},t_{3}\in\Delta^{+}$ (Lemma \ref{lem:composition_mutation_maps}).

(1) Since $Z$ is compatibly pointed at $t,t[-1]$, by Proposition
\ref{prop:support_compatible_equivalent}, $Z$ has support dimension
$\suppDim g$. The same holds for the cluster monomial $\mu_{k}^{*}X_{k}(t')^{d}$.
We deduce that $\can_{g}^{t}$ and $X_{k}(t')^{d}$ have the same
support dimension. Lemma \ref{lem:support_to_admissible} implies
that $\can_{g}^{t}=\mu_{k}^{*}X_{k}(t')^{d}$.

(2) Notice that we have 
\begin{align*}
\seq_{t,t[1]}^{*}\seq_{k}^{*}I_{k}(t') & =\seq_{t,t[1]}^{*}\seq_{t',t}^{*}I_{k}(t')\\
 & =\seq_{t,t[1]}^{*}\seq_{t',t}^{*}\seq_{t'[1],t'}^{*}X_{\sigma k}(t'[1])\\
 & =\seq_{t'[1],t[1]}^{*}X_{\sigma k}(t'[1])\\
 & =\mu_{\sigma k}^{*}X_{\sigma k}(t'[1]).
\end{align*}
Since $Z$ is compatibly pointed at $t[1],t$, we obtain that $Z':=\seq_{t,t[1]}^{*}Z$
is compatibly pointed at $t[1],t$, such that $\deg^{t[1]}Z'=\phi_{t[1],t}\deg^{t}Z=\phi_{t[1],t}\deg^{t}\seq_{k}^{*}I_{k}(t')^{d}=\deg^{t[1]}\seq_{t,t[1]}^{*}\seq_{k}^{*}I_{k}(t')^{d}=\deg^{t[1]}\mu_{\sigma k}^{*}X_{\sigma k}(t'[1])^{d}$.
It follows from (1) that $Z'=\mu_{\sigma k}^{*}X_{\sigma k}(t'[1])^{d}$.
Therefore, $Z=(\seq_{t,t[1]}^{*})^{-1}Z'=(\seq_{t,t[1]}^{*})^{-1}\seq_{t,t[1]}^{*}\seq_{k}^{*}I_{k}(t')^{d}=\seq_{k}^{*}I_{k}(t')^{d}$.

\end{proof}

\subsection{Criteria for the existence of the common triangular bases\label{subsec:Criteria-existence}}

In this section, we give several criteria that guarantee the existence
of the common triangular bases. Let there be given seeds $t[1]=\seq t=\seq_{t[1],t}t$
as before. 

Define seeds $t[r]=(\sigma^{r-1}\seq)t[r-1]$ recursively for $r\in\Z$.
Then we have $t[r]=t[r-1][1]$, and the set $\{t[r],r\in\Z\}$ is
called an injective-reachable chain, see \cite{qin2017triangular}
for more details. Notice that we have $(\mu_{k}t)[r]=\mu_{\sigma^{r}k}(t[r])$,
$\seq_{t[r+1],t[r]}=\sigma^{r}\seq_{t[1],t}$. Denote $\seq_{t[r],t[r+1]}=\seq_{t[r+1],t[r]}^{-1}$.

\begin{Prop}\label{prop:compatible_to_chain_bases}

Assume that a given subalgebra $\midAlg(t)\subset\qUpClAlg(t)$ possesses
the triangular bases $\can^{t}$. The following statements are true.

(1) The triangular bases $\can^{t[r]}$ for the seeds $t[r]$ exist,
$r\in\Z$, and they consist of similar pointed functions.

(2) If $\can^{t[1]}$ is compatibly pointed at $t,t[1]$, then $\can^{t[r]}$
is compatibly pointed at $t[r]$, $t[r-1]$, for any $r\in\Z$. 

(3) If $\seq_{t[1],t}^{*}\can^{t[1]}$ is the triangular basis $\can^{t}$,
then $\seq_{t[r],t[r-1]}^{*}\can^{t[r]}$ is the triangular basis
$\can^{t[r-1]}$, $r\in\Z$.

\end{Prop}

\begin{Prop}\label{prop:new_chain_bases}

Assume that a given subalgebra $\midAlg\subset\qUpClAlg$ possesses
compatible triangular bases at $\{t[r],r\in\Z\}$. Then, for any $k\in I_{\ufv}$,
it possesses compatible triangular bases at $\{(\mu_{k}t)[r],r\in\Z\}\cup\{t[r],r\in\Z\}$.

\end{Prop}

We postpone the proofs of Propositions \ref{prop:compatible_to_chain_bases},
\ref{prop:new_chain_bases} to the end of section.

\begin{Thm}[Existence by compatibility]\label{thm:existence_compatible}

Assume that a given subalgebra $\midAlg(t)\subset\qUpClAlg(t)$ possesses
the triangular bases $\can^{t}$. If $\can^{t}$ is compatibly pointed
at $t,t[1]$ and, moreover, $(\seq_{t[1],t}^{*})^{-1}\can^{t}$ is
the triangular basis $\can^{t[1]}$ for the seed $t[1]$, then $\can^{t}$
is the common triangular basis.

\end{Thm}

\begin{proof} 

Proposition \ref{prop:compatible_to_chain_bases} implies that the
triangular bases $\can^{t[r]}$, $r\in\Z$, exist and are compatible.
Using Proposition \ref{prop:new_chain_bases} recursively for adjacent
seeds starting from $\{t[r],r\in\Z\}$, we obtain that $\can^{t}$
is the common triangular basis.

\end{proof}

Recall that if the common triangular basis exists, then it is permuted
by all twist automorphisms (Proposition \ref{prop:tri_basis_permute_by_twist}).
As an inverse result, we have the following existence theorem which
implies the main theorem (Theorem \ref{thm:main_intro}). Denote $\seq_{t[1],t}=\seq=\mu_{k_{l}}\cdots\mu_{k_{2}}\mu_{k_{1}}$.
Denote $t_{s}=\seq_{\leq s}t=\mu_{k_{s}}\cdots\mu_{k_{2}}\mu_{k_{1}}t$
for $1\leq s\leq l$. 

\begin{Thm}[Existence by twist automorphisms]\label{thm:existence_mutation_sequence}

Assume that a given subalgebra $\midAlg(t)\subset\qUpClAlg(t)$ has
the triangular basis $\can^{t}$ with respect to the seed $t$. Further
assume that $\can^{t}$ satisfies the following properties:

\begin{enumerate}

\item There exists a twist automorphism $\tw$ of Donaldson-Thomas
type on $\cF(t)$ such that $\tw$ permutes $\can^{t}$.

\item The quantum cluster monomials $\seq_{\leq s}^{*}X_{k_{s}}(t_{s})^{d}$,
$d\in\N$, obtained along the mutation sequence $\seq_{t[1],t}$ starting
from $t$ are contained in $\can^{t}$. 

\end{enumerate}

Then $\can^{t}$ is the common triangular basis. In particular, it
contains all quantum cluster monomials.

\end{Thm}

\begin{proof}

Because $\can^{t}$ is permuted by $\tw$ and contains all cluster
monomials appearing along the mutation sequence $\seq$, we deduce
from Lemma \ref{lem:reduce_admissible} and Proposition \ref{prop:admissible_nearby_seed}
that it gives rise to the triangular basis $\can^{t_{s}}$ for the
seeds $t_{s}$, which are compatible with $\can^{t}$. In particular,
$\can^{t}$ and $\can^{t[1]}$ are compatible. We deduce the claim
from Theorem \ref{thm:existence_compatible} .

Let us give an alternative proof which does not depend on Theorem
\ref{thm:existence_compatible} nor Proposition \ref{prop:compatible_to_chain_bases}.
For any $r\in\Z$, by Lemma \ref{lem:similar_triangular_basis} for
similar seeds $t,t[r]$, we can construct the triangular bases $\can^{t[r]}$
for seed $t[r]$ as the set of elements similar to the elements in
$\can^{t}$. Then it contains the quantum cluster monomials appearing
along the sequence $\sigma^{r}\seq$ from $t[r]$ to $t[r+1]$ by
Lemma \ref{lem:similar_cluster_mutation}. It follows that $\can^{t[r]}$
and $\can^{t[r+1]}$ are compatible by Lemma \ref{lem:reduce_admissible}
and Proposition \ref{prop:admissible_nearby_seed}. Therefore, we
obtain compatible triangular bases at $\{t[r],r\in\Z\}$. The claim
follows by using Proposition \ref{prop:new_chain_bases} repeatedly
for adjacent seeds.

\end{proof}

\begin{proof}[Proof of Proposition \ref{prop:new_chain_bases}]

Denote $t'=\mu_{k}t$. Recall that $\mu_{\sigma^{r}k}t[r]=t'[r]$,
$r\in\Z$. Since $\can^{t[r]}$ is compatibly pointed at $t[r],t[r-1]$,
it contains $\mu_{\sigma^{r}k}^{*}X_{\sigma^{r}k}(t'[r])$ by Proposition
\ref{prop:compatible_to_admissible}(1). Similarly, $\can^{t[r+1]}$
contains $Z:=\mu_{\sigma^{r+1}k}^{*}X_{\sigma^{r+1}k}(t'[r+1])$.
In addition, since $\can^{t[r+1]}$ and $\can^{t[r]}$ are compatible,
$\can^{t[r]}=\seq_{t[r+1],t[r]}^{*}\can^{t[r+1]}$ also contains the
$Z':=\mu_{t[r+1],t[r]}^{*}Z$. Lemma \ref{lem:composition_mutation_maps}
implies that we have
\begin{align*}
Z' & =\mu_{t[r+1],t[r]}^{*}\mu_{t'[r+1],t[r+1]}^{*}X_{\sigma^{r+1}k}(t'[r+1])\\
 & =\mu_{t'[r+1],t[r]}^{*}X_{\sigma^{r+1}k}(t'[r+1])\\
 & =\mu_{t'[r],t[r]}^{*}\mu_{t'[r+1],t'[r]}^{*}X_{\sigma^{r+1}k}(t'[r+1])\\
 & =\mu_{t'[r],t[r]}^{*}I_{\sigma^{r}k}(t'[r])\\
 & =\mu_{\sigma^{r}k}^{*}I_{\sigma^{r}k}(t'[r]).
\end{align*}
Therefore, $\can^{t[r]}$ is admissible in direction $\sigma^{r}k$.
Then Proposition \ref{prop:admissible_nearby_seed} implies that $\can^{t[r]}$
gives rise to a compatible triangular basis $\can^{t'[r]}$ in the
adjacent seed $t'[r]$. The claim follows.

\end{proof}

\begin{proof}[Proof of Proposition \ref{prop:compatible_to_chain_bases}]

(1) We define $\can^{t[r]}$ to be the set of pointed functions in
$\LP(t[r])$ similar to the elements of $\can^{t}$. The claim follows
from Lemma \ref{lem:similar_triangular_basis}.

(2) By Proposition \ref{prop:support_compatible_equivalent}, $\can_{g}^{t[1]}$
has the support dimension $\suppDim g\in\N^{I_{\ufv}}\simeq N_{\ufv}^{\geq0}(t[1])$
for any $g\in\Z^{I}\simeq\Mc(t[1])$. By (1), $\can^{t[r]}$ and $\can^{t[1]}$
have similar elements. Let us relabel the set of vertices $I_{\ufv}$
by $(\sigma^{r-1})^{-1}$ when working with the seed $t[r]$, such
that $b_{ij}(t[1])=b_{ij}(t[r])$ (or, equivalently, we create a new
seed $\sigma^{1-r}t[r]$ such that $e_{k}(\sigma^{1-r}t[r]):=e_{\sigma^{r-1}k}(t[r])$).
Then we have $\suppDim\can_{g}^{t[r]}=\suppDim\can_{g}^{t[1]}=\suppDim g$
as elements in $\Z^{I_{\ufv}}$.

Since the support dimension $\suppDim g$ associated to $g\in\Z^{I}$
only depends on the principal $B$-matrix (see Remark \ref{rem:suppdim_theta_func}),
$\suppDim g$ for $t[1]$ and $t[r]$ are the same. Using Proposition
\ref{prop:support_compatible_equivalent} again, we deduce that $\can_{g}^{t[r]}$
is compatibly pointed at $t[r],t[r-1]$ for any $r\in\Z$.

(3) For any $g\in\Z^{I}$, $\can_{\sigma^{r-1}g}^{t[r]}$ and $\can_{g}^{t[1]}$
are similar by (1). If follows that $\seq_{t[r],t[r-1]}^{*}\can_{\sigma^{r-1}g}^{t[r]}$
and $\seq_{t[1],t}^{*}\can_{g}^{t[1]}$ are similar by Proposition
\ref{prop:mutation_similar_elements}. Since $\seq_{t[1],t}^{*}\can_{g}^{t[1]}\in\can^{t}$,
its similar function $\seq_{t[r],t[r-1]}^{*}\can_{\sigma^{r-1}g}^{t[r]}$
must belong to $\can^{t[r-1]}$ by (1). By the same argument, we have
$(\seq_{t[r],t[r-1]}^{*})^{-1}\can_{g'}^{t[r-1]}\in\can^{t[r]}$ for
any $g'\in\Z^{I}$. The claim follows.

\end{proof}

\section{Prerequisite for quantum unipotent subgroup\label{sec:Prerequisite-quantum-groups}}

We collect useful notions and results for quantum groups, mostly following
\cite{kimura2017twist}\cite{Kimura10}. Notice that different conventions
have been used in literature. We use an easy model to explain ours,
see Examples \ref{eg:sl3_type}, \ref{eg:sl3_canonical_basis}. We
choose the field $\kk=\Q(q)$ with a formal parameter $q$. Notice
that $\kk$ has a bar involution such that $\overline{q}=q^{-1}$.

\subsection{Quantized enveloping algebras}

\subsubsection*{Root datum}

A root datum consists of the following.

\begin{enumerate}

\item A finite set $[1,r]:=\{1,2,\ldots,r\}$.

\item A finite dimensional $\Q$-vector space $\frh$, its dual space
$\frh^{*}$, and the canonical pairing $\langle\ ,\ \rangle$.

\item A lattice $P\subset\frh^{*}$ called the weight lattice.

\item The dual lattice $P^{\vee}=\Hom_{\Z}(P,\Z)\subset\frh$ called
the coweight lattice.

\item A linearly independent subset $\{\alpha_{i}|i\in[1,r]\}\subset P$,
called the set of simple roots.

\item A linearly independent subset $\{\alpha_{i}^{\vee}|i\in[1,r]\}\subset P^{\vee}$,
called the set of simple coroots.

\item A $\Q$-valued symmetric $\Z$-bilinear form $(\ ,\ )$ on
$P$, such that

\begin{itemize}

\item $\sym_{i}:=\frac{1}{2}(\alpha_{i},\alpha_{i})\in\Z_{>0}$,

\item $\langle\alpha_{i}^{\vee},\mu\rangle=\frac{2(\alpha_{i},\mu)}{(\alpha_{i},\alpha_{i})}$
for any $\mu\in P$,

\item Denote $C_{ij}=\langle\alpha_{i}^{\vee},\alpha_{j}\rangle$.
Then $C=(C_{ij})_{1\leq i,j\leq r}$ is a generalized symmetrizable
Cartan matrix, i.e., $C_{ii}=2$ and, for any $i\neq j$, we have
$C_{ij}\in\Z_{\leq0}$, $C_{ij}=0$ if and only if $C_{ji}=0$.

\end{itemize}

\end{enumerate}

Notice that we have $\sym_{i}C_{ij}=\sym_{j}C_{ji}$. We define the
root lattice $Q=\oplus_{i}\Z\alpha_{i}\subset P$, the coroot lattice
$Q^{\vee}=\oplus_{i}\Z\alpha_{i}^{\vee}$. Define $Q_{\pm}=\pm\oplus_{i}\N\alpha_{i}$,
the semigroup of dominant weights $P_{+}=\{\lambda\in P|\langle\alpha_{i}^{\vee},\lambda\rangle\geq0,\forall i\}$.

Assume that we have fundamental weights $\varpi_{i}\in P_{+}$ such
that $\langle\alpha_{i}^{\vee},\varpi_{j}\rangle=\delta_{ij}$, $\forall i,j$.
Denote $\rho=\sum_{i\in[1,r]}\varpi_{i}$.

In addition, for any $i\in[1,r]$, we define the simple reflection
$s_{i}$ on $\frh^{*}$ such that, for any $\mu\in\frh^{*}$, 
\begin{eqnarray*}
s_{i}(\mu) & = & \mu-\langle\alpha_{i}^{\vee},\mu\rangle\alpha_{i}.
\end{eqnarray*}
The group $W$ generated by the simple reflections is called the Weyl
group. Let there be given any $w\in W$. Given a word $\ui=i_{1}\cdots i_{l}$
with symbols $i_{j}\in[1,r]$, $1\leq j\leq l$, $l\in\N$, we denote
$s_{\ui}=s_{i_{1}}\cdots s_{i_{l}}$, and the length $|\ui|=l$. Recall
that, if $l=\min\{|\ui'||s_{\ui'}=w\}$, we call the word $\ui$ a
reduced word for $w$, the product $s_{\ui}$ a reduced expression
of $w$, and $l(w)=l$ the length of $w$. Choosing a reduced word
$\ui$ for $w$, we define the support of $w$ in $[1,r]$ to be

\[
\supp w:=\supp s_{\ui}:=\supp\ui:=\{i|i\in\ui\}.
\]
One can check that $\supp w$ does not depend on the choice of the
reduced word.

Similarly, as in \cite[2.2.6]{Lus:intro}, we define reflections $s_{i}$
on $\frh$ such that, for any $h\in\frh$, we have

\begin{eqnarray*}
s_{i}(h) & = & h-\langle\alpha_{i},h\rangle\alpha_{i}^{\vee}.
\end{eqnarray*}
Then $\langle s_{i}(h),\mu\rangle=\langle h,s_{i}(\mu)\rangle$.

For any $k\in[1,l]$, define its successor and predecessor by 
\begin{align*}
s(k) & :=k[1]:=\min(\{j>k|i_{j}=i_{k}\}\cup\{+\infty\}),\\
p(k) & :=k[-1]:=\max(\{j<k|i_{j}=i_{k}\}\cup\{-\infty\}).
\end{align*}
Define $k[0]=k$. Recursively, define $k[d\pm1]=k[d][\pm1]$ for $d\in\Z$
if $k[d]\in[1,l]$. In addition, denote

\begin{align*}
k^{\max} & :=\max\{j\in[1,l]|i_{j}=i_{k}\},\\
k^{\min} & :=\min\{j\in[1,l]|i_{j}=i_{k}\}.
\end{align*}
For any $a\in\supp w$, we denote

\begin{align*}
^{\min}a & =\min\{k\in[1,l]|i_{k}=a\},\\
^{\max}a & =\max\{k\in[1,l]|i_{k}=a\}.
\end{align*}
Similarly, for $b\in\supp(i_{1}i_{2}\cdots i_{j})$, $j\in[1,l]$,
define

\begin{align*}
^{\min,\geq j}b & =\min(\{k\in[j,l]|i_{k}=b\}\cup\{+\infty\}),
\end{align*}

\begin{align*}
^{\max,\leq j}b & =\max(\{k\in[1,j]|i_{k}=b\}\cup\{-\infty\}).
\end{align*}
Finally, for $a\in[1,r]$, $j,k\in[1,l]$, denote the multiplicity:

\begin{align*}
m(a,[j,k]) & =|\{s\in[j,k]|i_{s}=a\}|,\\
m(a) & =m(a,[1,l]),\\
m_{k}^{+} & =m(i_{k},[k+1,l]),\\
m_{k}^{-} & =m(i_{k},[1,k-1]).
\end{align*}

\subsubsection*{Quantized enveloping algebras $\envAlg$}

Denote $q_{i}=q^{\sym_{i}}$, $[a]_{i}=[a]_{q_{i}}=\frac{q_{i}^{a}-q_{i}^{-a}}{q_{i}-q_{i}^{-1}}$
for $i\in[1,r]$, $a\in\N$. The quantized enveloping algebra $\envAlg$
is the non-commutative $\kk$-algebra generated by the generators
$E_{i}$, $F_{i}$, $K^{h}$, $i\in[1,r]$, $h\in P^{\vee}$ subject
to certain relations (see Section \ref{sec:Serre_relations}). Denote
$K_{i}=K^{\sym_{i}\alpha_{i}^{\vee}}$ and the divided power $F_{i}^{(k)}=\frac{F_{i}^{k}}{[k]_{i}!}$
and $E_{i}^{(k)}=\frac{E_{i}^{k}}{[k]_{i}}$.

Recall that we have a $\kk$-algebra triangular decomposition $\envAlg=\envAlg^{+}\otimes\envAlg^{0}\otimes\envAlg^{-}=\langle E_{i}\rangle_{i\in[1,r]}\otimes\langle K^{h}\rangle_{h\in P^{\vee}}\otimes\langle F_{i}\rangle_{i\in[1,r]}$.
We have a natural $Q$-grading $\wt(\ )$ on $\envAlg$ such that
$\wt E_{i}=\alpha_{i}$, $\wt F_{i}=-\alpha_{i}$, $\wt K^{h}=0$
respectively. For any weight $\gamma\in Q$, the homogeneous weight
$\gamma$ elements form a $\kk$-vector space

\begin{eqnarray*}
(\envAlg)_{\gamma} & = & \{u\in\envAlg|K^{h}uK^{-h}=q^{\langle h,\gamma\rangle}u,\forall h\in P^{\vee}\}.
\end{eqnarray*}

\subsubsection*{Automorphisms and bilinear forms}

\begin{Def}

Extend the bar involution $\overline{(\ )}$ on $\kk$ to an involution
on $\envAlg$ such that

\begin{eqnarray*}
\overline{q}=q^{-1},\qquad\overline{E_{i}} & = & E_{i},\qquad\overline{F_{i}}=F_{i},\qquad\overline{K^{h}}=K^{-h}.
\end{eqnarray*}

Define the $\kk$-algebra anti-involution $*$ on $\envAlg$ such
that

\begin{eqnarray*}
*(E_{i}) & = & E_{i},\qquad*(F_{i})=F_{i},\qquad*(K^{h})=K^{-h}.
\end{eqnarray*}

\end{Def}

Following the convention in \cite{kimura2017twist}, we consider Lusztig's
pairing $(\ ,\ )_{L}$ on $\envAlg^{-}$. It is a symmetric and non-degenerate
$\kk$-bilinear form. Notice that one could also work with Kashiwara's
bilinear form $(\ ,\ )_{K}$, and the results only differ by scalars
depending on the $Q$-grading, see \cite[Lemma 2.12]{Kimura10}.

\begin{Def}

Define the dual bar involution $\sigma$ on $\envAlg^{-}$ such that
$(\sigma(x),y)_{L}=\overline{(x,\overline{y})_{L}}$.

\end{Def}

\begin{Def}

Define the $\kk$-linear isomorphism $c_{\tw}:\envAlg^{-}\rightarrow\envAlg^{-}$
such that, for homogeneous $x$, we have

\begin{eqnarray*}
c_{\tw}(x) & = & q^{\frac{1}{2}(\wt x,\wt x)-(\wt x,\rho)}x.
\end{eqnarray*}

\end{Def}

For $\mu=\sum\mu_{i}\alpha_{i}\in Q$, define $\Tr(\mu)=\sum\mu_{i}\in\Z$.

\begin{Prop}[{\cite[Proposition 3.10]{kimura2017twist}}]

For a homogeneous element $x\in\envAlg^{-}$, we have

\begin{eqnarray*}
\sigma(x) & = & (-1)^{\Tr(\wt x)}c_{\tw}(\overline{(\ )}\circ*)(x).
\end{eqnarray*}
In particular, for homogeneous $x,y$, we have 
\begin{eqnarray*}
\sigma(xy) & = & q^{(\wt x,\wt y)}\sigma(y)\sigma(x).
\end{eqnarray*}

\end{Prop}

For example, we have $(F_{i},F_{i})_{L}=\frac{1}{1-q_{i}^{2}}$, $\sigma(F_{i})=-q_{i}^{2}F_{i}$,
for $i\in[1,r]$.

Define the twisted dual bar involution $\sigma'=c_{\tw}^{-1}\circ\sigma$
on $\envAlg^{-}$. Then we have $\sigma'(x)=(-1)^{\Tr(\wt x)}(\overline{(\ )}\circ*)(x)$
for homogeneous $x$. In particular, $\sigma(x)=x$ if and only if
$\sigma'(x)=c_{\tw}^{-1}(x)=q^{-\frac{1}{2}(\wt x,\wt x)+(\wt x,\rho)}x$.
It follows that $\sigma'(xy)=\sigma'(y)\sigma'(x)$.

In order to compare quantum groups with the quantum cluster algebras,
a rescaling will be needed. We often extend the base field from $\kk=\Q(q)$
to $\kk(q^{\Hf})=\Q(q^{\Hf})$. Denote $\envAlg_{\Q(q^{\Hf})}^{-}=\envAlg^{-}\otimes_{\Q(q)}\Q(q^{\Hf})$.
Define the $\Q(q^{\Hf})$-module endomorphism $\cor$ of $\envAlg_{\Q(q^{\Hf})}^{-}$
such that 
\begin{eqnarray*}
\cor x & = & q^{-\frac{1}{4}(\wt x,\wt x)+\Hf(\wt x,\rho)}x
\end{eqnarray*}
 for homogeneous $x$. Then $\sigma(x)=x$ if and only if $\sigma'(\cor x)=\cor x$.
We understand $\cor$ as a scalar correction.

\subsection{Quantum unipotent subgroups}

Take any $w\in W$. Choose a reduced expression $\ow=s_{\ui}=s_{i_{1}}\cdots s_{i_{l}}$
of $w$. We have the natural function $\ui(\ ):[1,l]\rightarrow[1,r]$
such that $\ui(j)=i_{j}$. For any $j<k\in[1,l]$, denote 

\begin{eqnarray*}
\ui[j,k] & = & i_{j}i_{j+1}\cdots i_{k}\\
\ow_{[j,k]} & = & s_{\ui[j,k]}\\
\ow_{\leq j} & = & s_{\ui[1,j]}\\
\ow_{\geq k} & = & s_{\ui[k,l]}
\end{eqnarray*}

Define the roots $\beta_{k}=\ow_{\leq k-1}\alpha_{i_{k}}$, $k\in[1,l]$.

Following the convention in \cite[Definition 3.23]{kimura2017twist}(see
$T_{i,1}^{''}$, $T'_{i,-1}$ in \cite[Section 37.1.3]{Lus:intro}\cite[Section 4.2.1]{Kimura10}\cite[Proposition 1.3.1]{saito1994pbw}),
for any $i\in[1,r]$, we define the $\kk$-algebra automorphism $T_{i}$
on $\envAlg$ (braid group action) such that

\begin{align*}
T_{i}(K^{h}) & =K^{s_{i}(h)}\\
T_{i}(E_{j}) & =\begin{cases}
-F_{i}K_{i} & i=j\\
\sum_{r+s=-C_{ij}}(-1)^{r}q_{i}^{-r}E_{i}^{(s)}E_{j}E_{i}^{(r)} & i\neq j
\end{cases}\\
T_{i}(F_{j}) & =\begin{cases}
-K_{i}^{-1}E_{i} & i=j\\
\sum_{r+s=-C_{ij}}(-1)^{r}q_{i}^{r}F_{i}^{(r)}F_{j}F_{i}^{(s)} & i\neq j
\end{cases}.
\end{align*}
Its inverse map is given by

\begin{align*}
T_{i}^{-1}(K^{h}) & =K^{s_{i}(h)}\\
T_{i}^{-1}(E_{j}) & =\begin{cases}
-K_{i}^{-1}F_{i} & i=j\\
\sum_{r+s=-C_{ij}}(-1)^{r}q_{i}^{-r}E_{i}^{(r)}E_{j}E_{i}^{(s)} & i\neq j
\end{cases}\\
T_{i}^{-1}(F_{j}) & =\begin{cases}
-E_{i}K_{i} & i=j\\
\sum_{r+s=-C_{ij}}(-1)^{r}q_{i}^{r}F_{i}^{(s)}F_{j}F_{i}^{(r)} & i\neq j
\end{cases}.
\end{align*}

Notice that $T_{i}((\envAlg)_{\gamma})=(\envAlg)_{s_{i}(\gamma)}$
for $\gamma\in Q$. Define $T_{\ow}=T_{i_{1}}T_{i_{2}}\cdots T_{i_{l}}$.
Then it is known that $T_{\ow}$ only depends on $w\in W$, which
we denote by $T_{w}$.

\begin{Prop}[{\cite[Section 37.2.4]{Lus:intro}}]\label{prop:change_braid_sign}

We have $*\circ T_{i}\circ*=T_{i}^{-1}$.

\end{Prop}

Choose any sign $e\in\{+1,-1\}$. We might omit $e$ when $e=+1$. 

For any $m\in\N$, let us define

\begin{eqnarray*}
F_{e}(m\beta_{k}) & = & T_{i_{1}}^{e}\cdots T_{i_{k-1}}^{e}(F_{i_{k}}^{(m)}).
\end{eqnarray*}
Notice that $\wt(F_{e}(m\beta_{k}))=-m\beta_{k}$.

For any multiplicity vector $\uc=(c_{1},\ldots,c_{l})\in\N^{[1,l]}$,
define the corresponding ordered product (\cite[Section 4.3]{Kimura10}):

\begin{eqnarray*}
F_{e}(\uc,\ow) & = & \begin{cases}
F_{e}(c_{1}\beta_{1})\cdots F_{e}(c_{l}\beta_{l}) & e=+1\\
F_{e}(c_{l}\beta_{l})\cdots F_{e}(c_{1}\beta_{1}) & e=-1
\end{cases}
\end{eqnarray*}

By Proposition \ref{prop:change_braid_sign}, we have $*F_{e}(\uc,\ow)=F_{-e}(\uc,\ow)$.

\begin{Prop}[\cite{Lus:intro}]

$\{F_{e}(\uc,\ow)|\uc\in\N^{[1,l]}\}$ forms a basis of a $\kk$-subspace
of $\envAlg^{-}$ which does not depend on the choice of $\ow$. Denote
the subspace by $\envAlg^{-}(w,e)$.

\end{Prop}

\begin{Def}

$\{F_{e}(\uc,\ow)|\uc\in\N^{[1,l]}\}$ is called the \emph{PBW basis
}(Poincar\'e-Birkhoff-Witt basis) of $\envAlg^{-}(w,e)$.

\end{Def}

When the context is clear, we could simply write $F_{e}(\uc)$ instead
of $F_{e}(\uc,\ow)$. 

For any $j\leq k$ such that $i_{j}=i_{k}$, set 
\begin{eqnarray*}
\beta_{[j,k]} & = & \sum_{s\in[j,k],i_{s}=i_{j}}\beta_{s}\in Q.
\end{eqnarray*}
It is straightforward to check that 
\begin{eqnarray*}
\beta_{[j,k]} & = & (\ow_{<j}-\ow_{\leq k})\varpi_{i_{j}}.
\end{eqnarray*}

We also define the multiplicity vector $\uc_{[j,k]}\in\N^{[1,l]}$
such that
\begin{eqnarray*}
(\uc_{[j,k]})_{s} & = & \begin{cases}
\delta_{i_{j},i_{s}}\  & s\in[j,k]\\
0 & \mathrm{elsewhere}
\end{cases}.
\end{eqnarray*}
 We denote 
\begin{eqnarray*}
F_{e}[j,k] & = & F_{e}(\uc_{[j,k]},\ow).
\end{eqnarray*}
Then $\wt F_{e}[j,k]=-\beta_{[j,k]}$.

We have the following Levendorskii-Soibelman straightening law (LS-law
for short, see \cite{LevSoi}\cite[Proposition 4.10]{Kimura10}).

\begin{Thm}

For any $j<k$, we have
\begin{eqnarray*}
F_{+1}(\beta_{k})F_{+1}(\beta_{j})-q^{-(\beta_{j},\beta_{k})}F_{+1}(\beta_{j})F_{+1}(\beta_{k}) & = & \sum_{\uc\in\N^{[j+1,k-1]}}f_{\uc}F_{+1}(\uc,\ow)\\
F_{-1}(\beta_{j})F_{-1}(\beta_{k})-q^{-(\beta_{j},\beta_{k})}F_{-1}(\beta_{k})F_{-1}(\beta_{j}) & = & \sum_{\uc\in\N^{[j+1,k-1]}}f{}_{\uc}F_{-1}(\uc,\ow)
\end{eqnarray*}
such that $f_{\uc}\in\kk$, and the equations are $Q_{-}$-grading
homogeneous.

\end{Thm}

It follows that $\envAlg^{-}(w,e)$ is the $\kk$-subalgebra of $\envAlg^{-}$
generated by $\{F_{e}(\beta_{j})|j\in[1,l]\}$.

\begin{Prop}[{\cite[Proposition 4.22]{Kimura10}\cite[38.2.3]{Lus:intro}}]

For any $\uc,\uc'\in\N^{[1,l]}$, we have 
\begin{eqnarray*}
(F_{e}(\uc,\ow),F_{e}(\uc',\ow))_{L} & = & \prod_{k=1}^{l}\delta_{c_{k},c'_{k}}\prod_{s=1}^{c_{k}}\frac{1}{1-q_{i_{k}}^{2s}}.
\end{eqnarray*}

\end{Prop}

\begin{Def}[{\cite[Definition 3.26]{kimura2017twist}}]

The algebra $\envAlg^{-}(w):=\envAlg^{-}(w,+1)$ is called the quantum
nilpotent subalgebra associated to $w\in W$. The algebra $\qO[N_{-}(w)]:=*(\envAlg^{-}(w))=\envAlg^{-}(w,-1)$
is called the quantum unipotent subgroup associated to $w$. 

\end{Def}

The quantum unipotent subgroup $\qO[N_{-}(w)]$ has a $Q_{-}$-grading
induced from that of $\envAlg^{-}$. It is a quantum analogue of the
coordinate ring of the unipotent subgroup $N_{-}(w)$, see \cite{kimura2017twist}
for more details.

\begin{Eg}[$\mathfrak{sl}_3$ type]\label{eg:sl3_type}

Consider the Cartan matrix $C=\left(\begin{array}{cc}
2 & -1\\
-1 & 2
\end{array}\right)$ and the reduced word $\ow=s_{1}s_{2}s_{1}$. We get $\beta_{1}=\alpha_{1}$,
$\beta_{2}=s_{1}\alpha_{2}=\alpha_{1}+\alpha_{2}$, $\beta_{3}=\alpha_{2}$.
Under our convention, we compute that $F(\beta_{1})=F_{1}$, $F(\beta_{2})=T_{1}F_{2}=F_{2}F_{1}-qF_{1}F_{2}$,
$F(\beta_{3})=T_{1}T_{2}F_{1}=F_{2}$. The LS-law now reads as

\begin{eqnarray*}
F(\beta_{3})F(\beta_{1})-qF(\beta_{1})F(\beta_{3}) & = & F(\beta_{2}).
\end{eqnarray*}
Notice that $F(\beta_{1})F(\beta_{3})=F((1,0,1),\ow)=F[1,3]$. By
applying the $*$ anti-involution, we get $F_{-1}(\beta_{2})=F_{1}F_{2}-qF_{2}F_{1}=T_{1}^{-1}(F_{2})$,
$F_{-1}[1,3]=F_{-1}(\beta_{3})F_{-1}(\beta_{1})=F_{2}F_{1}$, and
the LS-law reads as

\begin{eqnarray*}
F_{-1}(\beta_{1})F_{-1}(\beta_{3})-qF_{-1}(\beta_{3})F_{-1}(\beta_{1}) & = & F_{-1}(\beta_{2}).
\end{eqnarray*}

\end{Eg}

\subsection{Dual canonical bases\label{subsec:Dual-canonical-bases}}

Take a sign $e\in\{+1,-1\}$. We define the dual PBW basis of $\envAlg^{-}(w,e)$
to be
\begin{eqnarray*}
F_{e}^{\up}:=\{F_{e}^{\up}(\uc,\ow) & := & \frac{F_{e}(\uc,\ow)}{(F_{e}(\uc,\ow),F_{e}(\uc,\ow))_{L}}|\uc\in\N^{[1,l]}\}.
\end{eqnarray*}

The corresponding dual LS-law takes the following form.

\begin{Prop}[{\cite[Theorem 4.27]{Kimura10}}]\label{prop:dual_LS_law}

For any $j<k$, we have

\begin{eqnarray*}
q^{(\beta_{j},\beta_{k})}F_{-1}^{\up}(\beta_{j})F_{-1}^{\up}(\beta_{k})-F_{-1}^{\up}(\beta_{k})F_{-1}^{\up}(\beta_{j}) & = & \sum_{\uc\in\N^{[j+1,k-1]}}f_{\uc}^{*}F_{-1}^{\up}(\uc,\ow)
\end{eqnarray*}
such that $f_{\uc}^{*}\in\Q[q^{\pm}]$, and the equation is $Q_{-}$-grading
homogeneous.

\end{Prop}

For any two multiplicity vectors $\uc$, $\uc'$, we have the lexicographical
order $<$ such that $\uc'<\uc$ if, for some $s$, $c'_{s}<c_{s}$
and $c'_{i}=c_{i}$, $\forall i<s$. We also have the reverse lexicographical
order $<'$ such that $\uc'<'\uc$ if, for some $s$, $\uc'_{s}<\uc_{s}$
and $\uc'_{j}=\uc_{j}$ for all $j>s$. Introduce the partial order
$\prec$ such that $\uc'\prec\uc$ if $\uc'<\uc$ and $\uc'<'\uc$.

By Proposition \ref{prop:dual_LS_law}, we have 
\begin{eqnarray*}
\sigma F_{e}^{\up}(\uc,\ow)-F_{e}^{\up}(\uc,\ow) & \in & \sum_{\uc'\prec\uc}\Q[q^{\pm}]F_{e}^{\up}(\uc',\ow).
\end{eqnarray*}

The dual canonical basis $B_{e}^{\up}=\{B_{e}^{\up}(\uc,\ow)|\uc\in\N^{[1,l]}\}$
is defined to be the unique $\sigma$-invariant basis of $\envAlg^{-}(w,e)$
such that 
\begin{eqnarray*}
B_{e}^{\up}(\uc,\ow)-F_{e}^{\up}(\uc,\ow) & \in & \sum_{\uc'\prec\uc}q\Z[q]F_{e}^{\up}(\uc',\ow).
\end{eqnarray*}

As before, for $\uc=\uc_{[j,k]}$ such that $i_{j}=i_{k}$, we denote
$B_{e}^{\up}[j,k]=B_{e}^{\up}(\uc_{[j,k]},\ow)$ for simplicity.

\begin{Eg}\label{eg:sl3_canonical_basis}

Continue Example \ref{eg:sl3_type}. Recall that $\sigma(F_{i})=-q^{2}F_{i}$
for $i=1,2$, $\sigma(F_{1}F_{2})=q^{-1}\sigma(F_{2})\sigma(F_{1})=q^{3}F_{2}F_{1}$.
In addition, we can check $(F(\beta_{i}),F(\beta_{i}))_{L}=\frac{1}{1-q^{2}}$.
The dual PBW basis elements read as $F^{\up}(\beta_{1})=(1-q^{2})F_{1}$,
$F^{\up}(\beta_{3})=(1-q^{2})F_{2}$, $F^{\up}(\beta_{2})=(1-q^{2})(F_{2}F_{1}-qF_{1}F_{2})$.
It is straightforward to check that $\sigma(F^{\up}(\beta_{i}))=F^{\up}(\beta_{i})$,
$\forall i\in[1,3]$. Notice that $F^{\up}[1,3]=(1-q^{2})^{2}F[1,3]=(1-q^{2})^{2}F_{1}F_{2}$,
$\sigma F^{\up}[1,3]=(1-q^{-2})^{2}q^{3}F_{2}F_{1}$. The LS-law now
reads as

\begin{eqnarray*}
\sigma(F^{\up}[1,3])-F^{\up}[1,3] & = & (q^{-1}-q)F^{\up}(\beta_{2}).
\end{eqnarray*}

We can calculate the corresponding dual canonical basis element
\begin{eqnarray*}
B^{\up}[1,3] & = & F^{\up}[1,3]-qF^{\up}(\beta_{2}).
\end{eqnarray*}
We then obtain the following exchange relation in $\envAlg^{-}(w)$:

\begin{eqnarray*}
F^{\up}(\beta_{1})F^{\up}(\beta_{3}) & = & B^{\up}[1,3]+qF^{\up}(\beta_{2}).
\end{eqnarray*}

Apply the anti-involution $*$, we obtain the following exchange relation
in $\qO[N_{-}(w)]$:

\begin{eqnarray*}
F_{-1}^{\up}(\beta_{3})F_{-1}^{\up}(\beta_{1}) & = & B_{-1}^{\up}[1,3]+qF_{-1}^{\up}(\beta_{2}).
\end{eqnarray*}

Apply the dual bar involution $\sigma$, we get

\begin{eqnarray*}
q^{-1}F_{-1}^{\up}(\beta_{1})F_{-1}^{\up}(\beta_{3}) & = & B_{-1}^{\up}[1,3]+q^{-1}F_{-1}^{\up}(\beta_{2}).
\end{eqnarray*}
Notice that all factors appearing are dual canonical basis elements
of the form $B_{-1}^{\up}(\uc,\ow)$ for some $\uc$. Explicitly,
we have

\begin{align*}
F_{-1}^{\up}(\beta_{1}) & =(1-q^{2})F_{1}\\
F_{-1}^{\up}(\beta_{3}) & =(1-q^{2})F_{2}\\
F_{-1}^{\up}(\beta_{2}) & =(1-q^{2})(F_{1}F_{2}-qF_{2}F_{1})\\
B_{-1}^{\up}[1,3] & =(1-q^{2})(F_{2}F_{1}-qF_{1}F_{2}).
\end{align*}

\end{Eg}

\subsection{Subalgebras and unipotent quantum minors}

For any $j\leq k\in[1,l]$, we define the subalgebras $\envAlg^{-}(w,e)_{[j,k]}=\kk[F_{e}[\beta_{i}]]_{i\in[j,k]}$
of $\envAlg^{-}(w,e)$, $e=\pm1$, respectively. Notice that they
depend on the reduced expression $\ow$ of $w$. By the LS-law, $\envAlg^{-}(w,e)_{[j,k]}$
has the $\kk$-basis 
\begin{eqnarray*}
F_{e}^{\up} & : & =\{F_{e}^{\up}(\uc,\ow)|\uc\in\N^{[j,k]}\},
\end{eqnarray*}
called the dual PBW basis. 

As before, denote $\envAlg^{-}(w)_{[j,k]}=\envAlg^{-}(w,+1)_{[j,k]}$
and $\qO[N_{-}(w)]_{[j,k]}=\envAlg^{-}(w,-1)_{[j,k]}$. Then we have
we have $*\envAlg^{-}(w)_{[j,k]}=\qO[N_{-}(w)]_{[j,k]}$. 

The braid group action $T_{\leq j-1}=T_{\ow[1,j-1]}$ gives a $\kk$-algebra
isomorphism

\begin{eqnarray*}
T_{\leq j-1} & : & \envAlg^{-}(\ow_{[j,k]})\simeq\envAlg^{-}(w)_{[j,k]},
\end{eqnarray*}
because $T_{\leq j-1}$ sends the quantum root vector $T_{i_{j}}\cdots T_{i_{s-1}}F_{i_{s}}$
in $\envAlg^{-}(\ow_{[j,k]})$ to $(T_{i_{1}}\cdots T_{i_{j-1}})T_{i_{j}}\cdots T_{i_{s-1}}F_{i_{s}}=F(\beta_{s})\in\envAlg^{-}(w)_{[j,k]}$,
$\forall s\in[j,k]$. It follows that we have a $\kk$-algebra isomorphism
\begin{eqnarray*}
*T_{\leq j-1}* & : & \qO[N_{-}(\ow_{[j,k]})]\simeq\qO[N_{-}(w)]_{[j,k]}
\end{eqnarray*}
which identifies the dual PBW bases.

If we have $i_{j}=i_{k}=a\in[1,r]$, then the dual canonical basis
element $B_{-1}^{\up}[j,k]=B_{-1}^{\up}(\uc_{[j,k]},\ow)$ is called
a unipotent quantum minor, which we denote by $D[j,k]$. See $D[j,k]=D_{\ow_{\leq k}\varpi_{a},\ow_{<j}\varpi_{a}}$
in \cite[Definition 3.40]{kimura2017twist} or \cite[Section 6]{Kimura10}
\cite[Definition 5.3.2]{KimuraQin14}.

By definition of the dual canonical basis, $D[j,k]$ takes the following
form 
\begin{eqnarray}
D[j,k] & = & F_{-1}^{\up}(\beta_{k})\cdots F_{-1}^{\up}(\beta_{j[1]})F_{-1}^{\up}(\beta_{j})+Z\label{eq:quantum_minor_ld_term}
\end{eqnarray}
for some $Z\in\sum_{\uc\prec\uc_{[j,k]}}q\Z[q]B_{-1}^{\up}(\uc,\ow)$.
Notice that it is homogeneous with the weight

\begin{eqnarray*}
\wt D[j,k] & = & -\beta_{[j,k]}=-(\ow_{<j}-\ow_{\leq k})\varpi_{i_{j}}.
\end{eqnarray*}

Applying \cite{Kimura10} to the reduced word $\ow_{[j,k]}$ and using
the isomorphism $\qO[N_{-}(w)]_{[j,k]}\simeq\qO[N_{-}(\ow_{[j,k]})]$,
we obtain the following result.

\begin{Thm}[{\cite[Corollary 6.4 Section 6.2]{Kimura10}}]\label{thm:unipotent_minor_property}

Let there be given any $j\leq k\in[1,l]$ and $a\in\supp\ow_{[j,k]}$.
Denote $j_{1}=^{\min,\geq j}a$ and $j_{2}=^{\max,\leq k}a$.

(1) $D[j_{1},j_{2}]$ lies in the $q$-center of $\qO[N_{-}(w)]_{[j,k]}$,
i.e., for any $u\in\qO[N_{-}(w)]_{[j,k]}$, we have $D[j_{1},j_{2}]u=q^{s}uD[j_{1},j_{2}]$
for some $s\in\Z$.

(2) For any $d\in\N$, there exists some $s\in\Z$ such that $D[j_{1},j_{2}]^{d}=q^{s}B_{-1}^{\up}(d\uc_{[j_{1},j_{2}]},\ow)$.

\end{Thm}

\begin{Prop}[{\cite[Theorem 6.25]{Kimura10}}]\label{prop:factorization_dual_canonical_basis}

The dual canonical basis $B_{-1}^{\up}$ of $\qO[N_{-}(w)]$ factors
through the unipotent quantum minor $D[^{\min}a,^{\max}a]$, $a\in\supp w$,
up to a $q$-power, i.e., for any $\uc\in\N^{[1,l]}$, there exists
some $s\in\Z$ such that 
\begin{eqnarray*}
D[^{\min}a,^{\max}a]B_{-1}^{\up}(\uc,\ow) & = & q^{s}B_{-1}^{\up}(\uc+\uc_{[^{\min}a,^{\max}a]},\ow).
\end{eqnarray*}

\end{Prop}

\subsection{Localization}

For any $\lambda=\sum_{a\in\supp w}\lambda_{a}\varpi_{a}\in\oplus_{a}\N\varpi_{a}\simeq\N^{\supp w}$,
denote $\uc_{\lambda}=\sum\lambda_{a}\uc_{[^{\min}a,^{\max}a]}$.
Following \cite[Corollary 6.4]{Kimura10}\cite[Proposition 3.47]{kimura2017twist},
denote $D_{w\lambda,\lambda}=B_{-1}^{\up}(\uc_{\lambda},\ow)$ and
$\cD_{w}=\{q^{s}D_{w\lambda,\lambda}|s\in\Z,\forall\lambda\}$. By
Proposition \ref{prop:factorization_dual_canonical_basis}, the localization
$\qO[N_{-}(w)][\cD_{w}^{-1}]$ has the following basis

\begin{eqnarray*}
\circB_{-1}^{\up}(w) & := & \{q^{(\lambda,\wt S+\lambda-w\lambda)}D_{w\lambda,\lambda}^{-1}S|S=B_{-1}^{\up}(\uc,\ow),\forall\uc\in\N^{[1,l]},\forall\lambda\in\oplus_{a}\N\varpi_{a}\}.
\end{eqnarray*}

We refer the reader to \cite[5.1.3]{Kimura10}\cite[Definition 3.37]{kimura2017twist}
for the definition of the quotient algebra $\qO[N_{-}\cap X_{w}]$
of $\envAlg^{-}$, where $X_{w}$ denotes the Schubert variety. By
\cite[Theorem 5.13]{Kimura10}, the natural composition $\iota_{w}:\qO[N_{-}(w)]\hookrightarrow\envAlg^{-}\twoheadrightarrow\qO[N_{-}\cap X_{w}]$
is an embedding. For each basis element $b\in B_{-1}^{\up}$, denote
its image $\iota_{w}(b)=[b]$. Following \cite[Section 2.6, Section 4]{kimura2017twist},
the quantum unipotent cell associated to $w$ is defined as
\begin{align*}
\qO[N_{-}^{w}]: & =\qO[N_{-}\cap X_{w}][[\cD_{w}]^{-1}].
\end{align*}
By \cite[Theorem 4.13]{kimura2017twist}, $\iota_{w}$ induces a $Q$-graded
algebra isomorphism $\iota_{w}:\qO[N_{-}(w)][\cD_{w}^{-1}]\simeq\qO[N_{-}^{w}]$
called the De Concini-Procesi isomorphism. Then $\qO[N_{-}^{w}]$
with the following basis

\begin{eqnarray*}
 & \circB_{-1}^{\up,w} & :=\iota_{w}\circB_{-1}^{\up}(w)\\
 &  & =\{q^{(\lambda,\wt S+\lambda-w\lambda)}[D_{w\lambda,\lambda}]^{-1}[S]|[S]=[B_{-1}^{\up}(\uc,\ow)],\forall\uc\in\N^{[1,l]},\forall\lambda\in\oplus_{a}\N\varpi_{a}\}.
\end{eqnarray*}

We call $\circB_{-1}^{\up}(w)$ and $\circB_{-1}^{\up,w}$ the (localized)
dual canonical bases. 

The twisted dual bar involution $\sigma'$ naturally extends to a
$\Q$-anti-involution on the localization $\qO[N_{-}(w)][\cD_{w}^{-1}]\simeq\qO[N_{-}^{w}]$.
As before, define the $\Q(q)$-module endomorphism $c_{\tw}$ such
that $c_{\tw}(x)=q^{\Hf(\wt x,\wt x)-(\wt x,\rho)}x$ for homogeneous
$x$. Define the dual bar involution $\sigma=c_{\tw}\sigma'$ on the
localization. Then we still have $\sigma(xy)=q^{(\wt x,\wt y)}\sigma(y)\sigma(x)$
for homogeneous $x,y$. Moreover, the dual canonical bases $\circB_{-1}^{\up}(w)$
and $\circB_{-1}^{\up,w}$ are $\sigma$-invariant, see \cite[Proposition 4.9]{kimura2017twist}.

As before, define the $\Q(q)$-module endomorphism $\cor$ on $\qO[N_{-}(w)][\cD_{w}^{-1}]\simeq\qO[N_{-}^{w}]_{\Q(q^{\Hf})}$
such that $\cor(x)=q^{-\frac{1}{4}(\wt x,\wt x)+\Hf(\wt x,\rho)}x$
for homogeneous $x$. Then $x$ is $\sigma$-invariant if and only
if $\cor x$ is $\sigma'$-invariant. 

\section{Cluster structures on quantum unipotent subgroup\label{sec:Quantum-cluster-structures}}

A quantum unipotent subgroup $\qO[N_{-}(w)]$ is a \emph{symmetric
CGL extension}, see \cite[Defintion 2.6, Section 9]{goodearl2016berenstein}
and Section \ref{subsec:CGL-extension}. By a general theory in \cite{GY13},
it possesses a quantum cluster structure, which we will briefly introduce.
An explicit and detailed treatment could be found in the recent work
\cite{goodearl2020integral}.

\subsection{Quantum cluster structure}

Define the set of vertices $I=[1,l]$, $I_{\ufv}=\{k\in I|s(k)\neq+\infty\}$,
$I_{\fv}=I\backslash I_{\ufv}$. 

Following \cite[Section 4.8]{Kimura10}, define bilinear forms $c_{\ow}$,
$N_{\ow}$ on $\N^{[1,l]}$ such that

\begin{eqnarray*}
c_{\ow}(\uc,\uc') & = & \sum_{j<k}(c_{k}\beta_{k},c_{j}'\beta_{j})-\frac{1}{2}\sum_{k}c_{k}c_{k}'(\beta_{k},\beta_{k}),\\
N_{\ow}(\uc,\uc') & = & c_{\ow}(\uc,\uc')-c_{\ow}(\uc',\uc)\\
 & = & \sum_{k>j}(\beta_{k},\beta_{j})c_{k}c'_{j}-\sum_{k<i}(\beta_{k},\beta_{i})c_{k}c'_{i}.
\end{eqnarray*}

We say two elements $x,y$ $q$-commute if $xy=q^{s}yx$ for some
$s\in\Z$.

\begin{Prop}[{\cite[Proposition 4.33]{Kimura10}}]

Assume that $B_{-1}^{\up}(\uc)$ and $B_{-1}^{\up}(\uc')$ $q$-commute,
then we have

\begin{eqnarray}
B_{-1}^{\up}(\uc)B_{-1}^{\up}(\uc') & = & q^{N_{\ow}(\uc,\uc')}B_{-1}^{\up}(\uc')B_{-1}^{\up}(\uc).\label{eq:q_comm_canonical}
\end{eqnarray}

\end{Prop}

Now work with the extension $\qO[N_{-}(w)]_{\Q(q^{\Hf})}:=\qO[N_{-}(w)]\otimes_{\Q(q)}\Q(q^{\Hf})$.
Recall that we have $\cor x=q^{-\frac{1}{4}(\wt x,\wt x)+\Hf(\wt x,\rho)}x$
for homogeneous $x$, such that $\sigma(x)=x$ if and only if $\sigma'(\cor x)=\cor x$.

For $k\in[1,l]$, denote $X_{k}=\cor D[k^{\min},k]$. Notice that
$\wt X_{k}=-\beta_{[k^{\min},k]}=-(\beta_{k}+\beta_{k[-1]}+\cdots+\beta_{k^{\min}})$. 

Define the matrix $L=(L_{jk})_{j,k\in[1,l]}$ such that

\begin{eqnarray*}
L_{jk} & = & N_{\ow}(\uc_{[j^{\min},j]},\uc_{[k^{\min},k]}).
\end{eqnarray*}
Then we have $X_{j}X_{k}=q^{L_{jk}}X_{k}X_{j}.$ In particular, for
any $j\leq k$, we have the following (see \cite[Proposition 4.2]{kimura2017twist}\cite[Lemma 11.2]{GeissLeclercSchroeer11}):

\begin{eqnarray*}
L_{kj} & = & (\varpi_{i_{k}}+\ow_{\leq k}\varpi_{i_{k}},-\beta_{[j^{\min},j]}).
\end{eqnarray*}

\begin{Thm}[{\cite[Theorem 8.2]{GY13} \cite[Theorem B]{goodearl2020integral}}]\label{thm:quantum_cluster_structure}

(1) There exists a unique matrix $\tB(\ow)=(b_{ik})_{i\in I,k\in I_{\ufv}}$
such that $\tB(\ow)$ is compatible with $L$ such that $\sum_{j\in I}b_{ij}L_{jk}=-2\delta_{ik}\sym_{k}$
and $\sum_{i\in I}b_{ik}\wt X_{i}=0$ for $k\in I_{\ufv}$.

(2) Let $t_{0}=t_{0}(\ow)$ denote an initial seed associated to $\ow$,
such that $\tB(t_{0})=\tB(\ow)$, $\Lambda(t_{0})=L$, for $j\in[1,l]$.
Then we have an algebra isomorphism $\kappa$ from the partially compactified
quantum cluster algebra $\bClAlg(t_{0})_{\Q(q^{\Hf})}:=\bClAlg(t_{0})\otimes\Q(q^{\Hf})$
to the quantum unipotent subgroup $\qO[N_{-}(w)]_{\Q(q^{\Hf})}$,
such that the initial quantum cluster variables $X_{j}(t_{0})$, $j\in I$,
are identified with the rescaled unipotent quantum minors $X_{j}=\cor D[j^{\min},j]$.

\end{Thm}

\reviseStart

\begin{proof}

By \cite[Theorem 8.2]{GY13}\cite[Theorem B]{goodearl2020integral},
we have an isomorphism $\bClAlg(t_{0}')_{\Q(q^{\Hf})}\simeq\qO[N_{-}(w)]_{\Q(q^{\Hf})}$
for some quantum seed $t_{0}'$, such that $\tB(t_{0}')$ is uniquely
determined by $\Lambda(t_{0}')$ as in (1), and $X_{j}(t_{0}')$ are
identified with $\alpha_{j}D[j^{\min},j]$ for some scalars $\alpha_{j}$.

Notice that $L$ is determined by $X_{j}=\cor D[j^{\min},j]$ via
the $q$-commutative relations \eqref{eq:q_comm_canonical}, and $\Lambda(t_{0}')$
is determined by $\alpha_{j}D[j^{\min},j]$ similarly. So $L=\Lambda(t_{0}')$.
It follows that $\tB(t_{0}')=\tB(t_{0})$. A direct computation shows
that $\cor D[j^{\min},j]=\alpha_{j}D[j^{\min},j]$, see Remark \ref{rem:CGL_normalization_factor}.

\end{proof}

\reviseEnd

Notice that $\cor D[j^{\min},j]$ is $\sigma'$-invariant and $X_{j}(t_{0})=\kappa^{-1}\cor D[j^{\min},j]$
is invariant under the bar involution $\overline{(\ )}$ on $\bQClAlg(t_{0})$.
We obtain the following result.

\begin{Lem}\label{lem:identify_bar_involution}

The isomorphism $\kappa$ identifies the twisted dual bar involution
$\sigma'$ on $\qO[N_{-}(w)]_{\Q(q^{\Hf})}$ and $\overline{(\ )}$
on $\bQClAlg(t_{0})_{\Q(q^{\Hf})}$, i.e. for any $Z\in\bClAlg(t_{0})_{\Q(q^{\Hf})}$,
we have $\sigma'(\kappa Z)=\kappa\overline{Z}$.

\end{Lem}

\begin{Rem}\label{rem:initial_seed}

By \cite{BerensteinFominZelevinsky05}, the $I\times I_{\ufv}$-matrix
$\tB(\ow)=(b_{ik})_{i\in I,k\in I_{\ufv}}$is given by the following:

\begin{eqnarray*}
b_{jk} & = & \begin{cases}
1 & j=p(k)\\
-1 & j=s(k)\\
C_{i_{j}i_{k}} & j<k<s(j)<s(k)\\
-C_{i_{j}i_{k}} & k<j<s(k)<s(j)\\
0 & \mathrm{else}
\end{cases}.
\end{eqnarray*}
For verifying that such a matrix satisfies the condition in Theorem
\ref{thm:quantum_cluster_structure}, see \cite[Theorem 8.3]{BerensteinZelevinsky05}
(or arguments in \cite[Section 10.1]{GY13}).

\end{Rem}

\begin{Eg}

Continue Example \ref{eg:sl3_type} \ref{eg:sl3_canonical_basis}.
Notice that we have $\cor F_{i}=q^{-1}F_{i}$ and $\cor F_{1}F_{2}=q^{-\frac{3}{2}}F_{1}F_{2}$,
$\cor F_{2}F_{1}=q^{-\frac{3}{2}}F_{1}F_{2}$. We get

\begin{eqnarray*}
\cor(F_{-1}^{\up}(\beta_{1}))\cor(F_{-1}^{\up}(\beta_{3})) & = & q^{\Hf}\cor(B_{-1}^{\up}[1,3])+q^{-\Hf}\cor(F_{-1}^{\up}(\beta_{2})).
\end{eqnarray*}

Denote $X_{1}=\cor F_{-1}^{\up}(\beta_{1})$, $X_{1}'=\cor F_{-1}^{\up}(\beta_{3})$,
$X_{2}=\cor F_{-1}^{\up}(\beta_{2})$, $X_{3}=\cor B_{-1}^{\up}[1,3]$.
By using the defining relations (Section \ref{sec:Serre_relations}),
it is straightforward to check that (knowing that $F_{1}F_{2}^{2}F_{1}=F_{2}F_{1}^{2}F_{2}$)
we have $X_{1}X_{2}=q^{-1}X_{2}X_{1}$, $X_{1}X_{3}=qX_{3}X_{1}$,
$X_{2}X_{3}=X_{3}X_{2}$.

Denote $I=[1,3]$ and $I_{\ufv}=\{1\}$. We compute explicitly the
matrix $L=\left(\begin{array}{ccc}
0 & -1 & 1\\
1 & 0 & 0\\
-1 & 0 & 0
\end{array}\right)$. Take $(b_{ij})=\left(\begin{array}{ccc}
0 & -1 & 1\\
1 & 0 & -1\\
-1 & 1 & 0
\end{array}\right)$, then its $I\times I_{\ufv}$-submatrix $\tB$ is the same as in
Remark \ref{rem:initial_seed}. We have $\tB^{T}L=\left(\begin{array}{ccc}
2 & 0 & 0\end{array}\right)$.

By the above computation, $\qO[N_{-}(w)]_{\Q(q^{\Hf})}$ is a partially
compactified quantum cluster algebra with initial seed $t_{0}$ such
that $\tB(t_{0})=\tB(\ow)$ and $\Lambda(t_{0})=L$. Its mutation
rule reads as

\begin{align*}
X_{1}*X_{1}' & =q^{\Hf\Lambda(f_{1},f_{3})}X_{3}+q^{\Hf\Lambda(f_{1},f_{2})}X_{2}\\
 & =q^{\Hf}X_{3}+q^{-\Hf}X_{2}.
\end{align*}

\end{Eg}

\subsection{Quantum cluster variables}

\begin{Lem} 

Let there be given any $j\leq k\in[1,l]$ such that $i_{j}=i_{k}$.
The rescaled unipotent quantum minor $\cor D[j,k]$ is identified
with a quantum cluster variable in $\bClAlg(t_{0})$. 

\end{Lem}

\reviseStart

\begin{proof}

By \cite[Theorem 7.3]{goodearl2020integral}, there exists a scalar
$\xi\in q^{\Hf\Z}$ such that $\kappa^{-1}\xi D[j,k]$ is a quantum
cluster variable in some seed. Identify the bar-involution and $\sigma'$
by Lemma \ref{lem:identify_bar_involution}. Then the bar-invariant
quantum cluster variable is identified with the $\sigma'$-invariant
element $\cor D[j,k]$.

\end{proof}

\reviseEnd

\begin{Thm}[{\cite[Theorem 5.3, Theorem 8.2(c)]{GY13}}]\label{thm:inj_seed_admissible_mutations}

There is a seed $t_{0}[1]\in\Delta^{+}$ whose quantum cluster variables
are $\cor D[j,j^{\max}]$, $j\in[1,l]$. Moreover, there exists a
mutation sequence $\seq_{}$ from $t_{0}$ to $t_{0}[1]$ such that
the quantum cluster variables obtained along the mutation sequence
take the form $\cor D[j,k]$ for $j\leq k\in[1,l]$, $i_{j}=i_{k}$.

\end{Thm}

See Remark \ref{rem:injective_mutation_sequence} for a choice of
$\seq$. By Lemma \ref{lem:parametrize_injective}, the seed $t_{0}[1]$
in Theorem \ref{thm:inj_seed_admissible_mutations} is shifted from
$t_{0}$ in the sense of Definition \ref{def:inj_reachable}.

\section{Dual canonical bases are common triangular bases \label{sec:Dual-canonical-bases-results}}

In this section, we apply previous discussion of triangular bases
and cluster twist automorphisms to quantum unipotent subgroups and
quantum unipotent cells.

Let there be given any partially compactified quantum cluster algebra
$\bQClAlg(t_{0})_{\Q(q^{\Hf})}\simeq\qO[N_{-}(w)]_{\Q(q^{\Hf})}$
as in Theorem \ref{thm:quantum_cluster_structure}. Recall that $I=[1,l]$,
$I_{\fv}=\{j\in I|j=j^{\max}\}$.

\subsection{Parametrization}

Recall that the initial quantum cluster variables $X_{j}(t_{0})$
are identified with the rescaled unipotent quantum minors $\cor D[j^{\min},j]$,
$j\in I$. Moreover, the dual canonical basis elements $D[j^{\min},j]=B_{-1}^{\up}(\uc_{[j^{\min},j]},\ow)$
are parametrized by $\uc_{[j^{\min},j]}\in\N^{I}$ and $X_{j}(t_{0})$
by their leading degrees $f_{j}\in\Mc(t_{0})$ ($j$-th unit vector).
Following \cite{qin2017triangular}, we translate the multiplicity
vector $\uc\in\N^{[1,l]}$ to a vector in $\Mc(t_{0})$ by defining
a linear map $\theta^{-1}$ from $\N^{[1,l]}$ to $\Mc(t_{0})=\Z^{I}$
such that 
\begin{eqnarray*}
\theta^{-1}(\uc_{[j,j]})=\begin{cases}
f_{j}-f_{p(j)} & p(j)\neq-\infty\\
f_{j} & p(j)=-\infty
\end{cases}
\end{eqnarray*}
It follows that $\theta^{-1}$ is injective. We further define the
bijective linear map $\theta:\Mc(t_{0})\simeq\Z^{[1,l]}$, called
the parametrization map, such that

\begin{eqnarray*}
\theta(f_{j})= & \uc_{[j^{\min},j]}.
\end{eqnarray*}

Let us make the identification $\bQClAlg(t_{0})\simeq\qO[N_{-}(w)]$.
Recall that, by Theorem \ref{thm:quantum_cluster_structure}, the
rescaled unipotent quantum minors $\cor D[j,k]$, $i_{j}=i_{k}$,
are quantum cluster monomials. In particular, the rescaled dual PBW
generators $\cor F_{-1}^{\up}(\beta_{s})$, $s\in[1,l]$, are quantum
cluster variables.

\begin{Prop}\label{prop:pointed_basis_parametrization}

The following statements are true.

(1) The rescaled dual PBW basis $\cor F_{-1}^{\up}$ is a $\theta^{-1}(\N^{[1,l]})$-pointed
basis of $\bQClAlg(t_{0})$ such that $\cor F_{-1}^{\up}(\uc)$ is
$\theta^{-1}(\uc)$-pointed.

(2) The rescaled dual canonical basis $\cor B_{-1}^{\up}$ is a $\theta^{-1}(\N^{[1,l]})$-pointed
basis of $\bQClAlg(t_{0})$ such that $\cor B_{-1}^{\up}(\uc)$ is
$\theta^{-1}(\uc)$-pointed.

(3) The rescaled dual canonical basis $\cor B_{-1}^{\up}$ is $(\prec_{t_{0}},q\Z[q])$-unitriangular
to the rescaled dual PBW basis $\cor F_{-1}^{\up}$.

\end{Prop}

\begin{proof}

It is possible to prove the statements based on computation of the
bilinear $N_{\ow}$. Let us give a more abstract proof based on Lemma
\ref{lem:identify_bar_involution}, which identifies the twisted dual
bar involution $\sigma'$ and the bar involution $\overline{(\ )}$.

Recall that the rescaled dual PBW generators $\cor F_{-1}^{\up}(\beta_{s})$,
$s\in[1,l]$, are quantum cluster variables and, in particular, pointed.
It follows that, for any $\cor F_{-1}^{\up}(\uc)$, there exists a
scalar $\xi_{\uc}\in q^{\Hf\Z}$ such that $\tF_{-1}^{\up}(\uc):=\xi_{\uc}\cor F_{-1}^{\up}(\uc)$
is pointed. Denote $\tF_{-1}^{\up}=\{\tF_{-1}^{\up}(\uc)|\uc\}$,
then it is a pointed set. We have a $(<,q\Z[q])$-decomposition 
\begin{align*}
\cor B_{-1}^{\up}(\uc) & =\cor F_{-1}^{\up}(\uc)+\sum_{\uc'<\uc}b_{\uc'}\cor F_{-1}^{\up}(\uc')\\
 & =\xi_{\uc}^{-1}\tF_{-1}^{\up}(\uc)+\sum_{\uc'<\uc}\xi_{\uc'}^{-1}b_{\uc'}\tF_{-1}^{\up}(\uc')
\end{align*}
with coefficients $b_{\uc'}\in q\Z[q]$.

Notice that $<$ is an order on $\N^{[1,l]}$ bounded from below.
We prove by induction the following claim (0): $\xi_{\uc}=1$ for
all $\uc\in\N^{[1,l]}$, and $\cor B_{-1}^{\up}(\uc)$ is pointed
at $\deg^{t_{0}}\cor F_{-1}^{\up}(\uc)$. 

If no $\uc'$ appear in the decomposition, we have $\cor B_{-1}^{\up}=\xi_{\uc}^{-1}\tF_{-1}^{\up}(\uc)$
which is $\sigma'$-invariant and thus bar-invariant. Since $\tF_{-1}^{\up}(\uc)$
is pointed, the leading term coefficient $\xi_{\uc}^{-1}\in q^{\Hf\Z}$
must be bar-invariant, which implies that $\xi_{\uc}=1$. Assume by
induction that the claim has been verified for all $\uc'$ appearing.
Then the bar-invariance of $\cor B_{-1}^{\up}(\uc)$ implies that
the coefficients of its $\prec_{t_{0}}$-maximal degrees are bar-invariant,
which cannot take the form $\xi_{\uc'}^{-1}b_{\uc'}=b_{\uc'}\in q\Z[q]$.
It follows that $\cor B_{-1}^{\up}$ has the unique $\prec_{t_{0}}$-maximal
degree $\deg^{t_{0}}\tF_{-1}^{\up}(\uc)=\theta^{-1}(\uc)$ with the
coefficient $\xi_{\uc}=1$.

(1) Applying the above claim (0) to the cases $\cor B_{-1}^{\up}(\uc)=X_{j}=D[j^{\min},j]$,
$j\in[1,l]$, we obtain that 
\begin{align*}
\sum_{s\leq j,i_{s}=i_{j}}\deg^{t_{0}}F_{-1}^{\up}(\beta_{s}) & =\deg^{t_{0}}\cor F_{-1}^{\up}(\uc_{[j^{\min},j]})\\
 & =\deg^{t_{0}}X_{j}\\
 & =f_{j}.
\end{align*}
It follows that $\deg^{t_{0}}\cor F_{-1}^{\up}(\beta_{s})=\theta^{-1}(\uc_{[s,s]})$
for $s\in[1,l]$. As a consequence, $\deg^{t_{0}}\cor F_{-1}^{\up}(\uc)=\theta^{-1}(\uc)$.
Again, by the above claim (0), $\cor F_{-1}^{\up}(\uc)$ is pointed.
The statement follows.

(2) The statement follows from the statement (1) and the above claim
(0).

(3) Lemma \ref{lem:finite_decomposition_triangular} implies that
the decomposition must be $\prec_{t_{0}}$-unitriangular. The statement
follows.

\end{proof}

Endow $\hLP(t_{0})$ with the natural $Q$-grading induced from that
of $\wt X_{i}(t_{0})$, $i\in I$. By Theorem \ref{thm:quantum_cluster_structure},
$\wt Y_{k}(t_{0})=0$ for $k\in I$. Then all pointed functions in
$\hLP(t_{0})$ are homogeneous. The following result follows.

\begin{Lem}\label{lem:pointed_func_weight}

If $S_{g}\in\hLP(t_{0})$ is $g$-pointed for $g\in\Mc(t_{0})$, then
$\wt S_{g}=-\sum c_{j}\beta_{j}$ were we denote $\uc=(c_{j})_{j\in I}=\theta(g)$.

\end{Lem}

\begin{Lem}\label{lem:parametrize_injective}

For any $k\in I_{\ufv}$, we have $I_{k}(t_{0})=\cor D[k[1],k^{\max}]$
and
\begin{eqnarray}
\deg^{t_{0}}I_{k}(t_{0}) & = & -f_{k}+f_{k^{\max}}.\label{eq:deg_injective}
\end{eqnarray}

\end{Lem}

\begin{proof}

Notice that $\cor D[k[1],k^{\max}]$ is a quantum cluster variable
by Theorem \ref{thm:inj_seed_admissible_mutations}. In addition,
we have $\deg^{t_{0}}\cor D[k[1],k^{\max}]=\theta^{-1}(\uc_{[k[1],k^{\max}]})=-f_{k}+f_{k^{\max}}$.
It follows that $\cor D[k[1],k^{\max}]=\seq_{t_{0}[1],t_{0}}^{*}X_{\sigma k}(t_{0}[1])=I_{k}(t_{0})$
by definition.

\end{proof}

For completeness, we compare the partial orders on both sides of the
map $\theta$, though it will not be used in this paper. Recall that
$<'$ denote the reverse lexicographical order (on $\Z^{[1,l]}$),
see Section \ref{subsec:Dual-canonical-bases}.

\begin{Lem}

For any $g'\prec_{t_{0}}g$ in $\Mc(t_{0})$, we have $\theta(g')<'\theta(g)$
in $\Z^{[1,l]}$.

\end{Lem}

\begin{proof}

It suffices to check that $\theta(\deg Y_{k})<'0$ for all $k\in I_{\ufv}$.
We have $\deg Y_{k}=-f_{s(k)}+\sum_{j<s(k)}b_{jk}f_{j}$ by the definition
of $\tB(t_{0})$ (Remark \ref{rem:initial_seed}). It follows that
$\theta(\deg Y_{k})\in-\uc_{[k^{\min},s(k)]}+\Z^{[1,s(k)-1]}<'0$.
The claim follows.

\end{proof}

\subsection{Integral form\label{subsec:Integral-form}}

Denote $\bA=\Q[q^{\pm}].$ Recall that $\qO[N_{-}(w)]$ is the $\Q(q)$-algebra
generated by the dual PBW generators $F_{-1}^{\up}(\beta_{i})$, $i\in I$
over $\Q(q)$, subject to the $Q$-grading homogeneous relations given
by the LS-law:

\begin{eqnarray*}
q^{(\beta_{j},\beta_{k})}F_{-1}^{\up}(\beta_{j})F_{-1}^{\up}(\beta_{k})-F_{-1}^{\up}(\beta_{k})F_{-1}^{\up}(\beta_{j}) & = & \sum_{\uc\in\N^{[j+1,k-1]}}b_{j,k}(\uc)F_{-1}^{\up}(\uc,\ow)
\end{eqnarray*}
for $j<k$, where the coefficients $b_{j,k}(\uc)\in\Q[q^{\pm}]$.

The $\bA$-algebra $\qO[N_{-}(w)]_{\bA}$ generated by the dual PBW
generators $F_{-1}^{\up}(\beta_{i})$, $i\in I$, is called the \emph{integral
form} of $\qO[N_{-}(w)]$. See \cite{Kimura10} for more details.

Work with the (partially compactified) quantum cluster algebra $\bQClAlg(t_{0})_{\Q(q^{\Hf})}\simeq\qO[N_{-}(w)]_{\Q(q^{\Hf})}$.
By applying the linear map $\cor$ to each PBW generator, we get the
following rescaled LS-law:

\begin{eqnarray}
q^{(\beta_{j},\beta_{k})}\cor F_{-1}^{\up}(\beta_{j})\cor F_{-1}^{\up}(\beta_{k})-\cor F_{-1}^{\up}(\beta_{k})\cor F_{-1}^{\up}(\beta_{j})\label{eq:commutator_generator}\\
=\sum_{\uc\in\N^{[j+1,k-1]}}b_{j,k}'(\uc)\cor F_{-1}^{\up}(\uc,\ow)\nonumber 
\end{eqnarray}
where $b_{j,k}'(\uc)\in b_{j,k}(\uc)\cdot q^{\Hf\Z}$.

Notice that $\cor F_{-1}^{\up}(\beta_{i})$ are quantum cluster variables.
So it is not surprising that the coefficients appearing in \eqref{eq:commutator_generator}
should belong to $\Z[q^{\pm\Hf}]$, which we give a rigorous proof
below using cluster theory.

\begin{Lem}\label{lem:integer_coefficient}

We have $b_{j,k}'(\uc)\in\Z[q^{\pm\Hf}]$.

\end{Lem}

\begin{proof}

Denote $Z=q^{(\beta_{j},\beta_{k})}\cor F_{-1}^{\up}(\beta_{j})\cor F_{-1}^{\up}(\beta_{k})-\cor F_{-1}^{\up}(\beta_{k})\cor F_{-1}^{\up}(\beta_{j})$.
Then $Z\in\LP(t_{0})$. Notice that $\cor F_{-1}^{\up}(\uc)$ is a
pointed set in $\LP(t_{0})$ by Proposition \ref{prop:pointed_basis_parametrization}.
Then the finite decomposition in \eqref{eq:commutator_generator}
is the $\prec_{t_{0}}$-decomposition in $\LP(t_{0})$ with coefficients
in $\Z[q^{\pm\Hf}]$ by Lemma \ref{lem:decomposition_subring}.

\end{proof}

\begin{Cor}\label{cor:str_const_integer}

We have $b_{j,k}(\uc)\in\Z[q^{\pm}]$.

\end{Cor}

\begin{proof}

We have $b_{j,k}(\uc)\in\Z[q^{\pm\Hf}]\cap\Q[q^{\pm}]$. The claim
follows.

\end{proof}

\reviseStart

Note that one can also deduce Corollary \ref{cor:str_const_integer}
from properties of the dual canonical basis \cite[Proposition 14.2.6]{Lus:intro}
and the dual PBW basis \cite[Proposition 4.26, Theorem 4.29]{Kimura10}.\reviseEnd

Correspondingly, the dual PBW generators $F_{-1}^{\up}(\beta_{i})$,
$i\in I$, generate a $\Z[q^{\pm}]$-algebra, which we denote by $\qO[N_{-}(w)]_{\Z[q^{\pm}]}$.
Then the dual PBW basis $\{F_{-1}^{\up}(\uc,\ow)|\uc\in\N^{[1,l]}\}$
is a $\Z[q^{\pm}]$-basis of $\qO[N_{-}(w)]_{\Z[q^{\pm}]}$.

Consider the extension $\qO[N_{-}(w)]_{\Z[q^{\pm\Hf}]}=\qO[N_{-}(w)]_{\Z[q^{\pm}]}\otimes_{\Z[q^{\pm}]}\Z[q^{\pm\Hf}]$.
Theorem \ref{thm:quantum_cluster_structure} could be strengthened
as the following, which was also proved in the recent work \cite[Theorem B]{goodearl2020integral}. 

\begin{Thm}\label{thm:integral_cl_structure}

Take the initial seed $t_{0}=t_{0}(\ow)$ as before. We have a $\Z[q^{\pm\Hf}]$-algebra
isomorphism $\kappa:\bQClAlg(t_{0})\simeq\qO[N_{-}(w)]_{\Z[q^{\pm\Hf}]}$
such that the initial quantum cluster variables $X_{j}(t_{0})$, $j\in I$,
are identified with the rescaled unipotent quantum minors $\cor D[j^{\min},j]$.

\end{Thm}

\begin{proof}

Make the identification $\bQClAlg(t_{0})_{\Q(q^{\Hf})}\simeq\qO[N_{-}(w)]_{\Q(q^{\Hf})}$
by Theorem \ref{thm:quantum_cluster_structure}.

Recall that the rescaled dual PBW basis $\cor F_{-1}^{\up}$ is a
$\Z[q^{\pm\Hf}]$-basis of $\qO[N_{-}(w)]_{\Z[q^{\pm\Hf}]}$. In addition,
the rescaled dual PBW generators are quantum cluster variables. Therefore,
$\qO[N_{-}(w)]_{\Z[q^{\pm\Hf}]}$ is a subalgebra of the quantum cluster
algebra $\bQClAlg(t_{0})$.

Moreover, the rescaled dual PBW basis $\cor F_{-1}^{\up}$ is a $\Q(q^{\Hf})$-basis
of $\bQClAlg(t_{0})_{\Q(q^{\Hf})}\simeq\qO[N_{-}(w)]_{\Q(q^{\Hf})}$.
It follows that any $Z\in\bQClAlg(t_{0})\subset\LP(t_{0})$ has a
finite decomposition in terms of $\cor F_{-1}^{\up}$. By Lemma \ref{lem:decomposition_subring},
this decomposition is the $\prec_{t_{0}}$-decomposition in $\LP(t_{0})$
with coefficients in $\Z[q^{\pm\Hf}]$. Therefore, the rescaled dual
PBW basis $\cor F_{-1}^{\up}$ is a $\Z[q^{\pm\Hf}]$-basis of $\bQClAlg(t_{0})$.

\end{proof}

Applying Theorem \ref{thm:integral_cl_structure} to symmetric Kac-Moody
cases, we get the following result.

\begin{Cor}[{\cite[Theorem 9.1.3]{qin2017triangular}\cite[Main Theorem 1]{Kang2018}}]

Conjecture \cite[Conjecture 12.7]{GeissLeclercSchroeer11} holds true.

\end{Cor}

Since the dual PBW basis is a $\Z[q^{\pm}]$-basis of the algebra
$\qO[N_{-}(w)]_{\Z[q^{\pm}]}$, so is the dual canonical basis $B_{-1}^{\up}$.
In particular, the multiplication structure constants of $B_{-1}^{\up}$
take value in $\Z[q^{\pm}]$. Correspondingly, let $\qO[N_{-}(w)\cap wG_{0}^{\min}]_{\Z[q^{\pm}]}$
and $\qO[N_{-}^{w}]_{\Z[q^{\pm}]}$ denote the free $\Z[q^{\pm}]$-modules
spanned by the localized dual canonical bases $\circB_{-1}^{\up}(w)$
and $\circB_{-1}^{\up,w}$ respectively. Then they are isomorphic
$\Z[q^{\pm}]$-algebras:
\begin{eqnarray*}
\iota_{w} & : & \qO[N_{-}(w)\cap wG_{0}^{\min}]_{\Z[q^{\pm}]}\simeq\qO[N_{-}^{w}]_{\Z[q^{\pm}]}.
\end{eqnarray*}

In addition, notice that $\qO[N_{-}(w)\cap wG_{0}^{\min}]_{\Z[q^{\pm}]}$
can also be viewed as the localization of the $\qO[N_{-}(w)]_{\Z[q^{\pm}]}$
with respect to $\cD_{w}$. By Theorem \ref{thm:integral_cl_structure},
it is isomorphic to the (localized) quantum cluster algebra $\qClAlg(t_{0})$
after an extension to $\Z[q^{\pm\Hf}]$. We obtain the following.

\begin{Thm}\label{thm:integral_cl_structure_localized}

Take the initial seed $t_{0}=t_{0}(\ow)$ as before. We have a $\Z[q^{\pm\Hf}]$-algebra
isomorphism $\iota_{w}\kappa:\qClAlg(t_{0})\simeq\qO[N_{-}^{w}]_{\Z[q^{\pm\Hf}]}$
such that the initial quantum cluster variables $X_{j}(t_{0})$, $j\in I$,
are identified with the rescaled unipotent quantum minors $\cor[D[j^{\min},j]]$.

\end{Thm}

\begin{Rem}[Rescaled quantum cluster algebra]

We remark that the extension to include $q^{\Hf}$ is a technical
trick needed for changing from the dual bar involution $\sigma$ to
the twisted dual bar involution $\sigma'$, where $\sigma'$ is identified
with the usual bar involution $\overline{(\ )}$ for quantum cluster
algebras (Lemma \ref{lem:identify_bar_involution}). For readers do
not want to introduce $q^{\Hf}$, one can instead rescale the quantum
cluster algebras as below (see \cite[Section 10.4]{GeissLeclercSchroeer11}).

Notice that every quantum cluster monomial $x$ is a homogeneous element.
Correspondingly, we can define the rescaled quantum cluster monomials
$\cor^{-1}x$, which correspond to dual canonical basis elements by
Theorem \ref{thm:conjecture_proved}. The rescaled partially compactified
quantum cluster algebra can be defined to be the $\Z[q^{\pm}]$-algebra
generated by the rescaled quantum cluster variables. Then it is isomorphic
to $\qO[N_{-}(w)]_{\Z[q^{\pm}]}$. Correspondingly, the rescaled (localized)
quantum cluster algebra defined as a localization is isomorphic to
$\qO[N_{-}^{w}]_{\Z[q^{\pm}]}$.

\end{Rem}

\subsection{Localization and the initial triangular basis\label{subsec:Localization-and-bases}}

Taking the localization at frozen variables $X_{j}=\kappa^{-1}D[j^{\min},j]$,
$j\in I_{\fv}$, we obtain the (localized) quantum cluster algebra
$\qClAlg(t_{0})=\bQClAlg(t_{0})[X_{j}^{-1}]_{j\in I_{\fv}}$. Using
Proposition \ref{prop:factorization_dual_canonical_basis}, we deduce
that by taking the localization of the dual canonical basis at frozen
variables, $\qClAlg(t_{0})$ has a $\Mc(t_{0})$-pointed $\Z[q^{\pm\Hf}]$-basis
\begin{eqnarray}
\can^ {} & := & \{[X^{d}*S]^{t_{0}}|S\in\kappa^{-1}B_{-1}^{\up},d\in\Z^{I_{\fv}}\}.\label{eq:initial_can_basis}
\end{eqnarray}

\begin{Prop}[{\cite[Corollary 9.1.9]{qin2017triangular}}]\label{prop:initial_triangular_basis}

The basis $\can^ {}$ is the triangular basis with respect to the
initial seed $t_{0}$.

\end{Prop}

\begin{proof}

We refer the reader to \cite[Section 9.1]{qin2017triangular} for
a detailed proof. The key step is \cite[Proposition 9.1.8]{qin2017triangular}
which verifies the $\prec_{t_{0}}$-unitriangularity property of the
dual canonical basis by an induction on the length of $w$ with the
help of the dual PBW basis.

\end{proof}

\begin{Thm}

The isomorphism $\iota_{w}\kappa:\qClAlg(t_{0})\simeq\qO[N_{-}^{w}]_{\Z[q^{\pm\Hf}]}$
identifies the initial triangular basis $\can$ and the rescaled localized
dual canonical basis $\cor\circB_{-1}^{\up,w}$.

\end{Thm}

\begin{proof}

Recall that we denote $\iota_{w}b=[b]$ for dual canonical basis elements
$b\in B_{-1}^{\up}$.

By construction, elements of $\can$ take the form $b(\lambda,S)=[X^{-\lambda}*S]^{t_{0}}$
for $S\in\kappa^{-1}B_{-1}^{\up}$, $\lambda\in\N^{I_{\fv}}$. The
corresponding elements $b'(\lambda,S)$ in $\circB_{-1}^{\up,w}$
take the form $q^{(\lambda,\wt S+\lambda-w\lambda)}[D_{w\lambda,\lambda}]^{-1}[S]$,
where we identify $\N^{I_{\fv}}\simeq\oplus_{a\in\supp w}\N\varpi_{a}$
such that $D_{w\lambda,\lambda}=\kappa X^{\lambda}$. Then we have
$\kappa b(\lambda,S)=\xi\cor b'(\lambda,S)$ for some $\xi\in q^{\Hf\Z}$.
On the one hand, since $b(\lambda,S)$ is invariant under the bar
involution $\overline{(\ )}$, $\kappa b(\lambda,S)$ must be invariant
under the twisted dual bar involution $\sigma'$. On the other hand,
since $b'(\lambda,S)$ is $\sigma$-invariant, $\cor b'(\lambda,S)$
is $\sigma'$-invariant. Because $\sigma'$ and the bar involution
$\overline{(\ )}$ are identified, we have $\xi=1$. It follows that
$\iota_{w}\kappa\can=\cor\circB_{-1}^{\up,w}$.

\end{proof}

\subsection{Twist automorphism\label{subsec:Twist-automorphism-compare}}

\cite{kimura2017twist} introduced a quantum analogue of the twist
automorphism for quantum unipotent cells $\qO[N_{-}^{w}]$, which
was denoted by $\eta_{w,q}$.

\begin{Thm}[{\cite[Theorem 6.1]{kimura2017twist}}]\label{thm:twist_unipotent_cell}

There exists a $\Q(q)$-algebra automorphism $\eta_{w,q}$ on $\qO[N_{-}^{w}]$
such that, for $k\in[1,l]$, we have
\begin{align*}
\eta_{w,q}[D[k^{\min},k]] & =q^{-(\varpi_{i_{k}},-\beta_{[k^{\min},k]})}(D[k^{\min},k^{\max}])^{-1}D[k[1],k^{\max}],\ k\in I_{\ufv},\\
\eta_{w,q}[D[k^{\min},k^{\max}]^{-1}] & =q^{(\varpi_{i_{k}},-\beta_{[k^{\min},k^{\max}]})}D[k^{\min},k^{\max}].
\end{align*}
In particular, $\wt\eta_{w,q}x=-\wt x$ for homogeneous $x$. Moreover,
$\eta_{w,q}$ restricts to a permutation on the (localized) dual canonical
basis $\circB_{-1}^{\up,w}$ and commutes with the dual bar involution
$\sigma$.

\end{Thm}

The rescaling map $\cor$ decomposes as $\cor=\cor_{2}\cor_{1}=\cor_{1}\cor_{2}$
for the $\Z[q^{\pm\Hf}]$-module endomorphisms $\cor_{1}$, $\cor_{2}$
on $\qO[N_{-}^{w}]_{\Z[q^{\pm\Hf}]}$ such that, for any homogeneous
$x$, we have 
\begin{align*}
\cor_{1}(x) & =c_{1}(x)x:=q^{\Hf(\wt x,\rho)}x\\
\cor_{2}(x) & =c_{2}(x)x:=q^{-\frac{1}{4}(\wt x,\wt x)}x.
\end{align*}
Then $\cor_{1}$ is an algebra automorphism of the $Q$-graded algebra
$\qO[N_{-}^{w}]_{\Z[q^{\pm\Hf}]}$. Notice that $\eta_{w,q}$ is well-defined
on $\qO[N_{-}^{w}]_{\Z[q^{\pm}]}$. We rescale $\eta_{w,q}$ to the
following automorphism on $\qO[N_{-}^{w}]_{\Z[q^{\pm}]}$:

\begin{eqnarray}
\tw_{w} & := & \eta_{w,q}\cor_{1}^{-2}.\label{eq:rescale_twist_automorphism}
\end{eqnarray}

\begin{Lem}\label{lem:rescale_twsit_auto_commute_bar}

The rescaled twist automorphism $\tw_{w}$ commutes with $\sigma'$.

\end{Lem}

\begin{proof}

For any homogeneous $x$, we have $\sigma\eta_{w,q}x=\eta_{w,q}\sigma x$
by Theorem \ref{thm:twist_unipotent_cell}. Since $\wt\eta_{w,q}x=-\wt x$,
we have $c_{1}(\eta_{w,q}x)=c_{1}(x)^{-1}$ and $c_{2}(\eta_{w,q}x)=c_{2}(x)$.
Recall that $\sigma=\cor^{-2}\sigma'$. We compute

\begin{align*}
\sigma\eta_{w,q}x & =\cor_{2}^{-2}\cor_{1}^{-2}\sigma'\eta_{w,q}(x)\\
 & =c_{2}(\eta_{w,q}x)^{-2}c_{1}(\eta_{w,q}x)^{-2}\sigma'\eta_{w,q}x\\
 & =c_{2}(x)^{-2}c_{1}(x)^{2}\sigma'\eta_{w,q}x\\
 & =c_{2}(x)^{-2}\sigma'(c_{1}(x)^{-2}\eta_{w,q}x)\\
 & =c_{2}(x)^{-2}\sigma'(\eta_{w,q}c_{1}(x)^{-2}x)\\
 & =c_{2}(x)^{-2}\sigma'\tw_{w}x
\end{align*}
Similarly, we have 
\begin{eqnarray*}
\eta_{w,q}\sigma x & = & \eta_{w,q}c_{2}(x)^{-2}c_{1}(x)^{-2}\sigma'x\\
 & = & c_{2}(x)^{-2}\eta_{w,q}c_{1}(\sigma'x)^{-2}\sigma'x\\
 & = & c_{2}(x)^{-2}\tw_{w}(\sigma'x).
\end{eqnarray*}

It follows that $\sigma'\tw_{w}x=\tw_{w}\sigma'x$.

\end{proof}

The automorphism $\tw_{w}$ on $\qO[N_{-}^{w}]_{\Z[q^{\pm}]}$ induces
an automorphism $\tw_{w}$ on $\qClAlg(t_{0})$ via the identification
$\qClAlg(t_{0})\simeq\qO[N_{-}^{w}]_{\Z[q^{\pm\Hf}]}$. We translate
Theorem \ref{thm:twist_unipotent_cell} for cluster algebras as the
following. 

\begin{Thm} \label{thm:twist_auto_invariance}

Let there be given $t_{0}=t_{0}(\ow)$ as before. The algebra automorphism
$\tw_{w}$ on $\qClAlg(t_{0})$ satisfies
\begin{eqnarray}
\tw_{w}(X_{k}(t_{0})) & = & [X_{k^{\max}}^{-1}*I_{k}(t_{0})]^{t_{0}},\qquad k\in I_{\ufv}\label{eq:twist_unfrozen}\\
\tw_{w}(X_{j}(t_{0})) & = & X_{j}(t_{0})^{-1},\qquad j\in I_{\fv}.\label{eq:twist_frozen}
\end{eqnarray}
Moreover, $\tw_{w}$ commutes with the bar involution $\overline{(\ )}$
on $\qClAlg(t_{0})$, and it restricts to a permutation on the initial
triangular basis $\can$.

\end{Thm}

\begin{proof}

Since $\sigma'$ is identified with $\overline{(\ )}$ by Lemma \ref{lem:identify_bar_involution},
Lemma \ref{lem:rescale_twsit_auto_commute_bar} implies that $\tw_{w}$
commutes with $\overline{(\ )}$.

Next, for any $S\in\circB_{-1}^{\up,w}$, we have $\eta_{w,q}S\in\circB_{-1}^{\up,w}$
by Theorem \ref{thm:twist_unipotent_cell}. We compute 
\begin{eqnarray*}
\tw_{w}\cor S & = & \eta_{w,q}\cor_{1}^{-2}\cor_{1}\cor_{2}S\\
 & = & \eta_{w,q}c_{1}(S)^{-1}c_{2}(S)S\\
 & = & c_{1}(S)^{-1}c_{2}(S)\eta_{w,q}S\\
 & = & c_{1}(\eta_{w,q}S)c_{2}(\eta_{w,q}S)\eta_{w,q}S\\
 & = & \cor(\eta_{w,q}S)
\end{eqnarray*}
 It follows that $\tw_{w}$ preserves $\cor\circB_{-1}^{\up,w}\simeq\can$.

Finally, applying the rescaling $\cor$ to the unipotent quantum minors
appearing in Theorem \ref{thm:twist_unipotent_cell} we get 
\begin{eqnarray*}
\tw_{w}(X_{k}(t_{0})) & = & \xi_{k}X_{k^{\max}}^{-1}*I_{k}(t_{0}),\qquad k\in I_{\ufv},\\
\tw_{w}(X_{j}(t_{0})) & = & \xi_{j}X_{j}(t_{0})^{-1},\qquad j\in I_{\fv},
\end{eqnarray*}
for some scalars $\xi_{k},\xi_{j}\in q^{\Hf\Z}$. Since $\tw_{w}$
commutes with the bar involution $\overline{(\ )}$ and quantum cluster
variables are bar-invariant, the scalars are chosen such that the
right hand side are bar-invariant. It follows that the right hand
side are normalization of the ordered product.

\end{proof}

Next, we want to show that the automorphism $\tw_{w}$ is a DT-type
cluster twist automorphism defined in Section \ref{sec:Twist-automorphisms}.

The automorphism $\tw_{w}$ naturally extends to an automorphism on
$\cF(t_{0})$. Recall that the mutation sequence induces a $\kk$-algebra
isomorphism $\seq^{*}:\cF(t_{0}[1])\simeq\cF(t_{0})$. Let us define
$\var_{w}:\cF(t_{0})\simeq\cF(t_{0}[1])$ by $\var_{w}=(\seq^{*})^{-1}\tw_{w}$.\reviseStart

From Theorem \ref{thm:twist_auto_invariance}, we deduce that\reviseEnd

\begin{align*}
\var_{w}(X_{k}(t_{0})) & =X_{k^{\max}}^{-1}\cdot X_{\sigma k}(t_{0}[1]),\qquad k\in I_{\ufv},\\
\var_{w}(X_{j}) & =X_{j}^{-1},\qquad j\in I_{\fv}.
\end{align*}
In particular, we see that $\var_{w}$ can be induced from the linear
homomorphism $\var_{w}:\Mc(t_{0})\simeq\Mc(t_{0}[1])$ such that

\begin{eqnarray*}
\var_{w}(f_{k}) & = & f_{\sigma(k)}'-f_{k^{\max}}',\qquad,k\in I_{\ufv}\\
\var_{w}(f_{j}) & = & -f_{j}',\qquad j\in I_{\fv},
\end{eqnarray*}
where $f_{i}$, $f_{i}'$ denote the $i$-th unit vector in $\Mc(t_{0})$
and $\Mc(t_{0}[1])$ respectively.

\begin{Prop}

The map $\var_{w}$ is a variation map.

\end{Prop}

\begin{proof}

It remains to check that $\var_{w}Y_{k}(t_{0})=Y_{\sigma k}(t_{0}[1])$
for any $k\in I_{\ufv}$, or, equivalently, $\var_{w}\deg^{t_{0}}Y_{k}(t_{0})=\deg^{t_{0}[1]}Y_{\sigma k}(t_{0}[1])$,
i.e., $\var_{w}(\sum_{i\in I}b_{ik}f_{i})=\sum_{i\in I}b_{ik}'f_{i}'$,
where $b_{ik}=b_{ik}(t_{0})$, $b_{ik}'=b_{ik}(t_{0}[1])$, $f_{i}$,
$f_{i}'$ denote the $i$-th unit vectors in $\Mc(t_{0})$ and $\Mc(t_{0}[1])$
respectively.

It might be possible to verify the claim by working with the matrices
directly (Remarks \ref{rem:initial_seed}, \ref{rem:injective_seed}).
Let us give a more conceptual proof.

On the one hand, we deduce from the definition of $\tw_{w}$ and Lemma
\ref{lem:parametrize_injective} that $\deg^{t_{0}}\tw_{w}X_{j}(t_{0})=-f_{j}$,
$j\in I$. It follows that $\deg^{t_{0}}\tw_{w}Y_{k}(t_{0})=-\deg^{t_{0}}Y_{k}(t_{0})$.

On the other hand, since $t_{0}$ and $t_{0}[1]$ are similar up to
$\sigma$, we have $b_{ik}=b_{\sigma i,\sigma k}'$ for $i,k\in I_{\ufv}$.
It follows from the definition of $\var_{w}$ that $\var_{w}(\sum b_{ik}f_{i})-\sum b_{ik}'f_{i}'=u_{k}$
for some $u_{k}\in\Z^{I_{\fv}}$. Therefore, $\var_{w}Y_{k}(t_{0})=Y_{\sigma k}(t_{0}[1])\cdot p_{k}$
where $p_{k}=X^{u_{k}}\in\cRing$. We deduce that $\tw_{w}Y_{k}(t_{0})=\seq^{*}\var_{w}Y_{k}(t_{0})=\seq^{*}Y_{\sigma k}(t_{0}[1])\cdot p_{k}$.
In addition, we know that $\deg^{t_{0}}\seq^{*}Y_{\sigma k}(t_{0}[1])=-\deg^{t_{0}}Y_{k}(t_{0})$
by properties of $c$-vectors, see \cite[Proposition 3.3.8]{qin2019bases}.
Therefore, $\deg^{t_{0}}\tw_{w}Y_{k}(t_{0})=-\deg^{t_{0}}Y_{k}(t_{0})+u_{k}$.

Combining the above discussions, we obtain $u_{k}=0$. It follows
that $\var_{w}\deg^{t_{0}}Y_{k}(t_{0})=\deg^{t_{0}[1]}Y_{\sigma k}(t_{0}[1])$.

\end{proof}

Consequently, we obtain the following result.

\begin{Thm}\label{thm:compare_twist_automorphism}

The automorphism $\tw_{w}$ decomposes as $\tw_{w}=\seq^{*}\var_{w}$
such that $\var_{w}:\cF(t_{0})\simeq\cF(t_{0}[1])$ is a variation
map. In particular, $\tw_{w}$ is a twist automorphism of Donaldson-Thomas
type in the sense of Section \ref{subsec:Twist-automorphism-DT-type}.

\end{Thm}

\begin{Eg}

Continue Example \ref{eg:sl3_type} \ref{eg:sl3_canonical_basis}.
Notice that $\seq^{*}(X_{1})=I_{1}(t_{0})=X_{1}'=X^{-f_{1}+f_{3}}+X^{-f_{1}+f_{2}}$.
The automorphism $\tw_{w}$ on $\qClAlg(t_{0})$ takes the following
form

\begin{align*}
\tw_{w}(X_{1}) & =[X_{3}^{-1}*I_{1}(t_{0})]^{t_{0}}=[X_{3}^{-1}*X_{1}']^{t_{0}}=X^{-f_{1}}+X^{-f_{1}+f_{2}-f_{3}}\\
\tw_{w}(X_{2}) & =X_{2}^{-1}\\
\tw_{w}(X_{2}) & =X_{3}^{-1}
\end{align*}

Define $\var_{w}:\cF(t_{0})\simeq\cF(t_{0}[1])$ such that 
\begin{eqnarray*}
\var_{w}(X_{1}) & = & q^{-\Hf\Lambda'_{13}}X_{1}'*X_{3}^{-1}=q^{\Hf\Lambda_{13}}X_{1}'*X_{3}^{-1}=q^{\Hf}X_{1}'*X_{3}^{-1}\\
\var_{w}(X_{2}) & = & X_{2}^{-1}\\
\var_{w}(X_{3}) & = & X_{3}^{-1}
\end{eqnarray*}
Then it is straightforward to check that $\var_{w}$ is a $\kk$-algebra
automorphism, $\tw_{w}=\seq^{*}\var_{w}$ and $\var_{w}(Y_{1})=X^{f_{2}-f_{3}}=Y_{1}'$.
Consequently, $\tw_{w}$ is an twist automorphism in our previous
sense. Notice that we have 
\begin{eqnarray*}
\tw_{w}(X_{1}') & = & \tw_{w}(X_{1}^{-1}*(q^{\Hf}X_{3}+q^{-\Hf}X_{2}))\\
 & = & ((1+qY_{1})*X_{1}^{-1})^{-1}*q^{-\Hf}(qX_{3}^{-1}+X_{2}^{-1})\\
 & = & X_{1}*(1+qY_{1})^{-1}*q^{-\Hf}X_{2}^{-1}*(qY_{1}+1)\\
 & = & X^{f_{1}-f_{2}}.
\end{eqnarray*}

\end{Eg}

\subsection{Consequences\label{subsec:Consequences}}

\subsubsection*{Quantization conjecture}

\begin{Thm}\label{thm:canonical_triangular}

The rescaled localized dual canonical basis $\can$ in \eqref{eq:initial_can_basis}
is the common triangular basis of $\qClAlg\simeq\qO[N_{-}^{w}]_{\Z[q^{\pm\Hf}]}$.
In particular, it contains all quantum cluster monomials.

\end{Thm} 

\begin{proof}

The rescaled localized dual canonical basis $\can$ is the initial
triangular basis for $t_{0}$ by Proposition \ref{prop:initial_triangular_basis}.
Thanks to its invariance under the twist automorphism (Theorems \ref{thm:twist_auto_invariance},
\ref{thm:compare_twist_automorphism}) and the admissibility along
the mutation sequence in Theorem \ref{thm:inj_seed_admissible_mutations},
we can apply the existence Theorem \ref{thm:existence_mutation_sequence}
to the initial triangular basis and the claim follows.

\end{proof}

The quantization conjecture follows as a consequence.

\begin{Thm}[{\cite[Conjecture 1.1]{Kimura10}}]\label{thm:conjecture_proved}

All quantum cluster monomials of $\clAlg(t_{0})$ are contained in
the dual canonical basis $B_{-1}^{\up}$ of $\envAlg^{-}(w)$ up to
$q^{\Hf}$-powers.

\end{Thm}

\subsubsection*{Comparing triangular bases}

Berenstein-Zelevinsky defined triangular basis for acyclic seeds in
\cite{BerensteinZelevinsky2012}, and their definition is different
from ours. Take $w=c^{2}$ where $c$ is a Coxeter word. The corresponding
quantum unipotent cell is an acyclic quantum cluster algebra. We have
the following natural result.

\begin{Cor}[{\cite[Conjecture 1.2.1]{qin2019compare}}]\label{cor:compare_triangular_basis}

The Berenstein-Zelevinsky triangular basis for acyclic seeds is the
same as the common triangular basis in \cite{qin2017triangular}.
Consequently, it contains all quantum cluster monomials.

\end{Cor}

\begin{proof}

\cite{qin2019compare} verifies the statements for symmetric Kac-Moody
cases. Now we have seen the existence of common triangular bases for
general Kac-Moody algebras. With this existence, the previous proof
\cite{qin2019compare} is also valid for symmetrizable Kac-Moody cases.

\end{proof}

\subsubsection*{Tropical properties of dual canonical bases}

We refer the readers to \cite{qin2019bases} for necessary definitions
and results relating bases and tropical points (Definition \ref{def:tropical_points}),
see also \cite{gross2018canonical} for a sophisticated geometric
view. Since common triangular bases are parametrized by the tropical
points, and we have seen that they agree with the localized dual canonical
bases (after rescaling), we get the following result.

\begin{Cor}\label{cor:tropical_dual_can_basis}

The dual canonical basis $B_{-1}^{\up}$ of the quantum unipotent
subgroup $\qO[N_{-}(w)]$ is parametrized by the set of tropical points
$\theta^{-1}(\N^{[1,l]})\subsetneq\Mc(t_{0})\simeq\tropMc$ of the
corresponding cluster variety in the sense of \cite{qin2019bases}. 

Similarly, the (localized) dual canonical basis $\circB_{-1}^{\up,w}$
of the quantum unipotent cell $\qO[N_{-}^{w}]$ is parametrized by
the set of the tropical points $\tropMc$ of the corresponding cluster
variety.

\end{Cor}

\appendix

\section{Cluster structures on quantum groups\label{sec:Cluster-structures-on}}

\subsection{Quantized enveloping algebras\label{sec:Serre_relations}}

The defining relations for a quantum enveloping algebra are the following.
For any $i,j\in[1,r]$, $h\in P^{\vee}$, we have

\begin{align*}
K^{h}K^{h'} & =K^{h+h'},\ K^{0}=1,\\
K^{h}E_{i}K^{-h} & =q^{\langle h,\alpha_{i}\rangle}E_{i},\\
K^{h}F_{i}K^{-h} & =q^{\langle h,-\alpha_{i}\rangle}F_{i},\\
E_{i}F_{j}-F_{j}E_{i} & =\delta_{ij}\frac{K_{i}-K_{i}^{-1}}{q_{i}-q_{i}^{-1}},\\
\sum_{k=0}^{1-C_{ij}}(-1)^{k}E_{i}^{(k)}E_{j}E_{i}^{(1-C_{ij}-k)} & =0,\quad i\neq j,\\
\sum_{k=0}^{1-C_{ij}}(-1)^{k}F_{i}^{(k)}F_{j}F_{i}^{(1-C_{ij}-k)} & =0,\quad i\neq j.
\end{align*}
In particular, $K_{i}E_{j}=q_{i}^{C_{ij}}E_{j}K_{i}$, $K_{i}F_{j}=q_{i}^{-C_{ij}}F_{j}K_{i}$.

\subsection{CGL extensions\label{subsec:CGL-extension}}

In the following, we explain that $\qO[N_{-}(w)]$ can be viewed as
iterated Ore extensions equipped with the action by a torus $\cH$.
See \cite{GY13} for details.

\begin{Def}[Twist derivation]\label{def:twist_derivation}

Let there be given a $\kk$-algebra $S$ and $\sigma\in\End S$. A
$\kk$-module endomorphism $\delta$ is called a $\sigma$-derivation
of $S$ if it satisfies $\delta(xy)=\delta(x)y+\sigma(x)\delta(y)$.

\end{Def}

When a $\kk$-algebra $S$ is equipped with a $\sigma$-derivation
$\delta$, we can define the Ore extension $S[x;\sigma,\delta]=\oplus_{d\in\N}x^{d}S$
as a free $S$-module, where $x$ is an indeterminate. We endow $S[x;\sigma,\delta]$
with the algebra structure by requiring $xs=\sigma(s)x+\delta(s)$.

Take the base field $\kk=\Q(q)$ as before. We can identify $Q$ with
the character lattice of the algebraic split torus $\cH=\Hom_{\Z}(Q,\Z)\otimes_{\Z}\kk^{*}\simeq(\kk^{*})^{r}$
such that $\gamma=\sum\gamma_{i}\alpha_{i}\in Q$ is identified with
$\chi^{\gamma}$, whose action on $\uh=(h_{1},\ldots,h_{r})\in\cH$
is given by $\chi^{\gamma}(\uh)=\prod h_{i}^{\gamma_{i}}$. Let $\cH$
act on $(\envAlg)_{\gamma}$ such that $\uh\cdot u=\chi^{\gamma}(\uh)u$.

\reviseStart

For any $i,j,k\in[1,r]$, we choose $\uh^{(k)},\uh^{(k)*}\in\cH$
such that $\chi^{\beta_{j}}(\uh^{(k)})=q^{(\beta_{j},\beta_{k})}$,
$\chi^{\beta_{i}}(\uh^{(k)*})=q^{-(\beta_{i},\beta_{k})}$ as the
following. Denote $\beta_{k}=\sum_{s=1}^{r}(\beta_{k})_{s}\alpha_{s}$.
We can choose $\uh^{(k)}=(h_{1},\ldots,h_{r})=q^{a}=(q^{a_{1}},\ldots,q^{a_{r}})$,
$a=\sum_{i=1}^{r}a_{i}\alpha_{i}^{\vee}$, $a_{i}=\sum_{s\in[1,r]}C_{is}(\beta_{k})_{s}\sym_{i}$,
and $\uh^{(k)*}=(\uh^{(k)})^{-1}=q^{-a}$.

\reviseEnd

Denote the action of $\uh^{(k)}$ on $\qO[N_{-}(w)]_{\leq k}$ by
$\sigma_{k}$ and the action of $\uh^{(k)*}$ on $\qO[N_{-}(w)]_{\geq k}$
by $\sigma_{k}^{*}$. By the dual LS-law (Proposition \ref{prop:dual_LS_law}),
we have the following $\kk$-module endomorphisms $\delta_{k}$ and
$\delta_{k}^{*}$:
\begin{eqnarray*}
\delta_{k}: & \qO[N_{-}(w)]_{[1,k-1]} & \rightarrow\qO[N_{-}(w)]_{[1,k-1]}\\
 & u & \mapsto F_{-1}^{\up}(\beta_{k})u-\sigma_{k}(u)F_{-1}^{\up}(\beta_{k})
\end{eqnarray*}

\begin{eqnarray*}
\delta_{k}^{*}: & \qO[N_{-}(w)]_{[k+1,l]} & \rightarrow\qO[N_{-}(w)]_{[k+1,l]}\\
 & u & \mapsto F_{-1}^{\up}(\beta_{k})u-\sigma_{k}^{*}(u)F_{-1}^{\up}(\beta_{k})
\end{eqnarray*}
Then $\qO[N_{-}(w)]$ can be written as the following iterated Ore
extensions equipped with the action by $\cH$:

\begin{align*}
\qO[N_{-}(w)] & =\kk[F_{-1}^{\up}(\beta_{1})][F_{-1}^{\up}(\beta_{2});\sigma_{2},\delta_{2}]\cdots[F_{-1}^{\up}(\beta_{l});\sigma_{l},\delta_{l}]\\
 & =\kk[F_{-1}^{\up}(\beta_{l})][F_{-1}^{\up}(\beta_{l-1});\sigma_{l-1}^{*},\delta_{l-1}^{*}]\cdots[F_{-1}^{\up}(\beta_{1});\sigma_{1}^{*},\delta_{1}^{*}].
\end{align*}
One can check that it is a symmetric Cauchon-Goodearl-Letzter (CGL)
extension in the sense of \cite[Defintion 2.2]{goodearl2016berenstein}. 

\begin{Rem}[Normalization factor]\label{rem:CGL_normalization_factor}

For $j<k$, denote $\uh^{(k)}F_{-1}^{\up}(\beta_{j})=\lambda_{kj}F_{-1}^{\up}(\beta_{j})$.
Then $\lambda_{kj}=q^{(\beta_{k},\beta_{j})}$. Define $\lambda_{kk}=1$,
$\lambda_{jk}=\lambda_{kj}^{-1}$. For any $j$, recall that $j^{\min}[m_{j}^{-}]=j$.
The normalization factor $\cS(\uc_{[j^{\min}j]})$ in \cite[(4.18)(2.7)]{GY13}
is given by

\begin{align*}
\cS(\uc_{[j^{\min},j]}) & :=\prod_{0\leq b<a\leq m_{j}^{-}}(\sqrt{\lambda_{j^{\min}[b],j^{\min}[a]}})^{-1}\\
 & =\prod_{b<a}q^{\Hf(\beta_{j^{\min}[b]},\beta_{j^{\min}[a]})}.
\end{align*}

Recall that $\cor B_{-1}^{\up}$ is $\sigma'$-invariant. Applying
$\cor$ and $\sigma'$ to $D[j^{\min},j]$ and using \eqref{eq:quantum_minor_ld_term},
we get

\begin{eqnarray*}
\cor D[j^{\min},j] & = & \sigma'\left(\xi_{j}\cor F_{-1}^{\up}(\beta_{j})\cdots\cor F_{-1}^{\up}(\beta_{j^{\min}[1]})\cor F_{-1}^{\up}(\beta_{j^{\min}})+\cor Z\right)\\
 & = & \xi_{j}^{-1}\cor F_{-1}^{\up}(\beta_{j^{\min}})\cor F_{-1}^{\up}(\beta_{j^{\min}[1]})\cdots\cor F_{-1}^{\up}(\beta_{j})+\sigma'\cor Z
\end{eqnarray*}
with the scalar $\xi_{j}=q^{-\Hf\sum_{m_{j}^{-}\geq a>b\geq0}(\beta_{j^{\min}[a]},\beta_{j^{\min}[b]})}.$
Then $\xi_{j}^{-1}$ agrees with the normalization factor $\cS(\uc_{[j^{\min}j]})$.
By \cite[Proposition 4.6, Theorem 8.2(b)]{GY13}, $\cor D[j^{\min},j]$
is a quantum cluster variable.

\end{Rem}

\subsection{Cluster structures}

\begin{Rem}\label{rem:injective_seed}

The $\tB$-matrix $\tB(t_{0}[1])$ is given by \cite[Section 10.1]{GY13}
as the following (\cite{GY13} treated finite dimensional Lie algebras,
but their statements concerning combinatorics such \cite[Proposition 10.4]{GY13}
are valid for a general root datum). 

Define the permutation $w_{0}$ on $[1,l]$ such that $w_{0}(j^{\min}[d])=j^{\max}[-d]$,
$j\in[1,l]$, $0\leq d<m(i_{j})$. Denote $j':=w_{0}(j)$. Then $I_{\fv}=\{j|p(j')=-\infty\}$.
The $(j,k)$-entry of $\tB(t_{0}[1])$ is given by

\begin{eqnarray*}
b_{jk} & = & \begin{cases}
1 & j'=p(k')\\
-1 & j'=s(k')\\
C_{i_{j}i_{k}} & p(j')<p(k')<j'<k'\\
-C_{i_{j}i_{k}} & p(k')<p(j')<k'<j'\\
0 & \mathrm{else}
\end{cases}.
\end{eqnarray*}

The $\Lambda$-matrix of $t_{0}[1]$ is determined by the $q$-commutativity
relation among the quantum cluster variables $\cor D[j,j^{\max}]$,
$j\in I$. By \cite[Theorem 9.5(b)]{GY13}, for any $j<k$, we have
\begin{align*}
N_{\ow}(\uc_{[j,j^{\max}]},\uc_{[k,k^{\max}]}) & =((\ow_{<j}+w)\varpi_{i_{j}},-\beta_{[k,k^{\max}]}).
\end{align*}

\end{Rem}

\begin{Eg}[{\cite[Example 13.2]{GeissLeclercSchroeer10}}]

Let us consider a more complicated example. Take the Cartan matrix
$C=\left(\begin{array}{ccc}
2 & -3 & -2\\
-3 & 2 & -2\\
-2 & -2 & 2
\end{array}\right)$, the reduced word $\ui=(1,2,1,3,1,2,1,2,3,2)$. Correspondingly,
$I=[1,10]$, $I_{\fv}=\{7,9,10\}$. We choose an initial seed $t_{0}=t_{0}(\ow)$
such that its $\tB$-matrix is given by Remark \ref{rem:initial_seed}.
We associate to its (skew-symmetric) matrix $(b_{ij})_{i,j\in I}$
a quiver $\tQ(t_{0})$, which is a directed graph with the set of
vertices $I$ and $[b_{ij}]_{+}$-many arrows from $i$ to $j$ for
all $i,j\in I$. We can choose $t_{0}$ such that the $\tQ(t_{0})$
is drawn as in Figure \ref{fig:non-adapt-initial} (frozen vertices
are diamond nodes and unfrozen ones are circle nodes; the number on
the arrow means the multiplicity).

The mutation sequence $\seq_{\ow}$ applies successively on the sequence
of vertices $(1,3,5,2,6,8,1,3,4,1,2,6,2)$ (read from left to right).
We have the seed $t_{0}[1]=\seq_{\ow}t_{0}$. Its quiver $\tQ(t_{0}[1])$
is drawn as Figure \ref{fig:non-adapt-shift}. Relabeling the vertices
by the permutation $w_{0}$ of $I$ given in Remark \ref{rem:injective_seed},
we obtain Figure \ref{fig:non-adapt-shift-relabeled}.

\end{Eg}

\begin{figure}[htb!]  \centering \beginpgfgraphicnamed{fig:non-adapt-initial}   \begin{tikzpicture}[scale=0.7]  \node [shape=circle, draw] (v1) at (10,5) {1};      \node  [shape=circle, draw] (v2) at (9,2) {2};      \node [shape=circle,  draw] (v3) at (8,5) {3}; \node [shape=circle, draw] (v4) at (7,-1) {4}; 
\node  [shape=circle, draw] (v5) at (6,5) {5};      \node [shape=circle,  draw] (v6) at (4,2) {6};      \node [shape=diamond, draw] (v7) at (3,5) {7};      \node  [shape=circle, draw] (v8) at (2,2) {8};      \node [shape=diamond,  draw] (v9) at (0,-1) {9};
\node [shape=diamond, draw] (v10) at (-1,2) {10};    \draw[-triangle 60] (v2) edge  node[near start] {3} (v1); \draw[-triangle 60] (v4) edge  node[near start] {2} (v2);        \draw[-triangle 60] (v4) edge  node[near start] {2} (v3);    \draw[-triangle 60] (v5) edge  node[near start] {3} (v2);    \draw[-triangle 60] (v6) edge  node[near end] {3} (v5);    \draw[-triangle 60] (v7) edge  node[near start] {3} (v6);  \draw[-triangle 60] (v7) edge  node[near end] {2} (v4);    \draw[-triangle 60] (v8) edge  node[near start] {2} (v4);   
\draw[-triangle 60] (v9) edge  node[near end] {2} (v8);    \draw[-triangle 60] (v9) edge  node[near end] {2} (v7);   
\draw[-triangle 60] (v10) edge  node[near start] {3} (v7);    \draw[-triangle 60] (v10) edge  node[near start] {2} (v9);   
\draw[-triangle 60] (v1) edge (v3); \draw[-triangle 60] (v3) edge (v5);    \draw[-triangle 60] (v5) edge (v7);    \draw[-triangle 60] (v2) edge (v6); \draw[-triangle 60] (v6) edge (v8);    \draw[-triangle 60] (v8) edge (v10);    \draw[-triangle 60] (v4) edge (v9);      \end{tikzpicture} \endpgfgraphicnamed \caption{The quiver $\tQ(t_0)$ associated to the reduced word $\ui$.} \label{fig:non-adapt-initial} \end{figure}

\begin{figure}[htb!]  \centering \beginpgfgraphicnamed{fig:non-adapt-shift}   \begin{tikzpicture}[scale=0.7]   \node [shape=diamond, draw] (v1) at (10,5) {7};      \node  [shape=diamond, draw] (v2) at (9,2) {10};      \node [shape=circle,  draw] (v3) at (8,5) {5}; \node [shape=diamond, draw] (v4) at (7,-1) {9}; 
\node  [shape=circle, draw] (v5) at (6,5) {3};      \node [shape=circle,  draw] (v6) at (4,2) {8};      \node [shape=circle, draw] (v7) at (3,5) {1};      \node  [shape=circle, draw] (v8) at (2,2) {6};      \node [shape=circle,  draw] (v9) at (0,-1) {4};
\node [shape=circle, draw] (v10) at (-1,2) {2};    \draw[-triangle 60] (v2) edge  node[near start] {3} (v1);      \draw[-triangle 60] (v4) edge  node[near start] {2} (v1);    \draw[-triangle 60] (v3) edge  node[near start] {3} (v2);    \draw[-triangle 60] (v6) edge  node[near end] {3} (v3);    \draw[-triangle 60] (v7) edge  node[near start] {3} (v6);  \draw[-triangle 60] (v5) edge  node[near end] {2} (v4);    \draw[-triangle 60] (v6) edge  node[near start] {2} (v4);   
\draw[-triangle 60] (v9) edge  node[near end] {2} (v6);    \draw[-triangle 60] (v9) edge  node[near end] {2} (v5);   
\draw[-triangle 60] (v8) edge  node[near start] {3} (v7);    \draw[-triangle 60] (v10) edge  node[near start] {2} (v9);   
\draw[-triangle 60] (v1) edge (v3); \draw[-triangle 60] (v3) edge (v5);    \draw[-triangle 60] (v5) edge (v7);    \draw[-triangle 60] (v2) edge (v6); \draw[-triangle 60] (v6) edge (v8);    \draw[-triangle 60] (v8) edge (v10);    \draw[-triangle 60] (v4) edge (v9);      \end{tikzpicture} \endpgfgraphicnamed \caption{The quiver $\tQ(t_0[1])$ associated with the reduced word $\ui$.} \label{fig:non-adapt-shift} \end{figure}

\begin{figure}[htb!]  \centering \beginpgfgraphicnamed{fig:non-adapt-shift-relabeled}   \begin{tikzpicture}[scale=0.7]   \node [shape=diamond, draw] (v1) at (10,5) {1'};      \node  [shape=diamond, draw] (v2) at (9,2) {2'};      \node [shape=circle,  draw] (v3) at (8,5) {3'}; \node [shape=diamond, draw] (v4) at (7,-1) {4'}; 
\node  [shape=circle, draw] (v5) at (6,5) {5'};      \node [shape=circle,  draw] (v6) at (4,2) {6'};      \node [shape=circle, draw] (v7) at (3,5) {7'};      \node  [shape=circle, draw] (v8) at (2,2) {8'};      \node [shape=circle,  draw] (v9) at (0,-1) {9'};
\node [shape=circle, draw] (v10) at (-1,2) {10'};    \draw[-triangle 60] (v2) edge  node[near start] {3} (v1);      \draw[-triangle 60] (v4) edge  node[near start] {2} (v1);    \draw[-triangle 60] (v3) edge  node[near start] {3} (v2);    \draw[-triangle 60] (v6) edge  node[near end] {3} (v3);    \draw[-triangle 60] (v7) edge  node[near start] {3} (v6);  \draw[-triangle 60] (v5) edge  node[near end] {2} (v4);    \draw[-triangle 60] (v6) edge  node[near start] {2} (v4);   
\draw[-triangle 60] (v9) edge  node[near end] {2} (v6);    \draw[-triangle 60] (v9) edge  node[near end] {2} (v5);   
\draw[-triangle 60] (v8) edge  node[near start] {3} (v7);    \draw[-triangle 60] (v10) edge  node[near start] {2} (v9);   
\draw[-triangle 60] (v1) edge (v3); \draw[-triangle 60] (v3) edge (v5);    \draw[-triangle 60] (v5) edge (v7);    \draw[-triangle 60] (v2) edge (v6); \draw[-triangle 60] (v6) edge (v8);    \draw[-triangle 60] (v8) edge (v10);    \draw[-triangle 60] (v4) edge (v9);     \end{tikzpicture} \endpgfgraphicnamed \caption{The quiver $\tQ(t_0[1])$ associated to the reduced word $\ui$ with vertices relabeled.} \label{fig:non-adapt-shift-relabeled} \end{figure}

\begin{Rem}\label{rem:injective_mutation_sequence}

By the combinatorial algorithm in \cite[Section 13.1]{GeissLeclercSchroeer10},
we can construct a mutation sequence for Theorem \ref{thm:inj_seed_admissible_mutations}
as follows. For $k\in[1,l],$denote the following mutation sequence
(read from right to left)

\begin{align*}
\seq_{k} & :=\mu_{k^{\min}[m(i_{k},[k,l])-2]}\cdots\mu_{k^{\min}[1]}\mu_{k^{\min}}.
\end{align*}
Notice that the sequence is trivial when $m(i_{k}[k,l])-2<0$. Define
the mutation sequence 
\begin{eqnarray*}
\seq_{\ow} & := & \seq_{l}\cdots\seq_{2}\seq_{1}.
\end{eqnarray*}
Then $t_{0}[1]=\seq_{\ow}t_{0}$. 

Moreover, $t_{0}$ and $t_{0}[1]$ are similar up to the permutation
$\sigma$ on on $[1,l]$ such that\reviseStart

\begin{eqnarray*}
\sigma(k)= & k, & \qquad k=k^{\max}\\
\sigma(k^{\min}[d])= & k^{\max}[-d-1], & \qquad0\leq d<m(i_{k}).
\end{eqnarray*}
See Remarks \ref{rem:initial_seed}, \ref{rem:injective_seed}.\reviseEnd

\end{Rem}

\bibliographystyle{amsalphaURL}
\def\cprime{$'$}
\providecommand{\bysame}{\leavevmode\hbox to3em{\hrulefill}\thinspace}
\providecommand{\MR}{\relax\ifhmode\unskip\space\fi MR }
\providecommand{\MRhref}[2]{%
  \href{http://www.ams.org/mathscinet-getitem?mr=#1}{#2}
}
\providecommand{\href}[2]{#2}

\end{document}